\documentclass[a4paper,10pt]{article}

\newcommand{\afficherfigures}{true}

\usepackage[latin1]{inputenc}
\usepackage[T1]{fontenc}
\usepackage[english]{babel}
\usepackage{amsmath}
\usepackage{amsfonts}
\usepackage{amssymb}
\usepackage{amsthm}
\usepackage{dsfont}

\usepackage{graphicx}
\usepackage{hyperref}
\usepackage{xcolor}
\hypersetup{linkcolor=blue}
\hypersetup{colorlinks=true}
\usepackage[width=14cm, height=22cm]{geometry}

\newtheorem{thm}{Theorem}
\newtheorem{prop}{Proposition}
\newtheorem{definition}{Definition}
\newtheorem{lem}{Lemma}
\newtheorem{cor}{Corollary}
\newtheorem{example}{Example}
\newtheorem{rk}{Remark}

\numberwithin{equation}{section}

\usepackage{ifthen}
\newcommand{\afficherfigure}[1]{\ifthenelse{\equal{\afficherfigures}{true}}{#1}{}}

\title{
    Quadratic obstructions
    to small-time local controllability
    for scalar-input differential systems
}
\author{
 Karine Beauchard\footnote{
  Karine Beauchard, IRMAR, Ecole Normale Sup\'erieure de Rennes, UBL, CNRS, Campus de Ker Lann, 35170 Bruz, France, \texttt{karine.beauchard@ens-rennes.fr}},
 Fr\'ed\'eric Marbach\footnote{Fr\'ed\'eric Marbach, 
  Sorbonne Universit\'es, UPMC Universit\'e Paris 06, CNRS UMR 7598, Laboratoire Jacques-Louis Lions, Place Jussieu, 75005 Paris, France, \texttt{frederic.marbach@upmc.fr}}
}

\newcommand{\ds}{\mathrm{d}s}
\newcommand{\dt}{\mathrm{d}t}
\newcommand{\dtau}{\mathrm{d}\tau}

\newcommand{\Slie}[1]{\mathcal{S}_{#1}}
\newcommand{\Sone}{\Slie{1}}
\newcommand{\Stwo}{\Slie{2}}

\newcommand{\ad}{\mathrm{ad}}
\newcommand{\Span}{\mathrm{Span~}}
\newcommand{\Supp}{\mathrm{supp}}
\newcommand{\Lie}{\mathrm{Lie}}
\newcommand{\M}{\mathcal{M}}
\newcommand{\R}{\mathbb{R}}
\newcommand{\N}{\mathbb{N}}
\newcommand{\tr}[1]{#1^{\mathrm{tr}}}
\newcommand{\eqmod}[1]{=_{[#1]}}
\newcommand{\norm}[2]{\left\| #1 \right\|_{#2}}
\newcommand{\parl}{{\scriptscriptstyle \parallel}}
\newcommand{\stlc}[1]{#1 small-time locally controllable}
\newcommand{\sstlc}{smoothly small-time locally controllable}
\newcommand{\alp}[1]{\left(\beta, #1\right)}
\newcommand{\aux}{\xi}
\newcommand{\tm}{\mathcal{T}}
\newcommand{\Og}[1]{\mathcal{O}_\gamma{\left(#1\right)}}
\newcommand{\Oo}[1]{\mathcal{O}_1{\left(#1\right)}}
\newcommand{\Oz}[1]{\mathcal{O}_0{\left(#1\right)}}
\newcommand{\Ou}[1]{\mathcal{O}{\left(#1\right)}}

\begin{document}
    
\maketitle

\begin{abstract}
 We consider nonlinear scalar-input differential control systems in the 
 vicinity of an equilibrium. When the linearized system at the equilibrium 
 is controllable, the nonlinear system is smoothly small-time locally 
 controllable, i.e., whatever $m>0$ and $T>0$, the state can reach a whole 
 neighborhood of the equilibrium at time $T$ with controls arbitrary small in 
 $C^m$-norm. When the linearized system is not controllable, we prove that 
 small-time local controllability cannot be recovered from the quadratic 
 expansion and that the following quadratic alternative holds.
 
 Either the state is constrained to live within a smooth strict 
 invariant manifold, up to a cubic residual, or the quadratic order adds a 
 signed drift in the evolution with respect to this manifold. In the second 
 case, the quadratic drift holds along an explicit  Lie bracket of length $(2k+1)$, 
 it is quantified in terms of an $H^{-k}$-norm of the control, 
 it holds for controls small in $W^{2k,\infty}$-norm.  
 These spaces are optimal for general nonlinear systems and are 
 slightly improved in the particular case of control-affine systems. 
 
 Unlike other works based on Lie-series formalism, our proof is based on 
 an explicit computation of the quadratic terms by means of appropriate
 transformations. In particular, it does not require that the vector fields 
 defining the dynamic are smooth. We prove that $C^3$ regularity is sufficient
 for our alternative to hold.
 
 This work underlines the importance of the norm used 
 in the smallness assumption on the control: depending on this choice of 
 functional setting, the same system may or may not be small-time locally 
 controllable, even though the state lives within a finite dimensional space.
 
 \bigskip
 
 
\end{abstract}

\newpage
\setcounter{tocdepth}{2}
\tableofcontents

\section{Introduction}

\subsection{Scalar-input differential systems}

Let $n \in \N^*$. Throughout this work, we consider differential control 
systems where the state~$x(t)$ lives in $\R^n$ and the control
is a scalar input $u(t) \in \R$. For $f \in C^1(\R^n\times\R,\R^n)$, we 
consider the nonlinear control system:
\begin{equation} \label{system.fxu}
 \dot{x}=f(x,u).
\end{equation}

\begin{definition} \label{Definition:sol.fxu}
    Let $T > 0$ and $x^* \in \R^n$ be a given initial data. We say that a couple 
    $(x,u) \in C^0([0,T],\R^n) \times L^\infty((0,T),\R)$ is
    a \emph{trajectory} of~\eqref{system.fxu} associated with $x^*$ when:
    \begin{equation} \label{eq.ivp}
    \forall t \in [0,T], \quad 
    x(t) = x^* + \int_0^t f(x(s), u(s)) \ds.
    \end{equation}
\end{definition}

\begin{prop} \label{Prop:fxu}
    Let $T > 0$, $x^* \in \R^n$ and $u \in L^\infty((0,T),\R)$.
    System~\eqref{system.fxu} admits a unique maximal trajectory 
    defined on $[0,T_u)$ for some $T_u \in (0,T]$. 
\end{prop}

\begin{proof}
    Once a control $u \in L^\infty((0,T),\R)$ is fixed, system~\eqref{system.fxu} 
    can be seen as $\dot{x}(t) = g(t,x(t))$, where we introduce $g(t,x) := f(x,u(t))$.
    The function $g$ is not continuous with respect to time. Hence, we cannot apply
    usual Cauchy-Lipschitz-Picard-Lindel\"of existence and uniqueness theorem. However,
    the existence and uniqueness of a solution in the sense of 
    Definition~\ref{Definition:sol.fxu} holds for such functions $g$ 
    (see e.g.~\cite[Theorem 54, page 476]{MR1640001}). The proof relies
    on a fixed-point theorem applied to the integral formulation~\eqref{eq.ivp}.
\end{proof}

In the particular case of control-affine systems, we will work 
with a slightly different functional framework. Let $f_0, f_1 \in
C^1(\R^n,\R^n)$. A control-affine system takes the form:
\begin{equation} \label{system.f0f1}
 \dot{x} = f_0(x) + u f_1(x).
\end{equation}
Such systems are both important from the point of view of applications
and mathematically as a first-order Taylor expansion with respect to a small
control of a nonlinear dynamic.

\begin{definition} \label{Definition:sol.f0f1}
    Let $T > 0$ and $x^* \in \R^n$ be a given initial data. We say that a couple 
    $(x,u) \in C^0([0,T],\R^n) \times L^1((0,T),\R)$ is
    a \emph{trajectory} of~\eqref{system.f0f1} associated with $x^*$ when:
    \begin{equation} \label{eq.ivp2}
    \forall t \in [0,T], \quad 
    x(t) = x^* + \int_0^t \left( f_0(x(s)) + u(s) f_1(x(s)) \right) \ds.
    \end{equation}
\end{definition}

\begin{prop}
    Let $T > 0$, $x^* \in \R^n$ and $u \in L^1((0,T),\R)$.
    System~\eqref{system.f0f1} admits a unique maximal trajectory 
    defined on $[0,T_u)$ for some $T_u \in (0,T]$. 
\end{prop}

\begin{proof}
 Here again, one applies a fixed-point theorem to the integral 
 formulation~\eqref{eq.ivp2}.
\end{proof}

In the sequel and where not explicitly stated, it is implicit that we handle
well-defined trajectories of our differential systems, either by restricting
to small enough times, small enough controls or sufficiently nice dynamics
preventing blow-up. 

Moreover, we will often need to switch point of view 
between nonlinear and control-affine systems. 
Given a control-affine system 
characterized by $f_0$ and $f_1$, one can always see it as a particular case of 
a nonlinear system by defining:
\begin{equation} \label{def.f}
 f(x, u) := f_0(x) + u f_1(x).
\end{equation}
Conversely, given a nonlinear system characterized by $f$, its dynamic can be
approximated using
$f(x, u) \approx f_0(x) + u f_1(x) + O(u^2)$ where we define:
\begin{equation} \label{def.f0f1}
 f_0(x) := f(x, 0) \quad \text{and} \quad f_1(x) := \partial_u f(x, 0).
\end{equation}

\subsection{Small-time local controllability}

Multiple definitions of small-time local controllability can be found in the
mathematical literature. Here, we put the focus on the smallness assumption
made on the control. This notion is mostly relevant in the vicinity
of an equilibrium. We use the following definitions:

\begin{definition}
 We say that $(x_e,u_e) \in \R^n \times \R$ is an 
 \emph{equilibrium} of system~\eqref{system.fxu} 
 when $f(x_e, u_e) = 0$ and an 
 equilibrium of system~\eqref{system.f0f1}
 when $f_0(x_e) + u_e f_1(x_e) = 0$. 
 Up to a translation, one can always assume that $x_e = 0$
 and $u_e = 0$. Thus, in the sequel, it is implicit that we consider
 systems for which the couple $(0,0)$ is an equilibrium.
\end{definition}

\begin{definition} \label{Definition:STLC}
    Let $\left(E_T, \norm{\cdot}{E_T}\right)$ be a family 
    of normed vector spaces of scalar functions defined on $[0,T]$, for 
    $T > 0$. We say that a scalar-input differential system is 
    \emph{\stlc{$E$}}
    when the following property holds: for any $T > 0$, for any
    $\eta > 0$, there exists $\delta > 0$ such that, for any
    $x^*, x^\dagger \in \R^n$ with $|x^*| + |x^\dagger| \leqslant \delta$,
    there exists a trajectory $(x,u)$ of the differential system
    defined on $[0,T]$ with $u \in E_T$ satisfying:
    \begin{equation}
     \norm{u}{E_T} \leqslant \eta
     \quad \text{and} \quad
     x(0) = x^*
     \quad \text{and} \quad 
     x(T) = x^\dagger.
    \end{equation}
    Here and in the sequel, it is implicit that this notion of 
    local controllability refers to local controllability in the
    vicinity of the null equilibrium $x_e = 0$ and $u_e = 0$
    of our system. The translation of our results to other 
    equilibriums is left to the reader.
\end{definition}

It could be thought that, in a finite dimensional setting with
smooth dynamics, the notion of small-time local 
controllability should not depend on the smallness assumption made
on the control. However, it is not the case. This fact plays a key role in this work. 
We will see that the relevance of the quadratic approximation depends on the chosen norm.
We will also use the following notion:
\begin{definition} \label{Definition:SSTLC}
    We say that a scalar-input differential system is \emph{\sstlc}
    when it is \stlc{$C^m$} for any $m \in \N$.
\end{definition}

\subsection{Linear theory and the Kalman rank condition}
\label{Subsec:linear}

The natural approach to investigate the local controllability of
system~\eqref{system.fxu} near an equilibrium 
is to study the controllability of the linearized system,
which is given by:
\begin{equation} \label{system.yh}
 \dot{y} = H_0 y + u b,
\end{equation}
where $H_0 := \partial_x f(0,0)$ and $b := \partial_u f(0,0)$.
It is well known (see works~\cite{MR0120435} of 
Pontryagin,~\cite[Theorem~6]{MR0145169} of LaSalle
or~\cite[Theorem~10]{MR0155070} of Kalman, Ho and Narendra) 
that such linear control systems are controllable, 
independently on the allowed time $T$, if and only if they satisfy the
Kalman rank condition:
\begin{equation} \label{kalman}
 \Span \left\{ H_0^k b, \enskip k \in \{ 0, \ldots n-1 \}
 \right\} = \R^n.
\end{equation}
It is also classical to prove that the controllability of the linearized system
implies small-time local controllability for the nonlinear system 
(see~\cite[Theorem~3]{MR0186487} by Markus or~\cite[Theorem~1]{MR889459} by Lee 
and Markus). For example, we have:

\begin{thm} \label{Theorem:Linear}
    Assume that the Kalman rank condition~\eqref{kalman} holds. 
    Then the nonlinear system~\eqref{system.fxu} is \sstlc.
    Moreover, one can choose controls compactly supported within 
    the interior of the allotted time interval.
\end{thm}

\begin{proof}
 \textbf{Smooth controllability of the linearized system with compactly supported controls.}
 Let $T > 0$. We start by proving that we can use regular compactly supported controls to 
 achieve controllability for the linear system~\eqref{system.yh}. 
 We introduce the controllability Gramian:
 \begin{equation} \label{def.gramian}
  \mathfrak{C}_T := \int_0^T e^{(T-t) H_0} b \tr{b} e^{(T-t) \tr{H_0}} \dt.
 \end{equation}
 It is well-known that $\mathfrak{C}_T$ is invertible if and only if the
 Kalman rank condition holds (see \cite[Section~1.2]{MR2302744}).
 Let $(\rho_\epsilon)_{\epsilon>0}$ be a family of functions in 
 $C^\infty_c((0,T),\mathbb{R})$ with
 $\Supp (\rho_\epsilon) \subset [\epsilon,T-\epsilon]$,
 such that 
 $\rho_\epsilon(t) \underset{\epsilon \rightarrow 0}{\rightarrow} 1$ for every $t \in (0,T)$ and 
 $|\rho_\epsilon(t)| \leqslant 1$ for every $(\epsilon,t) \in \mathbb{R}^*_+ \times(0,T)$.
 Let:
 \begin{equation} \label{def.gramian_epsilon}
  \mathfrak{C}_{T,\epsilon} := \int_0^T \rho_{\epsilon}(t) e^{(T-t) H_0} b \tr{b} e^{(T-t) \tr{H_0}}  \dt.
 \end{equation}
 By the dominated convergence theorem, $\mathfrak{C}_{T,\epsilon}$ converges to $\mathfrak{C}_T$ in $\mathcal{M}_n(\R)$.
 Thus, for $\epsilon>0$ small enough, $\mathfrak{C}_{T,\epsilon}$ belongs to $GL_n(\R)$, because this is an open subset of $\mathcal{M}_n(\R)$.
 From now on, such an $\epsilon$ is fixed.
 Let $y^*, y^\dagger \in \R^n$.
 Using an optimal control or "Hilbert Uniqueness Method" approach
 (see \cite[Section~1.4]{MR2302744}), we define for $t \in [0,T]$:
 \begin{equation} \label{def.u.hum}
  u(t) := \rho_\epsilon(t) \tr{ b } e^{(T-t) \tr{H_0}} p,
 \end{equation}
  where $p \in \R^n$ is defined by
 \begin{equation} \label{eq.def.b.gram}
  p := \mathfrak{C}_{T,\epsilon}^{-1} \left( y^\dagger - e^{TH_0} y^* \right).
 \end{equation}  
 Using a Duhamel formula for~\eqref{system.yh},
 combined with~\eqref{def.gramian_epsilon},~\eqref{def.u.hum} and~\eqref{eq.def.b.gram}, 
 one checks that the solution of~\eqref{system.yh} with initial 
 condition $y(0)=y^*$ satisfies:
 \begin{equation}
  y(T) = e^{T H_0} y^* + \int_0^T u(t) e^{(T-t)H_0} b \dt 
  = e^{T H_0} y^* + \mathfrak{C}_{T,\epsilon} p
  = y^\dagger.
 \end{equation}
 From~\eqref{def.u.hum} and~\eqref{eq.def.b.gram}, one checks that:
 \begin{equation} \label{eq.thm1.uhm}
  \forall m \geqslant 0, \exists \eta_{T,m} > 0, \quad
  \norm{u}{C^m(0,T)} \leqslant \eta_{T,m} \left( |y^*| + |y^\dagger| \right).
 \end{equation}
 
 \bigskip \noindent
 \textbf{Smooth controllability for the nonlinear system, with compactly supported controls.}
 We move on to the nonlinear system using the approach followed 
 in~\cite[Theorem 3.6]{MR2302744}. Let $m \in \N$ and define:
 \begin{equation}
  C^m_\epsilon((0,T),\R)
  :=
  \left\{ f \in C^m((0,T),\R) , \enskip \Supp(f) \subset [\epsilon,T-\epsilon] \right\},
 \end{equation}
 endowed with the norm $\|.\|_{C^m(0,T)}$, which is a Banach space.
 We introduce the nonlinear mapping:
 \begin{equation}
  \mathcal{F} :
  \left\{
  \begin{aligned}
  \R^n \times C^m_\epsilon((0,T),\R) & \to \R^n \times \R^n, \\
  (x^*, u) & \mapsto (x^*, x(T)),
  \end{aligned}
  \right.
 \end{equation}
 where $x$ is the solution to~\eqref{system.fxu} with initial data $x^*$
 and control $u$. It is well known (see~\cite[Theorem~1, page~57]{MR1640001}),
 that $\mathcal{F}$ defines a $C^1$ map. Moreover, its differential
 $\mathcal{F}'(0,0)$ is the following linear map:
 \begin{equation}
  \mathcal{F}'(0,0): \left\{
  \begin{aligned}
  \R^n \times C^m_\epsilon((0,T),\R) & \to \R^n \times \R^n, \\
  (y^*, u) & \mapsto (y^*, y(T)),
  \end{aligned}
  \right.
 \end{equation}
 where $y$ is the solution to~\eqref{system.yh} with initial data $y^*$.
 From~\eqref{eq.thm1.uhm}, we know that this differential has a bounded right-inverse. 
 The inverse function theorem yields the existence
 of a $C^1$ right-inverse $\mathcal{G}$ to $\mathcal{F}$, defined
 in a small neighborhood of $(0,0)$. Hence, for any $\eta > 0$, there
 exists $\delta > 0$ such that:
 \begin{equation} \label{eq.thm1.xxug}
  |x^*| + |x^\dagger| \leqslant \delta
  \quad \Rightarrow \quad 
  \norm{u}{C^m(0,T)} \leqslant 
  \left| \mathcal{G}(x^*, x^\dagger) \right| \leqslant \eta.
 \end{equation}
 From~\eqref{eq.thm1.xxug}, the nonlinear system~\eqref{system.fxu} is \stlc{$C^m$}, with controls supported in $[\epsilon,T-\epsilon]$. 
 This holds for any $m \in \N$, 
 thus system~\eqref{system.fxu} is \sstlc, with compactly supported controls.
\end{proof}

When the linear test fails, it is necessary to continue the expansion further on
to determine whether small-time local controllability holds or not. Indeed, some
systems are (smoothly) small-time locally controllable despite failing 
the linear test~\eqref{kalman}.

\begin{example} \label{ex.a}
 Let $n = 2$. Consider the following scalar-input control-affine system:
 \begin{equation} \label{sys.ex.a}
 \left\{
 \begin{aligned}
 \dot{x}_1 & = u, \\
 \dot{x}_2 & = x_1^3.
 \end{aligned}
 \right.
 \end{equation}
 The linearized system of~\eqref{sys.ex.a} around the 
 null equilibrium is not controllable because the second direction is 
 left invariant. However, let us explain why system~\eqref{sys.ex.a} is 
 smoothly small-time locally controllable. We start by introducing 
 a smooth even function $\varphi \in C^\infty(\R,\R)$, such that:
 \begin{align}
  \label{hyp.exa.1}
  \varphi(x) & = 0 \quad \text{for } |x| \geqslant 1/4,  \\
  \label{hyp.exa.2}
  \varphi(x) & = 1 \quad \text{for } |x| \leqslant 1/8, \\
  \label{hyp.exa.3}
  \int_{\R_-} \varphi^3 =
  \int_{\R_+} \varphi^3 & = 1.
 \end{align}
 In particular, from~\eqref{hyp.exa.2}, $\varphi(0) = 1$.
 Let $x^*, x^\dagger \in \R^2$ be given initial and final data. 
 Let $T > 0$. For $t \in [0,T]$, we define:
 \begin{equation} \label{def.exa.u}
  u(t) := \frac{1}{T} x^*_1 \varphi'\left(\frac{t}{T}\right)
  + \frac{1}{T} \lambda \varphi'\left(\frac{2t-T}{2T}\right)
  + \frac{1}{T} x^\dagger_1 \varphi'\left(\frac{t-T}{T}\right),
 \end{equation}
 where $\lambda \in \R$ is a constant to be chosen later on.
 From~\eqref{sys.ex.a},~\eqref{hyp.exa.1},~\eqref{hyp.exa.2} and the initial 
 condition $x_1(0) = x^*_1$, we deduce that, for $t\in[0,T]$:
 \begin{equation} \label{eq.exa.x1}
  x_1(t) = x^*_1 \varphi\left(\frac{t}{T}\right)
  + \lambda \varphi\left(\frac{2t-T}{2T}\right)
  + x^\dagger_1 \varphi\left(\frac{t-T}{T}\right).
 \end{equation}
 From~\eqref{hyp.exa.1},~\eqref{hyp.exa.2} and~\eqref{eq.exa.x1}, we deduce that $x_1(T) = x^\dagger_1$.
 From~\eqref{sys.ex.a},~\eqref{hyp.exa.1},~\eqref{hyp.exa.3} and~\eqref{eq.exa.x1}, since the three contributions in $u$ have
 disjoint supports, we have:
 \begin{equation} \label{eq.exa.x2}
  x_2(T) = x^*_2 + T (x^*_1)^3 + T (x^\dagger_1)^3 + 2 T \lambda^3.
 \end{equation}
 Hence, from~\eqref{eq.exa.x2}, the constructed trajectory 
 satisfies $x_2(T) = x^\dagger_2$ if and only if:
 \begin{equation} \label{eq.exa.lambda}
  2 \lambda^3 = \frac{x^\dagger_2-x^*_2}{T} - (x^*_1)^3 - (x^\dagger_1)^3.
 \end{equation}
 Let $m \in \N$. Thanks to~\eqref{hyp.exa.1} and~\eqref{def.exa.u}, we have:
 \begin{equation} \label{eq.exa.uhm}
  \norm{u}{C^m(0,T)}^2 = T^{-2m-1} \norm{\varphi}{C^m(\R)}^2
  \left( \frac{1}{2} |x^*_1|^2 + \frac{1}{2} |x^\dagger_1|^2 + |\lambda|^2 \right).
 \end{equation}
 Let $\eta > 0$. From~\eqref{eq.exa.lambda} and~\eqref{eq.exa.uhm}, 
 there exists $\delta = \delta_{m,T,\eta} > 0$ such that:
 \begin{equation}
  | x^* | + | x^\dagger | \leqslant \delta 
  \quad \Rightarrow \quad 
  \norm{u}{C^m(0,T)} \leqslant \eta.
 \end{equation}
 Hence, we have constructed a control small in $C^m$-norm driving the state from $x^*$ to $x^\dagger$.
 This holds for every $T>0$ and $m \in \N$, 
 thus system~\eqref{sys.ex.a} is \sstlc{}.
 Moreover, thanks to~\eqref{hyp.exa.2}, our construction yields controls which
 are compactly supported within $(0,T)$. Hence, there is no ``control-jerk'' 
 near the initial or the final time.
\end{example}

\subsection{Iterated Lie brackets}

The main tool to study the controllability of nonlinear
systems beyond the linear test is the notion of iterated Lie brackets. Many 
works have investigated the link between Lie brackets and controllability with
the hope of finding necessary or sufficient conditions. We refer 
to~\cite[Sections 3.2 and 3.4]{MR2302744} by Coron and~\cite{MR1061394} by 
Kawski for surveys on this topic. Let us recall elementary definitions from 
geometric control theory that will be useful in the sequel.
\begin{definition}
 Let $X$ and $Y$ be smooth vector fields on $\R^n$. 
 The Lie bracket $[X,Y]$ of $X$ 
 and $Y$ is the smooth vector field defined by:
 \begin{equation}
  [X,Y](x) := Y'(x) X(x) - X'(x) Y(x).
 \end{equation}
 Moreover, we define by induction on $k \in \N$ the notations:
 \begin{align}
  \ad_X^0 (Y) & := Y, \\
  \ad_X^{k+1} (Y) & := \left[X, \ad_X^k(Y)\right].
 \end{align}
\end{definition}

In addition to these special brackets with a particular nesting
structure, we define the following classical linear subspaces of $\R^n$
for smooth control-affine systems.

\begin{definition} \label{def:Sk}
 Let $f_0$ and $f_1$ be smooth vector fields on $\R^n$.
 For $k \geqslant 1$, we define $\Slie{k}$ as
 the non decreasing sequence of linear subspaces of $\R^n$ spanned by the 
 iterated Lie brackets of $f_0$ and $f_1$ (with any possible nesting structure), 
 containing $f_1$ at most $k$ times, evaluated at the null equilibrium. 
 For nonlinear systems, we extend these definitions thanks to~\eqref{def.f0f1}.
\end{definition}

The spaces $\Sone$ and $\Stwo$ play a key role in this paper;
the former describes the set of controllable directions for the linearized 
system while the latter describes the directions involved at the quadratic 
order.  When the Kalman rank condition is not
fulfilled, the quadratic obstructions to small-time local controllability
will come from the components of the state living in the orthogonal of the
controllable space. 

\begin{definition}
 Let $\langle \cdot, \cdot \rangle$ denote the usual euclidian scalar product on 
 $\R^n$. We introduce $\mathbb{P} : \R^n \rightarrow \Sone$ the 
 orthogonal projection on $\Sone$ with respect to $\langle \cdot, \cdot \rangle$. 
 Similarly, we define $\mathbb{P}^\perp := \mathrm{Id} - \mathbb{P} : 
 \R^n \rightarrow \Sone^\perp$ the orthogonal projection on $\Sone^\perp$.
\end{definition}

\subsection{The first known quadratic obstruction}

At the quadratic order, the situation is more involved than at the linear
order and very little is known. Proposing a classification of the possible 
quadratic behaviors for scalar-input systems is the main motivation of this work. 
Historically, the following conjecture due to Hermes was proved by
Sussmann in~\cite{MR710995} for control-affine systems~\eqref{system.f0f1}.

\begin{prop} \label{Prop:Hermes}
    Let $f_0, f_1$ be analytic vector fields over $\R^n$ with $f_0(0) = 0$.
    Assume that:
    \begin{gather}
    \label{hermes.dim.lie}
    \left\{ g(0); \enskip g \in \Lie (f_0, f_1) \right\}=\R^n, \\
    \label{hermes.inclusions}
    \Slie{2k+2} \subset \Slie{2k+1} 
    \quad \text{for any~} k \in \N.
    \end{gather}
    Then system~\eqref{system.f0f1} is \stlc{$L^\infty$}.
\end{prop}

Reciprocally, for analytic vector fields, hypothesis~\eqref{hermes.dim.lie} is a necessary condition for small-time local 
controllability (for a proof, see~\cite[Proposition~6.2]{MR710995}). Sussmann
was mostly interested in investigating whether~\eqref{hermes.inclusions}
was also a necessary condition, in particular for $k = 0$, the condition:
\begin{equation} \label{s2s1}
\Stwo \subset \Sone.
\end{equation}
The first violation of~\eqref{s2s1} occurs when 
$[f_1,[f_0,f_1]](0) \notin \Sone$.
The following important known result is due to Sussmann 
(see~\cite[Proposition~6.3, page~707]{MR710995}).
\begin{prop} \label{Thm:Sussmann}
Let $f_0, f_1$ be analytic vector fields over $\R^n$ with $f_0(0) = 0$.
Assume that:
\begin{equation} \label{Obstruction_1}
 [f_1,[f_0,f_1]](0) \notin \Sone.
\end{equation}
Then system~\eqref{system.f0f1} is not \stlc{$L^\infty$}.
\end{prop}

Although Sussmann does not insist on the smallness assumption made on the 
control, it can be seen that his assumption is linked to the 
$W^{-1,\infty}$-norm of the control. 
Indeed, he works with arbitrary small-times $T$ and controls $u$ such that $|u|_{L^\infty(0,T)} \leqslant A$, 
with a fixed constant $A>0$. This guarantees that $|u_1|_{L^\infty(0,T)} \leqslant AT$ is
arbitrary small, where $u_1(t) := \int_0^t u(s)\ds$. We prove in 
Theorem~\ref{Theorem:QAA} that the smallness in $W^{-1,\infty}(0,T)$
is in fact the correct assumption for this first quadratic obstruction.
We also prove that the analyticity assumption is not necessary: 
it suffices that $f_0 \in C^3$ and $f_1 \in C^2$ 
(see Corollary~\ref{Corollary:AffineC3}).

\begin{example} \label{Example:easy_drift}
  Let $n = 2$. We consider the following control-affine system:
  \begin{equation}
   \left\{
    \begin{aligned}
     \dot{x}_1 & = u, \\
     \dot{x}_2 & = x_1^2.
    \end{aligned}
   \right.
  \end{equation}
  Around the null equilibrium, we have $\Sone = \R e_1$ and:
  \begin{equation} \label{ex.b.1}
   [f_1, [f_0, f_1]](0) = - 2 e_2 \notin \Sone.
  \end{equation}
  Equation~\eqref{ex.b.1} causes a drift in the direction $e_2$, 
  quantified by the $H^{-1}$-norm of the control. Indeed, if the initial 
  state is $x(0) = 0$, we have:
  \begin{equation}
   x_1(t) = u_1(t) 
   \quad \textrm{and} \quad
   x_2(t) = \int_0^t u_1^2(s) \ds,
  \end{equation}
  where $u_1(t) := \int_0^t u(s) \ds$. 
  Thus $x_2(t) \geqslant 0$ and the system is not locally controllable.
\end{example}

\subsection{The first Lie bracket paradox}

Sussmann also attempted to study further violations of condition~\eqref{s2s1}.
In particular, when $[f_1,[f_0,f_1]](0) \in \Sone$, the next
violation is $[f_1,[f_0,[f_0,[f_0,f_1]]]](0) \notin \Sone$. The 
intermediate violation involving two times $f_0$ never happens. Indeed, from the 
Jacobi identity, we have:
\begin{equation} \label{eq.obs.f0f0}
 [f_1, [f_0, [f_0, f_1]]](0)
 = - [[f_0,f_1], [f_1, f_0]](0) - [f_0, [[f_0,f_1], f_1]](0).
\end{equation}
The first term in the right-hand side of~\eqref{eq.obs.f0f0} vanishes because
of the antisymmetry of the Lie bracket operator. Moreover, when 
$[f_1, [f_0,f_1]](0) \in \Sone$, the second term in the right-hand side
of~\eqref{eq.obs.f0f0} belongs to $\Sone$ because this subspace is stable 
with respect to bracketing by $f_0$. Thus, the second simplest violation
of~\eqref{s2s1} occurs when:
\begin{equation} \label{eq.f03}
 \left[f_1, \ad^3_{f_0}(f_1)\right](0) \notin \Sone.
\end{equation}
However, Sussmann exhibits the following example which indicates that the
violation~\eqref{eq.f03} does not prevent a system from being 
\stlc{$L^\infty$}.

\begin{example} \label{ex.sussmann}
 Let $n = 3$ and consider the following control-affine system:
 \begin{equation} \label{system.sussmann}
  \left\{
  \begin{aligned}
   \dot{x}_1 & = u, \\
   \dot{x}_2 & = x_1, \\
   \dot{x}_3 & = x_1^3 + x_2^2.
  \end{aligned}
  \right.
 \end{equation}
 Around the null equilibrium, $\Sone = \R e_1 + \R e_2$. One checks that:
 \begin{align}
  [f_1, \ad^1_{f_0}(f_1)](0) & = 0, \\
  \label{ex.s.bad}
  [f_1, \ad^3_{f_0}(f_1)](0) & = 2 e_3.
 \end{align}
 Hence, this system exhibits the violation~\eqref{eq.f03}.
 However, it is \stlc{$L^\infty$} 
 (see~\cite[pages~711-712]{MR710995}). 
\end{example}

Historically, Example~\ref{ex.sussmann} stopped the investigation of 
whether condition~\eqref{s2s1} was a necessary condition for small-time
local controllability. One of the motivations of our work is to 
understand in what sense~\eqref{s2s1} can be seen as a necessary
condition. We give further comments on Example~\ref{ex.sussmann} 
in Subsection~\ref{Subsection:QACF_examples}.

\subsection{A short survey of related results}

The search for a necessary and sufficient condition for the small-time local
controllability of differential control systems has a long history and many
related references. We only provide here a short overview of some results 
connected to our work.

\subsubsection{Lie algebra rank condition}

A  well known necessary condition for the $L^\infty$ small-time local 
controllability of analytic systems is the Lie algebra rank 
condition~\eqref{hermes.dim.lie}
due to Hermann and Nagano~\cite{MR0149402,MR0199865,MR0321133}.
By considering the example $\dot{x}_1=ue^{-1/u^2}$ one sees that the 
analyticity  assumption on $f$ cannot be removed.

By the Frobenius theorem (see \cite[Theorem 4]{MR1425878} for a proof), if
there exists $k \in \{1, \ldots n \}$ such that, for every $x$ in a 
neighborhood of $0$, 
$\Span \{ h(x) ; \enskip h \in \Lie(f_0,f_1) \}$ is of 
dimension~$k$, then the reachable set
is (locally) contained in a submanifold of $\R^n$ of dimension~$k$.

\bigskip

In the particular case of $C^\infty$ driftless control-affine systems 
$\dot{x}=\sum_{i=1}^m u_i f_i(x)$, the Lie algebra rank condition is also a 
sufficient condition for the $L^\infty$ small-time local controllability; this 
is the Rashevski-Chow theorem proved in~\cite{MR0001880}. However, the case of
control-affine systems with drift is still widely open, even in the scalar-input
case as in system~\eqref{system.f0f1}.

\subsubsection{Lie brackets necessary and sufficient conditions}

Some necessary conditions and some sufficient conditions for $L^\infty$ 
small-time local controllability were proved by means of Lie-series formalism 
(Chen-Fliess series), for analytic control-affine systems. First, necessary
conditions:
\begin{itemize}
 \item the Sussmann condition~(\ref{Obstruction_1}), proved in~\cite[Proposition~6.3, page~707]{MR710995},
 \item the Stefani condition~\cite{MR935375}: $\text{ad}_{f_1}^{2m}(f_0)(0) \in \mathcal{S}_{2m-1}$ for every $m \in \mathbb{N}^*$.
\end{itemize}
Then, sufficient conditions:
\begin{itemize}
\item the Hermes condition~\cite{MR0638354}, recalled in Proposition \ref{Prop:Hermes} and proved by Sussmann in~\cite{MR710995},
\item the Sussman $S(\theta)$ condition, introduced in~\cite{MR872457} (see also \cite[Theorem 3.29]{MR2302744}):
there exists $\theta \in [0,1]$, such that
every bracket involving $f_0$ an odd number $l$ of times and $f_1$ 
an even number $k$ of times
must be a linear combination of brackets 
involving $k_i$ times $f_1$ and $l_i$ times $f_0$ and such that
$\theta l_i + k_i < \theta l + k$. In some sense, 
this condition says that bad brackets may be neutralized by good ones.
The weight $\theta$ has to be the same for all the bad brackets.
\item the Kawski condition (see \cite[Theorem 3.7]{MR1061394}):
there exists $\theta \in [0,1]$ such that
every bracket involving $f_0$ an odd number $l$ of times and $f_1$ an even number $k$ of times
is a linear combination of brackets of the form $\ad_{f_0}^{\nu_i}(h_i)$ where $\nu_i \geqslant 0$ and
$h_i$ is a bracket involving $k_i$ times $f_1$ and $l_i$ times $f_0$ with $\theta l_i + k_i < \theta l + k$.
\end{itemize}
In the particular case of control-affine systems that are homogeneous with 
respect to a family of dilatations (corresponding to time scalings in the 
control, not amplitude scalings), a necessary and sufficient condition for 
$L^\infty$ small-time  local controllability was proved by Aguilar and Lewis 
in~\cite[Theorem 4.1]{MR2968064}. 

\subsubsection{Control of bilinear systems without a priori bound on the control}

A scalar-input bilinear system $\dot{g} = d L_g (h_0 + u h_1)$
on a semi simple compact Lie group $U$ is (globally) controllable in large
time if and only if the Lie algebra generated by $h_0$ and $h_1$ is equal to the
compact semi-simple Lie algebra on $U$ (see~\cite{MR2004373,MR1776551,MR0338882}). 
In~\cite{MR2224821}, Agrachev and Chambrion use geometric control theory to estimate the
minimal time needed for the global controllability of such systems.
Contrary to the present article, these works do not impose any \emph{a priori}
bound on the control.

In~\cite{MR3250374,beauchard:hal-01333537}, Beauchard, Coron and Teisman propose classes of Schrödinger PDEs
(infinite dimensional bilinear control systems) for which approximate
controllability in $L^2$ is impossible in small time, even with large
controls.

\subsubsection{Quadratic approximations}

In~\cite{MR966201} and~\cite{MR0493664}, Agrachev and Hermes proceed to local 
investigations of mappings of type input-state $F_t:u \mapsto x(t)$
(with fixed initial condition $x(0)=x^*$) near a fixed critical point 
$\bar{u}$ (i.e.\ for which the linearized system is not controllable)
to determine whether $F_t(\bar{u})$ belongs to the interior or the 
boundary of the image of $F_t$, in the particular case when the linearized
system misses only one direction.

In~\cite{MR0493664}, Hermes proposes a sufficient condition for $0$ to be an interior point in 
any time and a necessary condition for $0$ to be a boundary point in small-time.

In~\cite{MR966201}, Agrachev proves that, when 
the quadratic form $F_t''(\bar{u})$ on $\mathrm{Ker~}F_t'(\bar{u})$
is definite on a subspace of $\mathrm{Ker~}F_t'(\bar{u})$ with finite codimension,
then it is enough to find its inertia index (which is either a nonnegative integer, or $+\infty$) to answer the question. He describes flexible explicit formulas for the inertia index of $F_t''$ and uses them for a general study of the quadratic mapping $F_t''$.

\bigskip

Finally, Brockett proposes in~\cite{MR3110058} sufficient conditions for 
controllability in large time and in small time for systems with quadratic 
drifts:
\begin{equation}
 \dot{y} = Ay + Bu \quad \text{and} \quad \dot{z} = \tr{y} Q y,
\end{equation}
where $y \in \R^d$ and $z \in \R$, with $Q \in \mathcal{M}_d(\R)$ 
(satisfying some specific structural assumptions).
In the case of a single scalar control, he proves that such systems are never
small-time locally controllable  (see~\cite[Lemma 4.1, page 444]{MR3110058}).
More precisely, he establishes that, if $Q\neq0$ is a
symmetric matrix contained within a specific subspace of dimension $d$, then
there exists a $0 \leqslant k <d$ such that
$\tr{(A^k B)} Q (A^k B) \neq 0$ and such that $z$ shares the same sign 
for trajectories starting from $y(0)=0$ and for small enough times.
The sign argument prevents small-time local controllability.

Our results can be seen as stemming from this sign argument. Indeed, we 
extend it to any matrix $Q$ and more generally to any
second-order expansion. Then, we improve the argument by proving that the
positive quantity is in fact coercive with respect to some specific norm of 
the control. Last, we use this coercivity to overwhelm higher-order
terms coming from Taylor expansions of general nonlinear systems.

\section{Main results and examples}
\label{Section:results}

This section is organized as follows. We state our main results in 
Subsection~\ref{Subsec:Main}. Then, we give comments in 
Subsection~\ref{Subsection:Comments}. We propose examples 
illustrating these statements and capturing the essential phenomena in Subsection~\ref{Subsection:QACF_examples},
and examples proving the optimality of our functional framework 
in Subsection \ref{Sec:Optimal}.

\bigskip

As stated in the abstract, the functional setting plays a key role in our
results. For $T > 0$ and $m \in \N$, we consider the usual Sobolev spaces 
$W^{m,\infty}(0,T)$ equipped with their natural norm. We also use the 
subspaces $W^{m,\infty}_0(0,T)$ corresponding to functions 
$\varphi \in W^{m,\infty}(0,T)$ which satisfy
$\varphi(0) = \varphi'(0) = \ldots = \varphi^{(m-1)}(0) = 0$ (no boundary 
condition is required at the right boundary for our proofs to hold). 
For $j \geqslant 0$, we define by induction the iterated primitives of $u$, 
denoted $u_j: (0,T) \to \R$ and defined by:
\begin{equation} \label{def:uj}
 u_0:=u \quad \text{and} \quad u_{j+1}(t):=\int_0^t u_j(s) \ds.
\end{equation}
For $p \in [1, +\infty]$ we let:
\begin{equation}
\norm{u}{W^{-1,p}(0,T)} := \norm{u_1}{L^p(0,T)}.
\end{equation}
By convention, we set $W^{-1,\infty}_0(0,T) := W^{-1,\infty}(0,T)$ to avoid 
singling out this particular case. Eventually, the adjective \emph{smooth} 
is used as a synonym of $C^\infty$ throughout the text.

\subsection{Statement of the main theorems} 
\label{Subsec:Main}

For a nonlinear system~\eqref{system.fxu}, our main result is the following quadratic alternative.

\begin{thm} \label{Theorem:QANL}
 Let $f \in C^{\infty}(\R^n \times \R,\R^n)$ with $f(0,0) = 0$.
 Let $\Sone$, $\Stwo$ be as in Definition~\ref{def:Sk}. We 
 define $d:=\mathrm{dim~}\Sone$ and the vector:
 \begin{equation} \label{DEF_d0}    
 d_0 := \partial^2_u f(0,0) \in \R^n.
 \end{equation}
 There exists a map $G \in C^\infty(\Sone,\Sone^\perp)$ 
 with $G(0) = 0$ and $G'(0) = 0$, such that the following alternative holds:
 \begin{itemize}
  
  \item \textbf{When} $\Stwo + \R d_0 = \Sone$, 
  up to a cubicly small error in the control,
  the state lives within the smooth manifold $\M \subset \R^n$ of dimension $d$ 
  given by the graph of $G$:
  \begin{equation} \label{def.MG1}
  \M := \left\{ p_\parl + G(p_\parl) ; \enskip p_\parl \in \Sone \right\}.
  \end{equation} 
  More precisely, for every $T>0$, there exists $C, \eta > 0$ such that, 
  for any trajectory $(x,u) \in C^0([0,T],\R^n) \times L^\infty((0,T),\R)$
  of system~\eqref{system.fxu} with $x(0) = 0$ and satisfying
  $\norm{u}{L^\infty} \leqslant \eta$, one has:
  \begin{equation} \label{Estimate:TQANL.1}
  \forall t \in [0,T], \quad
  \left|
  \mathbb{P}^\perp x(t)- 
  G \left( \mathbb{P} x(t) \right) \right| 
  \leqslant C \norm{u}{L^3}^3.
  \end{equation}
  
  \item \textbf{When} $d_0 \notin \Sone$, for sufficiently small-times
  and regular controls, the state drifts with respect to the invariant
  manifold $\M$ in the direction $\mathbb{P}^\perp d_0$. More precisely:
  \begin{itemize}
   \item System~\eqref{system.fxu} is not \stlc{$L^\infty$}.
   \item There exists $T^*>0$ such that, 
   for any $\tm \in (0,T^*)$, there exists $\eta > 0$ such that,
   for any $T \in (0,\tm]$ and any trajectory 
   $(x,u) \in C^0([0,T],\R^n) \times L^\infty((0,T),\R)$
   of system~\eqref{system.fxu} with $x(0) = 0$ and
   satisfying $\norm{u}{L^\infty} \leqslant \eta$, one has:
   \begin{equation} \label{Estimate:TQANL.2}
   \forall t \in [0,T], \quad
   \left\langle \mathbb{P}^\perp x(t) - 
   G \left( \mathbb{P} x(t) \right) , d_0 \right\rangle \geqslant 0.
   \end{equation}
  \end{itemize}
  
  \item \textbf{When} $d_0 \in \Sone$ \textbf{and} $\Stwo \not \subset \Sone$, 
  for sufficiently small-times
  and regular controls, the state drifts with respect to the invariant
  manifold $\M$ in a fixed direction. More precisely:
  \begin{itemize}
   
   \item There exists $k \in \{1, \ldots, d \}$ such 
   that
   \begin{align} 
   \label{hyp_j}
   [ \ad_{f_0}^{j-1}(f_1) , \ad_{f_0}^{j}(f_1) ] (0) & \in \Sone 
   \quad \text{ for } 1 \leqslant j < k, \\
   \label{hyp_k}
   [ \ad_{f_0}^{k-1}(f_1) , \ad_{f_0}^{k}(f_1) ] (0) & \notin \Sone.
   \end{align}

   \item System~\eqref{system.fxu} is not \stlc{$W^{2k,\infty}$}.
   
   \item There exists $T^*>0$ such that, 
   for any $\tm \in (0,T^*)$, there exists $\eta > 0$ such that,
   for any $T \in (0,\tm]$ and any trajectory 
   $(x,u) \in C^0([0,T],\R^n) \times L^\infty((0,T),\R)$
   of system~\eqref{system.fxu} with $x(0) = 0$ 
   and $u \in W^{2k,\infty}_0$ 
   satisfying $\norm{u}{W^{2k,\infty}} \leqslant \eta$, one has:
   \begin{equation} \label{Estimate:TQANL.3}
   \forall t \in [0,T], \quad
   \left\langle \mathbb{P}^\perp x(t)- 
   G \left( \mathbb{P} x(t) \right) , d_k \right\rangle \geqslant 0,
   \end{equation}
   where the drifting direction $d_k \neq 0$ is defined as:
   \begin{equation} \label{def.dk}
   d_k := - \mathbb{P}^\perp 
   [ \ad_{f_0}^{k-1}(f_1) , \ad_{f_0}^{k}(f_1) ] (0).
   \end{equation}
   
  \end{itemize}
 \end{itemize}
\end{thm}

\begin{cor} \label{Corollary:necessary_nonlinear}
 Assume that $f \in C^{\infty}(\R^n\times\R,\R^n)$ 
 with $f(0,0) = 0$ and system~\eqref{system.fxu} is \sstlc.
 Then $\Stwo + \R d_0 = \Sone$.
\end{cor}

\begin{proof}
 Assume that $\Stwo + \R d_0 \not\subset \Sone$.
 From Theorem~\ref{Theorem:QANL}, either
 $d_0 \notin \Sone$ and system~\eqref{system.fxu} is not \stlc{$C^0$}, 
 or $d_0 \in \Sone$ and $\Stwo \not \subset \Sone$ and there 
 exists $1 \leqslant k \leqslant d < n$ such that it
 is not \stlc{$C^{2k}$}. Both cases contradict 
 Definition~\ref{Definition:SSTLC}.
\end{proof}

In the particular case of control-affine systems~\eqref{system.f0f1}, the 
optimal functional framework for the conclusions of Theorem \ref{Theorem:QANL} to 
hold can be improved. In the first case, it is sufficient that the control be
small in $W^{-1,\infty}$-norm (instead of $L^\infty$-norm), whereas in the third 
case, it is sufficient that the control be small in $W^{2k-3,\infty}$-norm (instead of $W^{2k,\infty}$-norm).

\begin{thm} \label{Theorem:QAA}
 Let $f_0, f_1 \in C^{\infty}(\R^n,\R^n)$ with $f_0(0) = 0$
 and $d:=\mathrm{dim~}\Sone$. There exists a map $G \in C^\infty(\Sone,\Sone^\perp)$ 
 with $G(0) = 0$ and $G'(0) = 0$, such that the following alternative holds:
 \begin{itemize}
  
  \item \textbf{When} $\Stwo = \Sone$, up to a cubicly small error in the control,
  the state lives within a smooth manifold $\M \subset \R^n$ of dimension $d$ 
  given by the graph of $G$ (see~\eqref{def.MG1}).
  More precisely, for every $\tm > 0$, there exists $C, \eta > 0$ such that, 
  for any $T \in (0,\tm]$,
  for any trajectory $(x,u) \in C^0([0,T],\R^n) \times L^1((0,T),\R)$
  of system~\eqref{system.f0f1} with $x(0) = 0$ which satisfies
  $\norm{u}{W^{-1,\infty}} \leqslant \eta$, one has:
  \begin{equation} \label{CCL:TQA:1}
  \forall t \in [0,T], \quad
  \left|
  \mathbb{P}^\perp x(t) - 
  G \left( \mathbb{P} x(t) \right) \right| 
  \leqslant C \norm{u}{W^{-1,3}}^3.
  \end{equation}
  
  \item \textbf{When} $\Stwo \not \subset \Sone$, for sufficiently small-times
  and small regular controls, the state drifts with respect to the invariant
  manifold $\M$ in a fixed direction. More precisely:
  \begin{itemize}
   
   \item There exists $1\leqslant k \leqslant d$ such that~\eqref{hyp_j} and~\eqref{hyp_k} hold.
   
   \item System~\eqref{system.f0f1} is not \stlc{$W^{2k-3,\infty}$}.
   
   \item There exists $T^* > 0$ such that, for any $\tm \in (0,T^*)$, there exists
   $\eta > 0$ such that, for any $T \in (0,\tm]$ and any trajectory 
   $(x,u) \in C^0([0,T],\R^n) \times L^1((0,T),\R)$
   of system~\eqref{system.f0f1} with $x(0) = 0$ such that 
   $u \in W^{2k-3,\infty}_0$ with $\norm{u}{W^{2k-3,\infty}} \leqslant \eta$, 
   inequality~\eqref{Estimate:TQANL.3} holds.
   
  \end{itemize}
 \end{itemize}
\end{thm}

\subsection{Comments on the main theorems} 
\label{Subsection:Comments}

\begin{itemize}
 
 \item The linear subspace $\Sone$ is the tangent space to the manifold 
 $\mathcal{M}$ at $0$ because $G'(0)=0$.

 \item The optimality of our result, in terms of the norms used in the smallness 
 assumption on the control, is illustrated in Subsection~\ref{Sec:Optimal}. Even in the 
 case of bilinear systems (for which $f_0$ and $f_1$ are linear in $x$),
 the norms involved in Theorem~\ref{Theorem:QAA} are optimal.

 \item We give a characterization of the time $T^*$ in 
 paragraph~\ref{Subsubsection:coercivity}: it is the maximal time for which some coercivity property 
 holds for an appropriate second-order approximation of the system under study.
 Therefore, it does not depend on higher-order terms.

 \item When $\Stwo + \R d_0 \not\subset \Sone$, the drift 
 relations~\eqref{Estimate:TQANL.2} and~\eqref{Estimate:TQANL.3} can be 
 used to exhibit impossible motions. In particular, they imply that there 
 exists $T, \eta > 0$ such that, for any $x^\dagger \in \Sone^\perp$ satisfying
 $\langle x^\dagger, d_k \rangle < 0$, no motion from $x(0) = 0$ to
 $x(T) = x^\dagger$ is possible with controls of $W^{2k,\infty}$-norm
 smaller then $\eta$ (or $W^{2k-3,\infty}$-norm smaller then $\eta$
 for control-affine systems). Other impossible motions can also be exhibited, 
 by time reversibility. 
 For instance, no motion from an initial state $x^* \in \Sone^\perp$ such 
 that $\langle x^*, d_k \rangle > 0$ to $x(T) = 0$ is possible for small
 times and small controls. More generally, motions going against
 the drift direction are impossible for small-times and small controls.

 \item For a particular class of systems ("well-prepared systems", i.e.\ that 
 satisfy \eqref{hyp.integrator}), inequalities~\eqref{Estimate:TQANL.2} 
 and~\eqref{Estimate:TQANL.3} hold in the following stronger form:
 \begin{equation}
  \forall t \in [0,T], \quad
  \left\langle \mathbb{P}^\perp x(t)- 
  G \left( \mathbb{P} x(t) \right) , d_k \right\rangle 
  \geqslant C(\tm) \int_0^t u_k(s)^2 \ds,
 \end{equation}
 where $u_k$ is defined in~\eqref{def:uj}.
 Such an estimate also holds for general (i.e.\ not well-prepared) systems but
 with a value of $T^*$ that may be strictly smaller than the one explicitly given 
 in paragraph~\ref{Subsubsection:coercivity} (see paragraph~\ref{Subsubsection:IllCoercivity}).
 
 \item The drift relations~\eqref{Estimate:TQANL.2} 
 and~\eqref{Estimate:TQANL.3} hold for controls in $W^{m,\infty}_0$
 (with $m = 2k$ for nonlinear systems and $m=2k-3$ for control-affine systems).
 For controls which only belong to $W^{m,\infty}$, 
 they are still true at the final time, but the bound $\eta$ depends on $T$ 
 (and $\eta(T) \to 0$ as $T \to 0$). The persistence of the drift at the final 
 time under this weaker assumption will be clear from the proof 
 (see paragraph~\ref{Subsubsection:Gagliardo}) and allows to deny $W^{m,\infty}$
 small-time local controllability, and not only $W^{m,\infty}_0$ 
 small-time local controllability
 (see also paragraph~\ref{Subsub:Traces} for more insight on this topic).

 \item To lighten the statement of the theorems, we assumed that the vector
 fields are smooth. This guarantees that the spaces $\Sone$ and $\Stwo$
 are well-defined. However, it is in fact possible to give a meaning to the 
 considered objects when $f \in C^2(\R^n\times\R, \R^n)$
 (see Subsection~\ref{Subsec:S2_nonsmooth}). 
 Moreover, the conclusions of Theorem~\ref{Theorem:QANL} and~\ref{Theorem:QAA} 
 persist for nonsmooth dynamics (see Corollary~\ref{Corollary:NonlinearC3} 
 and~\ref{Corollary:AffineC3} in Subsection~\ref{Subsection:rough}).
 In particular, these corollaries imply that, even in the case of smooth
 vector fields, the constants in our main theorems only depend on low-order
 derivatives of the dynamic.

\end{itemize}

\subsection{Illustrating toy examples} 
\label{Subsection:QACF_examples}

\subsubsection{Evolution within an invariant manifold}

In the absence of drift (i.e.\ when $\Stwo = \Sone$), our theorems imply that 
the state must live within an invariant manifold, up to the cubic order. 
A simple case for which the manifold is not trivial (i.e.\ $\M \neq \Sone$) 
is the following toy model, for which the state stays exactly within the manifold,
without any remainder.

\afficherfigure{
 \begin{figure}[ht!]
  \begin{center}
   \includegraphics[width=6.6cm]{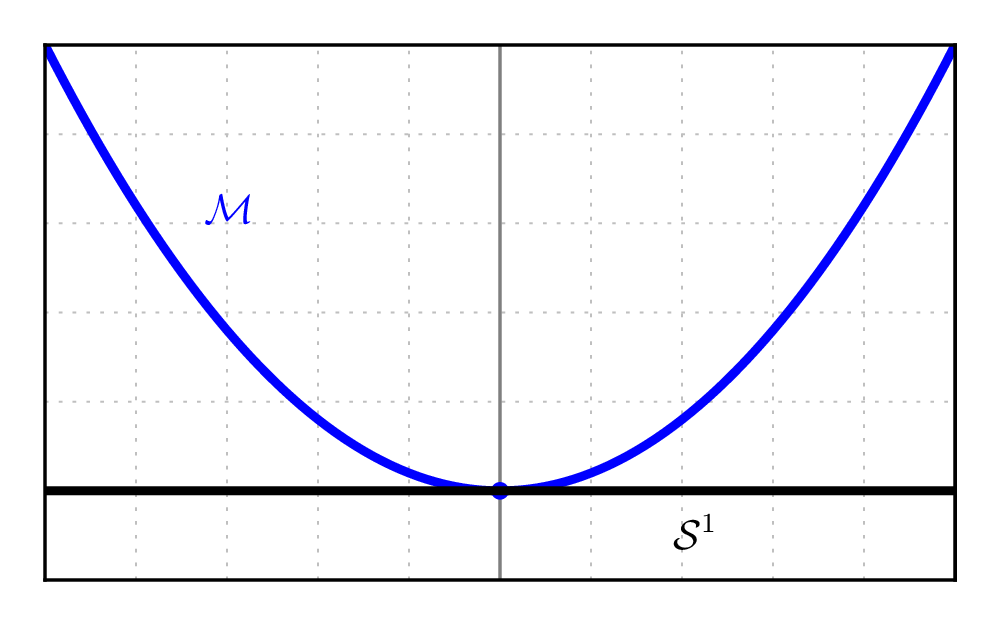}
   \caption{\label{Figure:toy_manifold}
    The toy invariant manifold~\eqref{def.toy.m}
    for system~\eqref{sys.ex.iho}.
   }
  \end{center}
 \end{figure}
}

\begin{example} \label{Example:toy_manifold}
	Let $n = 2$. We consider the following control-affine system:
	\begin{equation} \label{sys.ex.iho}
	\left\{
	\begin{aligned}
	\dot{x}_1 & = u, \\
	\dot{x}_2 & = 2 u x_1.
	\end{aligned}
	\right.
	\end{equation}
	Here, $\Sone = \Stwo = \R e_1$. Using the notation $u_1$ introduced 
	in~\eqref{def:uj}, system~\eqref{sys.ex.iho} can be integrated as 
	$x_1(t) = u_1(t)$ and $x_2(t) = u_1^2(t)$. For any control 
	and any time $t\in [0,T]$, $x(t) \in \M$ 
	(see Figure~\ref{Figure:toy_manifold}), where:
	\begin{equation} \label{def.toy.m}
	\mathcal{M} = \{ (x_1, x_2) \in \R^2, \enskip x_2 = x_1^2 \}.
	\end{equation}
	One checks that $\Sone$ is indeed the tangent space to $\M$ at the origin.
\end{example}

\subsubsection{Quadratic drifts and coercivity}

The simplest kind of quadratic drift for control-affine systems was exposed in
Example~\ref{Example:easy_drift}. Let us study two new examples.

\begin{example} \label{Example:Absorbed}
 Let $n = 1$. We consider the following control-affine system:
 \begin{equation} \label{system.ex.x1ux12}
  \dot{x}_1 = u + x_1^2.
 \end{equation}
 Here, $\Sone = \R e_1$ and Theorem~\ref{Theorem:Linear} asserts that system~\eqref{system.ex.x1ux12} is \sstlc.
 The potential drift direction $[f_1, [f_0, f_1]] = - 2 e_1 \in \Sone$
 and is actually absorbed by the linear controllability in the vicinity
 of the null equilibrium.
\end{example}

\begin{example} \label{Example:Competition}
 Let $n = 3$. We consider the following control-affine system,
 which is also proposed by Brockett in~\cite[page 445]{MR3110058}:
 \begin{equation} \label{system.quad.competition}
  \left\{
   \begin{aligned}
    \dot{x}_1 & = u, \\
    \dot{x}_2 & = x_1, \\
    \dot{x}_3 & = x_1^2 - x_2^2.
   \end{aligned}
  \right.
 \end{equation}
 Here, $\Sone = \R e_1 + \R e_2$ and $[f_1, [f_0, f_1]] = - 2 e_3 \notin \Sone$.
 Thus, Theorem~\ref{Theorem:QAA} states that, in small-time, there will be a 
 quadratic drift preventing the system to be \stlc{$W^{-1,\infty}$}. 
 One can actually compute explicitly the minimal time and prove that:
 \begin{itemize}
  \item if $T \leqslant \pi$, system~(\ref{system.quad.competition}) is not 
  locally controllable in time $T$ (even with large controls),
 \item if $T>\pi$, system~(\ref{system.quad.competition}) is smoothly locally 
 controllable in time $T$.
 \end{itemize}
 Here, there is a quadratic competition, and one of the terms is dominant in 
 small-time:
 \begin{equation} \label{eq.competition}
  x_3(T)=x_3(0) + \int_0^T \left(\dot{x}_2(t)^2 - x_2(t)^2\right) \dt.
 \end{equation}
 The proof of both statements is related to the Poincaré-Wirtinger inequality
 (see~\cite[p.~47]{MR0442564}), that holds for any $T>0$ and any function 
 $\varphi \in H^1_0((0,T),\R)$:
 \begin{equation} \label{eq.wirtinger}
  \int_0^T \varphi'(t)^2 \dt \geqslant \left(\frac{\pi}{T}\right)^2 \int_0^T \varphi(t)^2 \dt.
 \end{equation}
 The constant in the right-hand side is optimal and achieved by the function 
 $\varphi(t):=\sin\left( \frac{\pi t}{T} \right)$.
 
 \bigskip \noindent \textbf{First case: $T \leqslant \pi$}. 
 Let $(x,u) \in C^0([0,T],\R^3) \times L^1((0,T),\R)$ be a trajectory of 
 system~\eqref{system.quad.competition} such that $x_2(0) = x_2(T) = 0$.
 Combining~\eqref{eq.competition} and~\eqref{eq.wirtinger} yields:
 \begin{equation}
  x_3(T) - x_3(0)
  \geqslant \left( \left(\frac{\pi}{T}\right)^2 - 1 \right) \int_0^T x_2(t)^2 \dt 
  \geqslant 0.
 \end{equation}
 Thus the system is not locally controllable in time $T$.

 \bigskip \noindent \textbf{Second case: $T > \pi$}. 
 We prove that, for any $m \in \N$ and any $\eta > 0$, there exists $\delta > 0$
 such that, for any $x^*, x^\dagger \in \R^3$ with 
 $|x^*|+|x^\dagger| \leqslant \delta$, there exists a trajectory 
 of~\eqref{system.quad.competition} 
 $(x,u) \in C^0([0,T],\R^3) \times L^1((0,T),\R)$
 such that $x(0) = x^*$, $x(T) = x^\dagger$ and
  $\norm{u}{C^m} \leqslant \eta$.
  
 By an argument that relies on the Brouwer fixed point theorem 
 (see~\cite[Chapter 8]{MR2302744}), it suffices to prove the existence 
 of $u^\pm \in C^\infty([0,T],\R)$ such that the associated solution of system~(\ref{system.quad.competition}) with initial conditions
 $x^\pm (0)=0$ satisfy $x^\pm(T)=(0,0,\pm 1)$. 

 Thanks to the previous case, one can obtain $u^+$ by rescaling any smooth
 function supported on an interval of length less than $\pi$
 and such that $\int_0^T u(t) \dt = \int_0^T t u(t) \dt = 0$.

 Now, we construct $u^-$. To that end, 
 we introduce $(\rho_\epsilon)_{\epsilon>0}$ a family of functions in 
 $C^\infty_c((0,T),[0,1])$ with
 $\Supp(\rho_\epsilon) \subset [\epsilon,T-\epsilon]$,
 even with respect to $T/2$, 
 such that, for any $t \in (0,T)$, $\rho_\epsilon(t) \to 1$ as $\epsilon \to 0$.
 By the dominated convergence theorem, one has:
 \begin{align}
 \int_0^T \left| \rho_\epsilon(t) \frac{\pi}{T} \cos\left( \frac{\pi t }{T} \right) \right|^2 \dt 
 & \underset{\epsilon \rightarrow 0}{\longrightarrow} \int_0^T \left| \frac{\pi}{T} \cos\left( \frac{\pi t }{T} \right) \right|^2 \dt = \frac{\pi^2}{2T}, \\
 \int_0^T \left| \int_0^t \rho_\epsilon(\tau) \frac{\pi}{T} \cos\left( \frac{\pi \tau }{T} \right) d\tau  \right|^2 \dt
 & \underset{\epsilon \rightarrow 0}{\longrightarrow} \int_0^T \left| \sin\left( \frac{\pi t }{T}  \right) \right|^2 \dt = \frac{T}{2}.
 \end{align}
 Thus, for $\epsilon>0$ small enough, since $T > \pi$:
 \begin{equation}
  \zeta_\epsilon := \int_0^T \left| \rho_\epsilon(t) \frac{\pi}{T} \cos\left( \frac{\pi t }{T} \right) \right|^2 \dt 
  - \int_0^T \left| \int_0^t \rho_\epsilon(\tau) \frac{\pi}{T} \cos\left( \frac{\pi \tau }{T} \right) \dtau  \right|^2 \dt < 0.
 \end{equation}
 Let us fix such an $\epsilon$. The solution of~(\ref{system.quad.competition}) 
 with initial condition $x^-(0)=0$ and control:
 \begin{equation}
  u^-(t) := \frac{1}{\sqrt{-\zeta_\epsilon}} \partial_t 
  \Big[ \rho_\epsilon(t) \frac{\pi}{T} \cos\left( \frac{\pi t }{T} \right) \Big]
 \end{equation}
 satisfies:
 \begin{align}
  x^-_1(t) & = \frac{1}{\sqrt{-\zeta_\epsilon}} \rho_\epsilon(t) \frac{\pi}{T} \cos\left( \frac{\pi t }{T} \right), \\
  x^-_2(t) & = \frac{1}{\sqrt{-\zeta_\epsilon}} \int_0^t \rho_\epsilon(\tau) \frac{\pi}{T} \cos\left( \frac{\pi \tau }{T} \right) \dtau, \\
  x^-_3(T) & =\int_0^T \Big( x^-_1(t)^2 - x^-_2(t)^2 \Big) \dt = \left( \frac{1}{\sqrt{-\zeta_\epsilon}} \right)^2 \zeta_\epsilon = -1.
 \end{align}
 Hence $x^-_1(T)=0$ because $\rho_\epsilon(T)=0$ and $x^-_2(T)=0$ as the integral over 
 $(0,T)$ of a function which is odd with respect to $T/2$.
\end{example}

\subsubsection{Resolution of Sussmann's paradox example}

We turn back to Sussmann's historical Example~\ref{ex.sussmann}.  
Taking a trajectory satisfying $x(0) = 0$, explicit integration 
of~\eqref{system.sussmann} yields:
\begin{equation} \label{ex.sus.x123}
x_1(t) = u_1(t), \quad
x_2(t) = u_2(t) \quad \textrm{and} \quad
x_3(t) = \int_0^t \left(u_1^3(s) + u_2^2(s)\right) \ds.
\end{equation}
Up to the second-order, the component $x_3(T)$ is 
a coercive quadratic form with respect to $\norm{u}{H^{-2}} = \norm{u_2}{L^2}$. 
However, the sum of the orders 0+1+2 is not a good approximation of the nonlinear
solution for the same norm $\norm{u}{H^{-2}}$ when $u$ is small in $L^\infty$.
Indeed:
\begin{equation}
 \int_0^T u_1^3(t) \dt 
 \neq \underset{\norm{u}{L^\infty} \to 0}{o} \left( \int_0^T u_2^2(t) \dt \right).
\end{equation}
Nevertheless, if the smallness assumption on the control in strengthened into
$\norm{u}{W^{1,\infty}} \ll 1$, then the quadratic approximation becomes dominant.
Considering a trajectory such that $x_1(T) = x_2(T) = 0$ and using integration 
by parts yields:
\begin{equation} \label{ex.sus.u13}
 \int_0^T u_1^3(t) \dt 
 = - 2 \int_0^T u_2(t) u_1(t) u(t) \dt 
 = \int_0^T u_2^2(t) \dot{u}(t) \dt.
\end{equation}
Combining~\eqref{ex.sus.x123} and~\eqref{ex.sus.u13} yields:
\begin{equation} \label{eq.sus.drift}
 x_3(T) = \int_0^T u_2^2(t) \left(1+\dot{u}(t)\right) \dt
 \geqslant \frac{1}{2} \norm{u_2}{L^2}^2,
\end{equation}
provided that $\norm{u}{W^{1,\infty}(0,T)} \leqslant 1/2$.
From~\eqref{eq.sus.drift}, we deduce that it is impossible to reach states
of the form $(0,0,-\delta)$ with $\delta > 0$. Hence, system~\eqref{system.sussmann} 
is not \stlc{$W^{1,\infty}$}. 
Our point of view is that condition~\eqref{eq.f03} allows to deny 
small-time controllability with small $W^{1,\infty}$ controls. We think
that this notion is pertinent because it highlights that the quadratic drift can 
only be avoided with highly oscillating controls.
We recover here
the conclusion of Theorem~\ref{Theorem:QAA} with $k = 2$,
in the particular case of Example~\ref{ex.sussmann}.

\subsubsection{Drifts with respect to an invariant manifold}

In Examples~\ref{Example:easy_drift},~\ref{ex.sussmann} 
and~\ref{Example:Competition}, the drifts
hold relatively to the invariant manifold $\M = \Sone$. However, the drift
can also hold with respect to a bent manifold.

\begin{example} \label{Example:drift_bent}
 Let $n = 2$ and $\lambda\in\R$.
 We consider the following control-affine system:
 \begin{equation} \label{system.drift_bent}
 \left\{
 \begin{aligned}
 \dot{x}_1 & = u, \\
 \dot{x}_2 & = 2 u x_1 + \lambda x_1 ^ 2.
 \end{aligned}
 \right.
 \end{equation}
 Straightforward explicit integration of~\eqref{system.drift_bent} yield,
 for any $t \geqslant 0$:
 \begin{equation} \label{eq.drift_bent}
 x_2(t) = x_1^2(t) + \lambda \int_0^t x_1^2(s) \ds.
 \end{equation}
 Equation~\eqref{eq.drift_bent} illustrates the idea that,
 when $\lambda \neq 0$, the state
 $(x_1, x_2)$ endures a quadratic drift with respect to 
 the manifold~\eqref{def.toy.m}. The sign of the
 drift depends on $\lambda$: when it is positive, the state is above
 the manifold and reciprocally.
\end{example}

The direction of the 
drift is not related in any way with the curvature of the manifold
at the origin. Elementary adaptations of~\eqref{system.drift_bent}
in a three dimensional context can be built so that the drift 
holds along a direction which is orthogonal to the curvature
of the manifold.

\subsubsection{Higher-order behaviors}

When $\Stwo = \Sone$, Theorem~\ref{Theorem:QAA} asserts that the state can 
only leave the manifold at cubic order with respect to the control. 
One must then continue the expansion further on, as multiple different
behaviors are possible (see Figure~\ref{Figure:manifolds}). We propose
examples exhibiting different higher-order properties.

\afficherfigure{
\begin{figure}[ht!]
 \begin{center}
  \includegraphics[width=6.6cm]{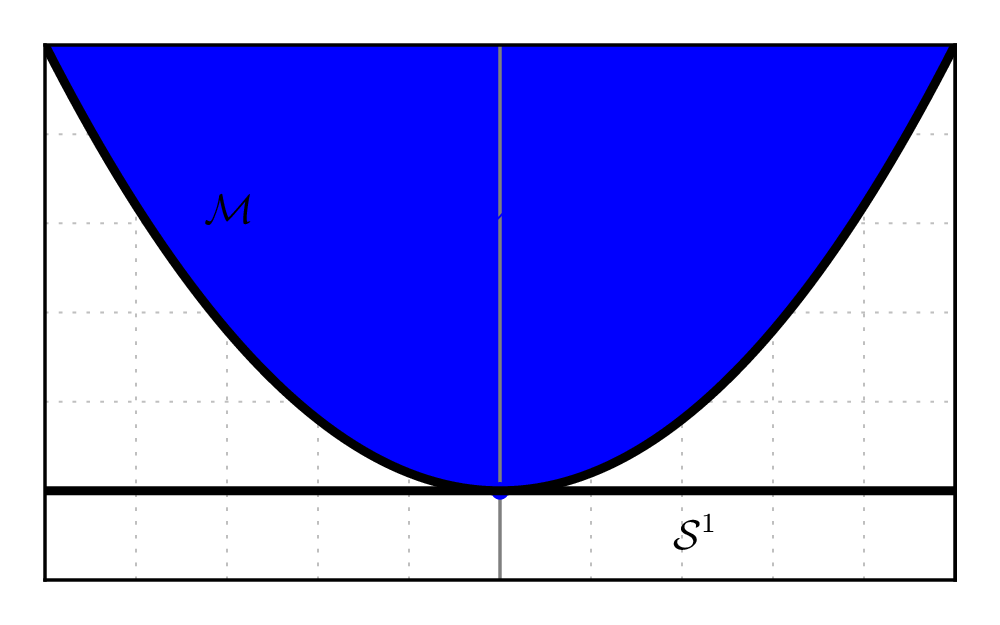}
  \includegraphics[width=6.6cm]{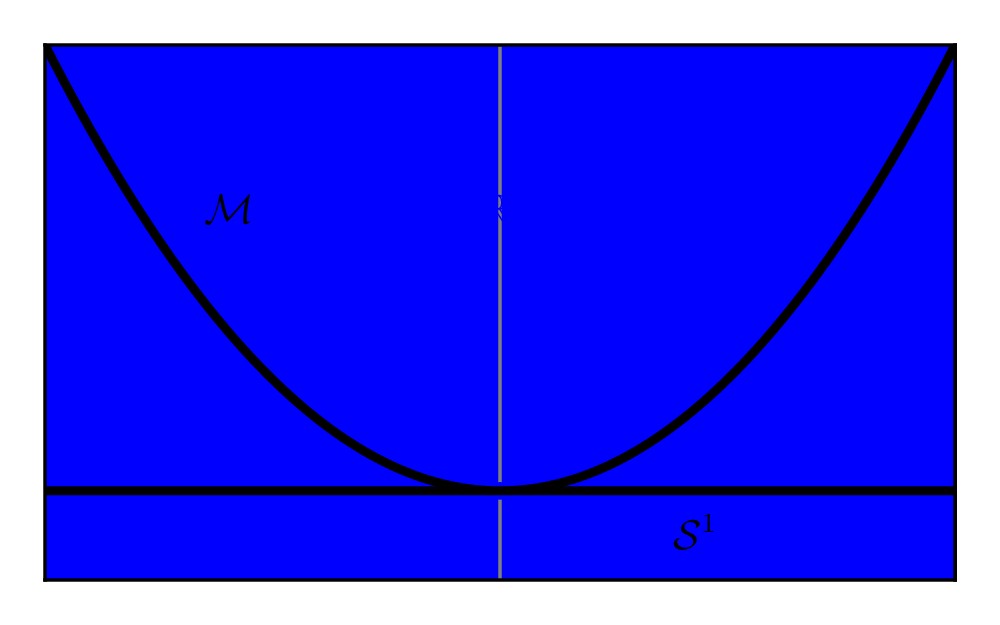}
  \caption{\label{Figure:manifolds} 
   Influence of higher-order terms on the reachable set. 
   \emph{Left}: a higher-order drift pushes the 
   state over the manifold. \emph{Right}: 
   small-time controllability is recovered.}
 \end{center}
\end{figure}
}

\begin{example}
  To recover the situation exposed in 
  Figure~\ref{Figure:manifolds}, left, where a higher-order drift pushes the
  state over the manifold, we perturb system~\eqref{sys.ex.iho} into:
  \begin{equation} \label{sys.ex.iho.2}
  \left\{
  \begin{aligned}
  \dot{x}_1 & = u, \\
  \dot{x}_2 & = 2 u x_1 + x_1^4.
  \end{aligned}
  \right.
  \end{equation}
  This modification changes neither $\Sone$ nor $\Stwo$. 
  For any $t \geqslant 0$:
  \begin{equation}
  x_2(t) = x_1^2(t) + \int_0^t x_1^4(s) \ds.
  \end{equation}
  Hence, the state is constrained to involve within:
  \begin{equation}
  \mathcal{M}_+ := \left\{ (x_1, x_2) \in \R^2, \enskip x_2 - x_1^2 \geqslant 0 \right\}.
  \end{equation}
\end{example}

\begin{example}
 Higher-order terms can also help to recover small-time local controllability, as in 
 Figure~\ref{Figure:manifolds}, right. One possibility is to 
 introduce $\beta \in \R$ and perturb system~\eqref{sys.ex.iho} into:
 \begin{equation} \label{sys.ex.iho.3}
 \left\{
 \begin{aligned}
 \dot{x}_1 & = u + \beta x_2, \\
 \dot{x}_2 & = 2 u x_1.
 \end{aligned}
 \right.
 \end{equation}
 Here again, $\Sone$ and $\Stwo$ are preserved. However, one checks that:
 \begin{equation} \label{ex.iho.3.f3}
 \left[ f_1, \left[ f_1, \left[ f_0, f_1 \right]\right]\right] = 12 \beta e_2.
 \end{equation}
 Hence, when $\beta \neq 0$, the Hermes local controllability 
 condition is satisfied since $\Stwo = \Sone$ and $\Slie{3} = \R^2$
 from~\eqref{ex.iho.3.f3}.
 Thus, one obtains that system~\eqref{sys.ex.iho.3} is \stlc{$L^\infty$}
 (see~\cite[Theorem 2.1, page 688]{MR710995}).
\end{example}

\begin{example}
 The following control-affine example was introduced by Kawski 
 in~\cite[Example 4.1]{MR1061394}, where it is proved that it is 
 \stlc{$L^\infty$}. Let $n = 4$ and consider:
 \begin{equation}
  \left\{
  \begin{aligned}
  \dot{x}_1 & = u, \\
  \dot{x}_2 & = x_1, \\
  \dot{x}_3 & = x_1^3, \\
  \dot{x}_4 & = x_2 x_3.
  \end{aligned}
  \right.
 \end{equation}
 One checks that $\M = \Stwo = \Sone = \R e_1 + \R e_2$. 
 Therefore, Theorem~\ref{Theorem:QAA} does not raise an obstruction
 to controllability (at best, the lost directions will be recovered
 at the cubic order).
\end{example}

\subsubsection{Examples for nonlinear systems} 
\label{Subsection:QANL_examples}

Regarding our work, nonlinear systems exhibit two new features with 
respect to control-affine systems: the possibility of a $u^2$ term and
a $u^3$ term in the vector field.

\begin{example}
 Let $n = 2$ and $\lambda \in \R$. We consider the following non-linear system:
 \begin{equation}
 \left\{
 \begin{aligned}
 \dot{x}_1 & = u, \\
 \dot{x}_2 & = x_1^2 + \lambda u^2 + u^3.
 \end{aligned}
 \right.
 \end{equation}
 When $\lambda \neq 0$, there is a new strong drift, which involves the
 $L^2$-norm of the control and exceeds any other possible drift. It
 holds along the direction $\partial_u^2 f(0,0)$ (here, $2 \lambda e_2$)
 as long as $u$ is small in $L^\infty$-norm (so that the cubic
 term can be ignored).
 When $\lambda = 0$, the strongest possible drift is the one involving
 the $H^{-1}$-norm of the control and is similar to the one described for
 control-affine systems. We recover the obstruction corresponding to 
 the first bad Lie bracket $[f_1,[f_0,f_1]](0) = - 2 e_2$. 
 However, the smallness assumption involves
 the $W^{2,\infty}$-norm of the control in order to ensure that
 the $u^3$ term can be ignored (which is not the case with 
 smallness in $W^{-1,\infty}$).
\end{example}

\subsection{Optimality of the norm hypothesis}
\label{Sec:Optimal}

We give examples of systems proving that the smallness
assumptions used in Theorems~\ref{Theorem:QANL} and~\ref{Theorem:QAA} cannot 
be improved (at least in the range of Sobolev spaces).

\subsubsection{Optimality for control-affine systems}
\label{Subsec:Optimal_affine}

\begin{example} \label{Example:OptimalA} 
 Let $k \in \N^*$. We consider the following control-affine system set in $\R^{k+1}$:
 \begin{equation} \label{ex.opt.a}
 \left\{
 \begin{aligned}
 \dot{x}_1 & = u, \\
 \dot{x}_{j+1} & = x_j, \qquad\qquad \text{for } 1 \leqslant j < k, \\
 \dot{x}_{k+1} & = x_k^2 + x_1^3.
 \end{aligned}
 \right.
 \end{equation}
 The controllable space is such that $\Sone^\perp = \R e_{k+1}$. 
 Moreover, one checks that:
 \begin{align}
 [\ad^{j-1}_{f_0}(f_1), \ad^{j}_{f_0}(f_1)](0) & = 0, \qquad\qquad \text{for } 1 \leqslant j < k, \\
 [\ad^{k-1}_{f_0}(f_1), \ad^{k}_{f_0}(f_1)](0) & = -2 e_{k+1}.
 \end{align}
 Hence, we are in the setting of the quadratic obstruction of order $k$.
 Straightforward integration of~\eqref{ex.opt.a} for trajectories with $x(0) = 0$
 yields $x_j = u_j$ for $1 \leqslant j \leqslant k$ and:
 \begin{equation} \label{ex.opt.a.1}
 x_{k+1}(T) = \int_0^T u_k^2(t) \dt + \int_0^T u_1^3(t) \dt.
 \end{equation}
 Hence, the existence of drift amounts to the existence of a Sobolev  embedding relation. 
\end{example}

Let us prove that our functional setting is optimal. More precisely, we wish to
construct controls realizing both signs in relation~\eqref{ex.opt.a.1} and which
are small in the most regular spaces. We look directly for $u_k$.
We choose a function $\varphi \in C^\infty(\R,\R)$, 
compactly supported in $(0,T)$, such that $\int_0^T \varphi^2 = 1$ 
and $a := \int_0^T \left(\varphi^{(k-1)}\right)^3 \neq 0$. 
Let $\lambda \in \R$ and $\mu \geqslant 1$. We define
the family of dilatations:
\begin{equation}
\varphi_{\lambda,\mu}(t) := \lambda \varphi(\mu t).
\end{equation}
For $m \in [-1, +\infty)$, straightforward scaling arguments lead to:
\begin{equation}
 \norm{\varphi^{(k)}_{\lambda,\mu}}{W^{m,\infty}}
 = |\lambda| \mu^{m+k} \norm{\varphi^{(k)}}{W^{m,\infty}}.
\end{equation}
Hence, the control $u := \varphi^{(k)}_{\lambda,\mu}$ 
is small in $W^{m,\infty}$ when
$|\lambda| \mu^{m+k} \ll 1$.
Moreover,
\begin{equation}
\int_0^T \varphi_{\lambda,\mu}^2 
= \lambda^2 \int_0^T \varphi^2(\mu t) \dt
= \lambda^2 \mu^{-1} \int_0^{\mu T} \varphi^2
= \lambda^2 \mu^{-1}
\end{equation}
and:
\begin{equation}
\begin{split}
\int_0^T \left(\varphi^{(k-1)}_{\lambda,\mu}\right)^3
& = \lambda^3 \mu^{3k-3} \int_0^T \left(\varphi^{(k-1)}\right)^3(\mu t) \dt \\
& = \lambda^3 \mu^{3k-4} \int_0^{\mu T} \left(\varphi^{(k-1)}\right)^3 \\
& = a \lambda^3 \mu^{3k-4}.
\end{split}
\end{equation}
Hence, the cubic term dominates the quadratic term when
$|\lambda| \mu^{3k-3} \gg 1$. 
Since $\mu \geqslant 1$, this relation is compatible with 
smallness of the control in $W^{m,\infty}$
if and only if $m+k < 3k-3$. Thus, the lost direction
can be recovered with controls which are small in $W^{m,\infty}$ for
any $m < 2k-3$. 
Indeed, sign $\lambda$ can take both signs, and since $\varphi$ is
compactly supported in $(0,T)$, we can build controls 
$u^\pm$ and associated trajectories $x^\pm$ of~\eqref{ex.opt.a}
with $x^\pm(0) = 0$ satisfying $x^\pm(T) = (0, \ldots, 0, \pm 1)$.
Using an argument based on the Brouwer fixed point theorem 
as in~\cite[Chapter 8]{MR2302744} enables us to conclude that
system~\eqref{ex.opt.a} is \stlc{$W^{m,\infty}$} for $m < 2k-3$. 
From an homogeneity point of view,
our theorem involving a smallness assumption in $W^{2k-3,\infty}(0,T)$ 
cannot be improved using Sobolev spaces. 

\bigskip

When $k = 2$, system~\eqref{ex.opt.a} corresponds to 
Sussmann's Example~\ref{ex.sussmann} and we conclude that this system 
is \stlc{$W^{m,\infty}$} for any $m \in [-1,1)$. In particular,
we recover the fact, proved by Sussmann, that it is \stlc{$L^\infty$}.

\subsubsection{Optimality for nonlinear systems}
\label{Subsec:Optimal_nonlinear}

\begin{example} 
 Let $k \in \N^*$. We consider the following nonlinear system set in $\R^{k+1}$:
 \begin{equation} \label{ex.opt.a.nl}
 \left\{
 \begin{aligned}
 \dot{x}_1 & = u, \\
 \dot{x}_{j+1} & = x_j, \qquad\qquad \text{for } 1 \leqslant j < k, \\
 \dot{x}_{k+1} & = x_k^2 + u^3.
 \end{aligned}
 \right.
 \end{equation}
 From Theorem~\ref{Theorem:QANL}, this system is not is not \stlc{$W^{2k,\infty}$}.
 The arguments used in the previous paragraph can be adapted to prove that the smallness assumption 
 in $W^{2k,\infty}(0,T)$ cannot be improved using Sobolev spaces. 
\end{example}

\subsubsection{Optimality for bilinear systems}
\label{Subsec:Optimal_bilinear}

One could think that our assumption can be enhanced for some better-behaved
classes of systems, e.g.\ bilinear control systems, for which the pathologic 
cubic term cannot be so easily injected. It turns out to be false. The same
kind of counterexample can be built within the bilinear class. For brevity, 
we give an example only for the first obstruction 
(but examples for higher-order obstructions are also possible).

\begin{example} \label{Example:bilinear}
 We work with $n = 5$. We consider the bilinear system:
 \begin{equation}
 \left\{
 \begin{aligned}
 \dot{x}_1 & = 0, \\
 \dot{x}_2 & = u x_1, \\
 \dot{x}_3 & = 2 u x_2, \\
 \dot{x}_4 & = 3 u x_3, \\
 \dot{x}_5 & = x_5 + u x_2 + x_4,
 \end{aligned}
 \right.
 \end{equation}
 around its equilibrium $x_e = (1, 0, 0, 0, 0)$. One checks that $d = 1$,
 that $\Sone = \R e_2$ and that $\Stwo \not\subset \Sone$. Indeed,
 $[f_1, [f_0, f_1]] (x_e) = - 2 e_5$. Starting from the initial state
 $x(0) = x_e$, we have $x_1 = 1$, $x_2 = u_1$, $x_3 = u_1^2$ and $x_4 = u_1^3$. 
 Using the Duhamel formula, we obtain:
 \begin{equation}
 x_5(T) = \int_0^T e^{T-t} \left(u(t) u_1(t) + u_1^3(t)\right) \dt.
 \end{equation}
 Assuming that we are looking for trajectories satisfying
 $x_2(T) = 0$, we have $u_1(T) = 0$. Hence, integration by parts leads to:
 \begin{equation} \label{ex.opt.b.3}
 x_5(T) e^{-T} 
 = \frac{1}{2} \int_0^T e^{-t} u_1^2(t) \dt 
 + \int_0^T e^{-t} u_1^3(t) \dt.
 \end{equation}
 For singular controls and short times, the exponential multiplier in the 
 integrand does not play an important role. Thus, we recover the key 
 balance of~\eqref{ex.opt.a.1} and the smallness assumption of our theorem
 cannot be improved even within the favorable class of bilinear systems.
\end{example}

\subsubsection{Optimality of the trace hypothesis}
\label{Subsub:Traces}

We give an example illustrating why the drift relation holds for any time only 
when the controls have enough vanishing traces at the initial time. We consider
once more Sussmann's Example~\ref{ex.sussmann} and we compute the solution 
associated with a constant control $u(t) := \eta$, starting from $x(0) = 0$.
One has:
\begin{equation}
 x_1(t) = \eta t, \quad 
 x_2(t) = \frac{1}{2} \eta t^2 
 \quad \text{and} \quad
 x_3(t) = \frac{1}{20} \eta^2 t^5 + \frac{1}{4} \eta^3 t^4.
\end{equation}
Hence, for any $T > 0$, if $|\eta|$ is small enough, then $x_3(T) \geqslant 0$.
However, for any $\eta < 0$, there exists $t$ small enough such that
$x_3(t) < 0$. This illustrates the comment announced in 
Subsection~\ref{Subsection:Comments}: 
\begin{itemize}
 \item the drift relation holds at the final time without any additional 
 assumption on the traces of the control at the initial time,
 \item the drift relation also holds for any time provided that a sufficient 
 number of traces of the control vanish at the initial time.
\end{itemize}
Indeed, if we consider $u(t) = \eta t$ (a control for which $u(0) = 0$), then:
\begin{equation}
 x_1(t) = \frac{1}{2} \eta t^2, \quad 
 x_2(t) = \frac{1}{6} \eta t^3 
 \quad \text{and} \quad
 x_3(t) = \frac{1}{252} \eta^2 t^7 + \frac{1}{56} \eta^3 t^7.
\end{equation}
Hence, if $|\eta|$ is small enough, $x_3(t) \geqslant 0$ for any $t\geqslant 0$.
Of course, similar limiting examples can be built at higher orders. We refer
to paragraph~\ref{Subsubsection:Gagliardo} for a detailed proof of the 
existence of a drift, either at the final time or at any time.

\subsection{Plan of the paper} 

The remainder of the paper is organized in the following way:
\begin{itemize}
 \item In Section~\ref{Section:structure}, we prove preliminary algebraic 
 and analytic results concerning the Lie brackets involved in the spaces 
 $\Sone$ and $\Stwo$ that shed light on their structure.
 \item In Section~\ref{Section:auxiliary}, we introduce auxiliary systems
 useful in the proof of our alternative. They are obtained by iterated
 application of appropriate flows.
 \item In Section~\ref{Section:well_prepared}, we prove our quadratic 
 alternative for a particular class of well-prepared smooth systems, 
 whose linear behavior is a nilpotent iterated integrator.
 \item In Section~\ref{Section:reduction}, we show that the quadratic 
 behavior both of general control-affine and nonlinear systems can be reduced 
 to this particular class of well-prepared systems. We also extend our 
 alternative to non-smooth dynamics.
 \item In Section~\ref{Section:explicit}, we introduce an appropriate second
 order approximation of nonlinear systems that stays exactly within a smooth
 quadratic manifold which can be explicitly computed and corresponds to 
 the second-order approximation of the manifold constructed in the smooth
 case.
 \item In Section~\ref{Section:SSSTLC}, we explore 
 an alternative definition of small-time local controllability.
\end{itemize}

\section{Algebraic properties of the Lie spaces} 
\label{Section:structure}

We explore the structure of the Lie spaces $\Sone$ and $\Stwo$ to gain some 
insight on these objects and to prove useful lemmas. We compute the
involved Lie brackets explicitly and we extend
the definition of these spaces to nonsmooth vector fields.

\subsection{Computations modulo higher-order terms}

With a view to computing explicity the Lie brackets of $\Sone$ and $\Stwo$, we will need to
carry out computations modulo higher-order terms. Indeed, higher-order terms do 
not impact the Lie brackets we want to consider. We use the following concept.

\begin{definition}
 Let $m,p \in \N^*$ and 
 $\varphi, \tilde{\varphi} \in C^\infty(\R^p, \R^p)$. 
 We say that $\varphi$ and $\tilde{\varphi}$ are \emph{equal modulo $m$-th order terms}
 and write $\varphi \eqmod{m} \tilde{\varphi}$
 when there exists $\theta \in C^\infty(\R^p,\mathcal{T}^{m,1}_{p,p})$,
 where $\mathcal{T}^{m,1}_{p,p}$ denotes the space of 
 symmetric multilinear maps from $(\R^p)^m$ to $\R^p$, such that:
 \begin{equation} \label{eq.def.eqmod}
 \forall x \in \R^p, \quad 
 \varphi(x) = \tilde{\varphi}(x) + \theta(x) [x, \ldots, x].
 \end{equation}
\end{definition}

\begin{rk}
 In particular, when $\varphi \eqmod{m} \tilde{\varphi}$,
 since $\theta$ is locally bounded near $0$, we have:
 \begin{equation} \label{eq.phi.O}
 \varphi(x)=\tilde{\varphi}(x) + \underset{x\to 0}{O}\left(|x|^m\right).
 \end{equation}
 However,~\eqref{eq.phi.O} is a weaker hypothesis than~\eqref{eq.def.eqmod}
 because it does not allow differentiation, which is necessary in our
 context for the computation of Lie brackets.
\end{rk}

\begin{lem} \label{Lemma:eqmod123}
 Let $\varphi, \tilde{\varphi}, \psi, \tilde{\psi} \in C^\infty(\R^p, \R^p)$. 
 We assume that:
 \begin{equation}
 \label{lm.eqmod.hyp}
 \varphi \eqmod{1} 0, \quad
 \varphi \eqmod{3} \tilde{\varphi}, 
 \quad \text{and} \quad
 \psi \eqmod{2} \tilde{\psi}.
 \end{equation}
 Then, there holds:
 \begin{equation} \label{lm.eqmod.ccl}
 [\varphi, \psi] \eqmod{2} [\tilde{\varphi}, \tilde{\psi}].
 \end{equation}
\end{lem}

\begin{proof}
 From hypothesis~\eqref{lm.eqmod.hyp}, there exists functions
 $\theta_0 \in C^\infty(\R^p,\mathcal{T}^{1,1}_{p,p})$, 
 $\theta_\varphi \in C^\infty(\R^p,\mathcal{T}^{3,1}_{p,p})$
 and $\theta_\psi \in C^\infty(\R^p,\mathcal{T}^{2,1}_{p,p})$ such
 that, for any $x \in \R^p$:
 \begin{align}
 \label{lm.eqmod.pr.1}
 \varphi(x) & = \theta_0(x) [x], \\
 \label{lm.eqmod.pr.2}
 \varphi(x) & = \tilde{\varphi}(x) + \theta_\varphi(x) [x,x,x], \\
 \label{lm.eqmod.pr.3}
 \psi(x) & = \tilde{\psi}(x) + \theta_\psi(x) [x,x].
 \end{align}
 Differentiating~\eqref{lm.eqmod.pr.2} and~\eqref{lm.eqmod.pr.3} with respect
 to $x$ yields that, for any $x,h \in \R^p$:
 \begin{align}
 \label{lm.eqmod.pr.4}
 \varphi'(x) \cdot h
 & = \tilde{\varphi}'(x) \cdot h
 + (\theta_\varphi'(x) \cdot h) [x,x,x]
 + 3 \theta_\varphi(x) [h,x,x], \\
 \label{lm.eqmod.pr.5}
 \psi'(x) \cdot h
 & = \tilde'{\psi}(x) \cdot h
 + (\theta_\psi'(x) \cdot h) [x,x]
 + 2 \theta_\psi(x) [h,x].
 \end{align}
 From~\eqref{lm.eqmod.pr.3} and~\eqref{lm.eqmod.pr.4}, we have
 $\varphi' \cdot \psi \eqmod{2} \tilde{\varphi}' \cdot \tilde{\psi}$.
 Indeed, for any $x \in \R^p$:
 \begin{equation} \label{lm.eqmod.pr.6}
 \begin{split}
 \varphi'(x) \cdot \psi(x)
 & = \tilde{\varphi}'(x) \cdot \tilde{\psi}(x)
 + \tilde{\varphi}'(x) \cdot \theta_\psi(x) [x,x] \\
 & \qquad + (\theta_\varphi'(x) \cdot \psi(x)) [x,x,x]
 + 3 \theta_\varphi(x) [\psi(x),x,x]
 \end{split}
 \end{equation}
 From~\eqref{lm.eqmod.pr.1},~\eqref{lm.eqmod.pr.2} and~\eqref{lm.eqmod.pr.5},
 we have 
 $\psi' \cdot \varphi \eqmod{2} \tilde{\psi}' \cdot \tilde{\varphi}$.
 Indeed, for any $x \in \R^p$:
 \begin{equation} \label{lm.eqmod.pr.7}
 \begin{split}
 \psi'(x) \cdot \varphi(x)
 & = \tilde{\psi}'(x) \cdot \tilde{\varphi}(x)
 + \tilde{\psi}'(x) \cdot \theta_\varphi(x) [x,x,x] \\
 & \qquad 
 + (\theta_\psi'(x) \cdot \varphi(x)) [x,x]
 + 2 \theta_\psi(x) [\theta_0(x)[x],x].
 \end{split}
 \end{equation}
 Combining~\eqref{lm.eqmod.pr.6} and~\eqref{lm.eqmod.pr.7} yields~\eqref{lm.eqmod.ccl}.
\end{proof}

\begin{lem} \label{Lemma:eqmod22}
 Let $\varphi, \tilde{\varphi}, \psi, \tilde{\psi} \in C^\infty(\R^p, \R^p)$. 
 We assume that:
 \begin{gather}
 \label{lm.eqmod2.hyp}
 \varphi \eqmod{2} \tilde{\varphi}
 \quad \text{and} \quad
 \psi \eqmod{2} \tilde{\psi}.
 \end{gather}
 Then, there holds:
 \begin{equation} \label{lm.eqmod2.ccl}
 [\varphi, \psi](0) = [\tilde{\varphi}, \tilde{\psi}](0).
 \end{equation}
\end{lem}

\begin{proof}
 We proceed as above. Thanks to~\eqref{lm.eqmod2.hyp}, we introduce functions
 $\theta_\varphi \in C^\infty(\R^p,\mathcal{T}^{2,1}_{p,p})$
 and $\theta_\psi \in C^\infty(\R^p,\mathcal{T}^{2,1}_{p,p})$ such
 that, for any $x \in \R^p$:
 \begin{align}
 \label{lm.eqmod2.pr.1}
 \varphi(x) & = \tilde{\varphi}(x) + \theta_\varphi(x) [x,x], \\
 \label{lm.eqmod2.pr.2}
 \psi(x) & = \tilde{\psi}(x) + \theta_\psi(x) [x,x].
 \end{align}
 Differentiating~\eqref{lm.eqmod2.pr.1} and~\eqref{lm.eqmod2.pr.2} with respect
 to $x$ yields that, for any $x,h \in \R^p$:
 \begin{align}
 \label{lm.eqmod2.pr.3}
 \varphi'(x) \cdot h
 & = \tilde{\varphi}'(x) \cdot h
 + (\theta_\varphi'(x) \cdot h) [x,x]
 + 2 \theta_\varphi(x) [h,x], \\
 \label{lm.eqmod2.pr.4}
 \psi'(x) \cdot h
 & = \tilde{\psi}'(x) \cdot h
 + (\theta_\psi'(x) \cdot h) [x,x]
 + 2 \theta_\psi(x) [h,x].
 \end{align}
 From~\eqref{lm.eqmod2.pr.2} and~\eqref{lm.eqmod2.pr.3}, we have
 $\varphi' \cdot \psi \eqmod{1} \tilde{\varphi}' \cdot \tilde{\psi}$.
 From~\eqref{lm.eqmod2.pr.1},~\eqref{lm.eqmod2.pr.4}, we have 
 $\psi' \cdot \varphi \eqmod{1} \tilde{\psi}' \cdot \tilde{\varphi}$.
 Thus, $[\varphi,\psi] \eqmod{1} [\tilde{\varphi},\tilde{\psi}]$, which 
 yields~\eqref{lm.eqmod2.ccl} by evaluation at $x = 0$.
\end{proof}

\begin{lem} \label{Lemma:Taylor}
 Let $p,q \in \N^*$, $m \in \N$ and $\varphi \in C^{\infty}(\R^p,\R^q)$.
 We denote by $T_m(\varphi) \in C^\infty(\R^p,\R^q)$ the Taylor 
 expansion of order $m$ at the origin associated with $\varphi$. Then:
 \begin{equation} \label{Lemma:Taylor.ccl}
 \varphi \eqmod{m+1} T_m(\varphi).
 \end{equation}
\end{lem}

\begin{proof}
 This lemma is a consequence of the Taylor formula with integral remainder for
 vector-valued multivariate functions:
 \begin{equation} \label{taylor.pr}
 \varphi(x) = T_m(\varphi)(x)
 + \frac{1}{m!} \int_0^1 (1-t)^m 
 \varphi^{(m+1)} (tx) [x, \ldots, x] \dt,
 \end{equation}
 where $\varphi^{(m+1)}$ denotes the usual
 multi-linear differential map taking $m+1$ arguments in $\R^p$.
 From~\eqref{taylor.pr}, we deduce~\eqref{Lemma:Taylor.ccl}.
\end{proof}

\subsection{Explicit computation of the Lie brackets}
\label{Subsection:C2}

We compute explicitly first and second-order Lie brackets
using only low-order derivatives of the vector fields.
Let $f_0, f_1 \in C^\infty(\R^n,\R^n)$ with $f_0(0) = 0$.

\subsubsection{First-order Lie brackets}

We recall the following notations (already used in Subsection~\ref{Subsec:linear}):
\begin{equation} \label{def:H0b}
H_0 := f_0'(0) \in \mathcal{M}_n(\R)
\quad \text{and} \quad
b := f_1(0) \in \R^n.
\end{equation}
Since $f_0(0) = 0$, for any smooth vector field $g \in C^\infty(\R^n,\R^n)$
and any $k \in \N$:
\begin{equation} \label{adkf0g}
 \ad_{f_0}^k(g)(0) = (-H_0)^k g(0)
\end{equation}
In particular, applying~\eqref{adkf0g} to $f_1$, we define the
vectors $b_k \in \R^n$ for $k \in \N$ as:
\begin{equation} \label{def.vk}
 b_k := \ad_{f_0}^k(f_1)(0) = (-H_0)^k b.
\end{equation}
Hence, as these brackets span $\Sone$, we obtain:
\begin{equation} \label{S1H0b}
 \Sone 
 = \Span\left\{ \ad^k_{f_0}(f_1)(0), \enskip k \in \N \right\}
 = \Span \left\{ b_k, \enskip k \in \N \right\}.
\end{equation}
Therefore, the Kalman rank condition~\eqref{kalman} is equivalent to 
$\Sone = \R^n$, thanks to~\eqref{S1H0b} and the Cayley-Hamilton theorem applied to
the matrix $H_0 \in \mathcal{M}_n(\R)$.

\subsubsection{Second-order Lie brackets}

To carry on the computations at second order, we define the following matrix:
\begin{equation} \label{def:H1}
H_1 := f_1'(0) \in \mathcal{M}_n(\R) 
\end{equation}
We also introduce the following third order tensor, which defines a bilinear
map from $\R^n \times \R^n$ to $\R^n$ (and can be thought of as the 
Hessians of the components of $f_0$):
\begin{equation} \label{def.Q0}
Q_0 := \frac{1}{2} f_0''(0).
\end{equation}
By induction on $k \in \N$, we also define the linear operators:
\begin{align} 
\label{def.L0}
L_0 & := H_1, \\
\label{def.Lk}
L_{k+1} & := L_k H_0 - H_0 L_k - 2 Q_0(b_k, \cdot).
\end{align}
Thanks to these definitions, we can compute the first-order Lie brackets
more precisely.

\begin{lem} \label{Lemma:adkbklk}
 For any $k \in \N$:
 \begin{equation}
  \label{eq.adk.o2}
  \ad^k_{f_0}(f_1)(x) 
  \eqmod{2} b_k + L_k x.
 \end{equation}
\end{lem}

\begin{proof}
 From Lemma~\ref{Lemma:Taylor}, $f_0(0)=0$, 
 definitions~\eqref{def:H0b},~\eqref{def:H1} and~\eqref{def.Q0}:
 \begin{align}
 \label{eq.pr.f0.xeh}
 f_0(x) & \eqmod{3} H_0 x + Q_0(x,x), \\
 \label{eq.pr.f1.xeh}
 f_1(x) & \eqmod{2} b + H_1 x.
 \end{align}
 For $k = 0$,~\eqref{def.vk},~\eqref{def.L0} and~\eqref{eq.pr.f1.xeh}
 prove~\eqref{eq.adk.o2}. Proceeding by induction on $k \in \N$, 
 we evaluate the next bracket using Lemma~\ref{Lemma:eqmod123} with
 $\varphi(x) := f_0(x)$, $\tilde{\varphi}(x) := H_0 x + Q_0(x,x)$,
 $\psi(x) := \ad^k_{f_0}(f_1)(x)$ and
 $\tilde{\psi}(x) := b_k + L_k x$. Thus, using~\eqref{def.vk},~\eqref{def.Lk}
 and~\eqref{eq.pr.f0.xeh}:
 \begin{equation} \label{eq.pr.leftout}
 \begin{split}
 \left[f_0, \ad^k_{f_0}(f_1)\right](x) 
 & \eqmod{2} L_k (H_0 x + Q_0 (x,x)) - (H_0 + 2 Q_0(x, \cdot)) (b_k + L_k x) \\
 & \eqmod{2} b_{k+1} + L_{k+1} x + L_k Q_0(x,x) - 2 Q_0(x, L_k x) \\
 & \eqmod{2} b_{k+1} + L_{k+1} x,
 \end{split}
 \end{equation} 
 since the terms which have been dropped from~\eqref{eq.pr.leftout} are 
 bilinear in $x$. 
\end{proof}    

Using~\eqref{adkf0g},~\eqref{eq.adk.o2} and Lemma~\ref{Lemma:eqmod22},
we obtain the values of the second-order Lie brackets:
\begin{align}
 \label{eq.adkj}
 \left[\ad^k_{f_0}(f_1), \ad^j_{f_0}(f_1)\right](0) 
 & = L_j b_k - L_k b_j, \\
 \left(\ad^i_{f_0} \left(\left[\ad^k_{f_0}(f_1), \ad^j_{f_0}(f_1)\right]\right)\right)(0) 
 & = (-H_0)^i (L_j b_k - L_k b_j).
\end{align}
Since $\Stwo$ is spanned by such brackets, we obtain:
\begin{equation} \label{S2}
 \Stwo = \Span \left\{ (-H_0)^i (L_j b_k - L_k b_j), 
 \enskip (i,j,k) \in \N^3 \right\}.
\end{equation}

\subsection{Definitions for nonsmooth vector fields}
\label{Subsec:S2_nonsmooth}

Let $f_0 \in C^2(\R^n,\R^n)$ and $f_1 \in C^1(\R^n, \R^n)$. For nonlinear systems,
these assumptions are satisfied if
$f \in C^2(\R^n,\R^n)$ and $f_0$ and $f_1$ are defined
by~\eqref{def.f0f1}. With such regularity, the 
objects~\eqref{def:H0b},~\eqref{def:H1} and~\eqref{def.Q0} are well-defined.
Hence, definitions~\eqref{def.vk},~\eqref{def.L0} and~\eqref{def.Lk} make sense.
Therefore, one can use~\eqref{S1H0b} and~\eqref{S2} as definitions of the
spaces $\Sone$ and $\Stwo$. Moreover, relations like~\eqref{eq.adkj} can be 
used to give a meaning to individual brackets.

Another more flexible approach (which leads to the same spaces) is to introduce
regularized vector fields $\hat{f}_0 := T_2 f_0$ and $\hat{f}_1 := T_1 f_1$,
corresponding respectively to second-order and first-order Taylor expansions 
of $f_0$ and $f_1$ at $0$. Hence $\hat{f}_0, \hat{f}_1 \in C^\infty(\R^n,\R^n)$.
Then, one defines $\Sone$ and $\Stwo$ as the usual Lie spaces associated
with these regularized vector fields.

\subsection{Algebraic relations between second-order brackets}

The following algebraic relations highlight the fact that our theorems could 
actually be stated using other equivalent brackets (see~\eqref{equiv.dk}). 
Although these results are not new, 
we include the proofs (inspired from~\cite[pages~279-280]{MR0433288})
for the sake of completeness.

First, for any $i, l \in \N$ such that $0 \leqslant l \leqslant i$, it can
be proved by induction on $l$, using the Jacobi identity and the skew symmetry
of the bracketing operation that:
 \begin{equation} \label{formule_Krener}
 [f_1,\ad_{f_0}^{i}(f_1)] =  \sum_{j=0}^{l-1} (-1)^j [f_0,[\ad_{f_0}^j(f_1),\ad_{f_0}^{i-j-1}(f_1)]] 
 + (-1)^l [\ad_{f_0}^{l}(f_1),\ad_{f_0}^{i-l}(f_1)].
 \end{equation}
This formula has the following consequences.

\begin{prop} \label{Prop:Krener}
 Let $k \in \N^*$. We assume that $f_0, f_1 \in C^{\infty}(\R^n,\R^n)$ and~\eqref{hyp_j} holds. 
 \begin{enumerate}
  
  \item Then $[\ad_{f_0}^a(f_1),\ad_{f_0}^b(f_1)](0) \in \Sone$ for every $a,b \in \N$ such that $a+b \leqslant 2k-2$.
  
  \item Statement~\eqref{hyp_k} is equivalent to each of the following two statements:
  \begin{itemize}
   \item $[\ad_{f_0}^l(f_1),\ad_{f_0}^{2k-1-l}(f_1)](0) \notin \Sone$, for every $l \in \{0,...,2k-1\}$,
   \item $[\ad_{f_0}^l(f_1),\ad_{f_0}^{2k-1-l}(f_1)](0) \notin \Sone$, for some $l \in \{0,...,2k-1\}$.
  \end{itemize}
  Moreover:
  \begin{equation} \label{equiv.dk}
  d_k  = (-1)^{k-1+l}\, \mathbb{P}^\perp\Big(  [ \ad_{f_0}^{l}(f_1) , \ad_{f_0}^{2k-1-l}(f_1) ] (0)  \Big).
  \end{equation}          
  
  \item If~\eqref{hyp_k} holds, then the family $(f_1(0), \ldots, \ad_{f_0}^{k-1} (f_1)(0))$ is linearly independent. Thus, it can only hold for $k \leqslant d$.
  
  \item If $k > \mathrm{dim~} \Sone$, then $[\ad_{f_0}^a(f_1),\ad_{f_0}^b(f_1)](0) \in \Sone$ for every $a,b \in \N$.
  
 \end{enumerate}
\end{prop} 

\begin{proof}
 \textbf{Statement 1.} We decompose the proof in three steps.
 \begin{itemize}
  \item \textit{Step A.
   We prove that, if $[\ad_{f_0}^a(f_1),\ad_{f_0}^{b}(f_1)](0)$ belongs  to $\Sone$ for every $a,b \in \N$ such that $a+b \leqslant 2m-1$
   and for some $m \in \N^*$, then  $[\ad_{f_0}^a(f_1),\ad_{f_0}^{b}(f_1)](0)$ belongs  to $\Sone$ for every $a,b \in \N$ such that $a+b \leqslant 2m$.}
  
  Let $m \in \N$. We assume that $[\ad_{f_0}^a(f_1),\ad_{f_0}^{b}(f_1)](0) \in \Sone$ for every $a,b \in \N$ such that $a+b \leqslant 2m-1$.
  We deduce from the stability of $\Sone$ with respect to bracketing by $f_0$ and formula (\ref{formule_Krener}) with $i=2m$ and $l=m$ that 
  $[f_1,\ad_{f_0}^{2m}(f_1)](0)$ belongs to $\Sone$. Then, we deduce from the stability of $\Sone$ with respect to bracketing by $f_0$
  and formula  (\ref{formule_Krener}) with $i=2m$ that $[\ad_{f_0}^{l}(f_1),\ad_{f_0}^{2m-l}(f_1)](0) \in \Sone$ for every $l\in\{0,...,2m\}$.
  
  \item \textit{Step B. We prove that, if $[\ad_{f_0}^a(f_1),\ad_{f_0}^{b}(f_1)](0)$ belongs  to $\Sone$ for every $a,b \in \N$ such that $a+b \leqslant 2m$,
   and $[\ad_{f_0}^{m}(f_1),\ad_{f_0}^{m+1}(f_1)](0)$ belongs  to $\Sone$ for some $m \in \N$, 
   then, for every $a,b \in \N$ such that $a+b \leqslant 2m+1$,
   $[\ad_{f_0}^a(f_1),\ad_{f_0}^{b}(f_1)](0)$ belongs  to $\Sone$. }
  
  This is direct consequence of formula (\ref{formule_Krener}) with $i=2m+1$ and the stability of $\Sone$ with respect to bracketing by $f_0$.
  
  \item \textit{Step C. We prove Statement 1.}
  
  Using the case $j=1$ in assumption  (\ref{hyp_j})  and Step 2 with $m=1$ we obtain that 
  $[\ad_{f_0}^a(f_1),\ad_{f_0}^{b}(f_1)](0) \in \Sone$ for every $a,b \in \N$ such that $a+b \leqslant 2$.
  Then, using the case $j=2$ in assumption (\ref{hyp_j}), Step 3 with $m=1$ and Step 2 with $m=2$ we obtain that
  $[\ad_{f_0}^a(f_1),\ad_{f_0}^{b}(f_1)](0) \in \Sone$ for every $a,b \in \N$ such that $a+b \leqslant 4$. The proof is carried on by iteration.
 \end{itemize}
 
 \noindent \textbf{Statement 2.}
 This statement follows from formula (\ref{formule_Krener}) with $i=2k-1$, 
 Statement 1 and the stability of $\Sone$ with respect to bracketing by $f_0$.
 \bigskip
 
 \noindent \textbf{Statement 3.} We assume that (\ref{hyp_j}) holds 
 and the family $(f_1(0), \ldots \ad_{f_0}^{k-1} (f_1)(0))$ is not linearly independent. If $k=1$, then $f_1 (0) = 0$. Thus $\Sone=\{0\}$ and:
 \begin{equation}
 [f_1,[f_1,f_0]](0)
 = f_0''(0)\big(f_1(0),f_1(0)\big) + f_0'(0) f_1'(0) f_1(0) - 2 f_1'(0) f_0'(0) f_1(0) = 0 \in \Sone.
 \end{equation}
 Now, we assume that $k \geqslant 2$. 
 There exists $a_0,...,a_{k-2} \in \R$ such that:
 \begin{equation} \label{ar1}
 \ad_{f_0}^{k-1} (f_1)(0) = \sum_{r=0}^{k-2} a_r \ad_{f_0}^{r} (f_1)(0).
 \end{equation}
 Then:
 \begin{equation} \label{ar2}
 \ad_{f_0}^{k} (f_1)(0) =  \sum_{r=0}^{k-2} a_r \ad_{f_0}^{r+1} (f_1)(0),
 \end{equation}
 because $\ad_{f_0}^r(f_1)(0)=H_0^r f_1(0)$ for every $r \in \N$.
 We deduce from (\ref{hyp_j}) and Statement 1 that,
 for any $a, b \in \N$ such that $a + b \leqslant 2k-2$:
 \begin{equation} \label{ada.adb}
 \Big( \ad_{f_0}^a(f_1) \Big)'\Big(\ad_{f_0}^b (f_1)(0)\Big) \sim 
 \Big( \ad_{f_0}^b (f_1) \Big)' \Big(\ad_{f_0}^a  (f_1)(0)\Big),
 \end{equation}
 where $\sim$ means equality modulo additive terms in $\Sone$ 
 (equivalence classes in~$\R^n / \Sone$).
 Using~\eqref{ar1} and~\eqref{ar2}, we have:
 \begin{equation}
 \begin{split}
 & [\ad_{f_0}^{k}(f_1), \ad_{f_0}^{k-1}(f_1)](0) \\
 &=    \Big( \ad_{f_0}^{k-1}(f_1) \Big)'\Big( \ad_{f_0}^k (f_1)(0)\Big) - \Big( \ad_{f_0}^{k}(f_1) \Big)' \Big( \ad_{f_0}^{k-1}(f_1)(0)\Big) \\
 &=    \sum_{r=0}^{k-2} a_r 
 \left( \Big( \ad_{f_0}^{k-1}(f_1) \Big)' \Big(\ad_{f_0}^{r+1} (f_1)(0)\Big) 
 - \Big( \ad_{f_0}^{k}(f_1) \Big)' \Big( \ad_{f_0}^{r}(f_1)(0)\Big) \right).
 \end{split}
 \end{equation}
 Then, using~\eqref{ada.adb}, we have:
  \begin{equation}
  \begin{split}
  & [\ad_{f_0}^{k}(f_1), \ad_{f_0}^{k-1}(f_1)](0) \\
  & \sim \sum_{r=0}^{k-2} a_r \left(
  \Big( \ad_{f_0}^{r+1}(f_1) \Big)'\Big(\ad_{f_0}^{k-1} (f_1)(0)\Big) - \Big( \ad_{f_0}^{r}(f_1) \Big)' \Big(\ad_{f_0}^{k}(f_1)(0)\Big) \right) \\
  & \sim \sum_{r,R=0}^{k-2} a_r a_R \left(
  \Big(\ad_{f_0}^{r+1}(f_1) \Big)' \Big(\ad_{f_0}^{R} (f_1)(0)\Big) - \Big( \ad_{f_0}^{r}(f_1) \Big)'\Big( \ad_{f_0}^{R+1}(f_1)(0) \Big) \right) \\
  & \sim \sum_{r=0}^{k-2} a_r^2 L_r + \sum_{0 \leqslant r < R \leqslant k-2} a_r a_R L_{r,R},
  \end{split}
  \end{equation}
 where:
 \begin{align}
 L_r & := [\ad_{f_0}^{r+1}(f_1),\ad_{f_0}^{r}(f_1)](0), \\
 L_{r,R} & := 
 \Big( \ad_{f_0}^{r+1}(f_1) \Big)'\Big( \ad_{f_0}^{R} (f_1)(0)\Big) - \Big( \ad_{f_0}^{r}(f_1) \Big)'\Big( \ad_{f_0}^{R+1}(f_1)(0)\Big) \\
 \nonumber
 & \quad + \Big( \ad_{f_0}^{R+1}(f_1) \Big)'\Big(\ad_{f_0}^{r} (f_1)(0)\Big) - \Big( \ad_{f_0}^{R}(f_1) \Big)'\Big( \ad_{f_0}^{r+1}(f_1)(0)\Big) \\
 \nonumber
 & = [\ad_{f_0}^{r+1}(f_1),\ad_{f_0}^{R}(f_1)](0) - [\ad_{f_0}^{r}(f_1),\ad_{f_0}^{R+1}(f_1)](0).
 \end{align}
 Since $2r+1,r+R+1\leqslant 2k-3$, $L_r$ and $L_{r,R}$ belong to $\Sone$ for every $0\leqslant r < R \leqslant k-2$.
 Therefore $[\ad_{f_0}^{k-1}(f_1),\ad_{f_0}^{k}(f_1)](0)$ belong to $\Sone$, i.e.~\eqref{hyp_k} does not hold.
 \bigskip
 
 \noindent \textbf{Statement 4.} 
 By iterating Statement 3, we obtain that $[\ad_{f_0}^{j-1}(f_1),\ad_{f_0}^{j}(f_1)](0) \in \Sone$ for every $j \in \N$.
 Then, Statement 1 yields the conclusion.
\end{proof}

In particular, these algebraic relations lead to constraints on the
different ways that one can have $\Stwo \not \subset \Sone$.
More precisely, we prove:

\begin{lem} \label{Lemma:ComingOut}
 Let $f_0, f_1 \in C^\infty(\R^n,\R^n)$ with $f_0(0) = 0$
 and $d := \mathrm{dim~} \Sone$.
 Assume that $\Stwo \not\subset \Sone$. Then, there exists
 $1 \leqslant k \leqslant d$ such that~\eqref{hyp_j} and~\eqref{hyp_k}
 hold.
\end{lem}

\begin{proof}
 By contradiction, let us assume that~\eqref{hyp_j} holds with $k > d$.
 From Statement 4 of Proposition~\ref{Prop:Krener}, we get that
 $[\ad^a_{f_0}(f_1), \ad^b_{f_0}(f_1)](0) \in \Sone$ for any
 $a, b \in \N$. Since $\Sone$ is stable with respect to bracketing
 by $f_0$, we also have $\ad^c_{f_0} ([\ad^a_{f_0}(f_1), \ad^b_{f_0}(f_1)])(0)
 \in \Sone$ for any $a,b,c \in \N$. From~\eqref{S2}, this yields
 that $\Stwo \subset \Sone$, which contradicts our assumption.
 Thus, there exists $k \leqslant d$ such that~\eqref{hyp_k} holds.
 Taking the smallest such $k \geqslant 1$ such that~\eqref{hyp_k}
 holds ensures that~\eqref{hyp_j} also holds.
\end{proof}

\section{Construction of auxiliary systems} \label{Section:auxiliary}

We construct by induction auxiliary systems that are useful in the proof of our 
theorems.

\subsection{Definitions and notations}

Let $f \in C^\infty(\mathbb{R}^n \times \mathbb{R},\mathbb{R}^n)$
with $f(0,0) = 0$.
We define the $C^\infty$ map $G_0:\R \times \R^n \rightarrow \mathbb{R}^n$ as:
\begin{equation} \label{DEF_G0}
G_0(\tau,p):=\left\lbrace \begin{aligned}
& \frac{1}{\tau^2} \left( f(p,\tau) - f_0(p) - \tau f_1(p) \right) & \quad \text{ for } \tau \neq 0, \\
& \frac{1}{2} \partial_u^2 f(p,0)                         & \quad \text{ for } \tau=0.
\end{aligned}\right.
\end{equation}
Let $T>0$ and $(x,u) \in C^0([0,T],\R^n) \times L^\infty((0,T),\R)$ a trajectory 
of~\eqref{system.fxu} with $x(0)=0$. 
In the case of a control-affine system, the control $u$ is only assumed to belong to $L^1((0,T),\R)$.
We introduce, for $1 \leqslant j \leqslant d$:

\begin{itemize}
 
 \item the smooth vector fields:
 \begin{equation} \label{def.fj}
 f_j := (-1)^{j-1} \ad_{f_0}^{j-1}(f_1),
 \end{equation}
 
 \item the associated flows $\phi_j$, 
 that are smooth and well-defined on an open neighborhood 
 $\Omega_j \subset \R \times \R^n$ of $(0,0)$,
 which are defined as the solutions to:
 \begin{equation} \label{flow_fj}
 \left\{
 \begin{aligned}
 \partial_\tau \phi_j (\tau,p) 
 & = f_j \big( \phi_j(\tau,p) \big), \\
 \phi_j(0, p) & = p.
 \end{aligned}
 \right.
 \end{equation}
 and we will sometimes write $\phi_j^\tau(p)$ instead of $\phi_j(\tau,p)$,
 
 \item the smooth maps $F_j: \Omega_j \to \R^n$ defined by:
 \begin{equation} \label{def:Fj}
 F_j(\tau, p) := \left( \partial_p \phi_j(\tau,p) \right)^{-1} f_0 \big( \phi_j(\tau,p) \big),
 \end{equation}
 
 \item the smooth maps $G_j: \Omega_j \to \R^n$, defined by:
 \begin{equation} \label{def:Gj}
 G_j(\tau,p):=\left\lbrace \begin{aligned}
 & \frac{1}{\tau^2} \big( F_j(\tau,p) - F_j(0,p) - \tau 
 \partial_\tau F_j(0,p) \big)  & \quad \text{for } \tau \neq 0, \\
 & \frac{1}{2} \partial_\tau^2 F_j(0,p)
 & \quad \text{for } \tau=0,
 \end{aligned}\right.
 \end{equation}
 
 \item the auxiliary states $\aux_j:(0,T) \rightarrow \R^n$, defined by
 induction by:
 \begin{equation} \label{def:xj} 
 \aux_0:=x \quad \text{and} \quad
 \aux_{j+1}:=\phi_{j+1} \left(-u_{j+1} , \aux_j \right).
 \end{equation}
\end{itemize}

\subsection{Evolution of the auxiliary states}

We start by computing some derivatives of the utility functions $F_j$, from which
we then deduce by induction the evolution equation for the auxiliary states.

\begin{lem} \label{Lem_Fj}
 Let $1 \leqslant j \leqslant d$. One has:
 \begin{align}
 \label{Fj_1}
 \partial_\tau F_j(0, \cdot) & = f_{j+1}, \\
 \label{Fj_2}
 \partial_\tau^2 F_j(0, \cdot) & = [f_j, f_{j+1}].
 \end{align} 
\end{lem}

\begin{proof}
 Differentiating definition~\eqref{def:Fj} with respect to $\tau$ yields:
 \begin{equation} \label{lem_fj_1}
 \partial_\tau F_j
 = - (\partial_p \phi_j)^{-1} \partial_{\tau p} \phi_j 
 (\partial_p \phi_j)^{-1} f_0(\phi_j)
 + (\partial_p \phi_j)^{-1} f_0'(\phi_j) \partial_\tau \phi_j.
 \end{equation}
 Applying Schwarz's theorem and using~\eqref{flow_fj}, one computes:
 \begin{equation} \label{lem_fj_2}
 \partial_{\tau p} \phi_j = \partial_{p \tau} \phi_j 
 = f_j'(\phi_j) \partial_p \phi_j.
 \end{equation}
 Gathering~\eqref{lem_fj_1},~\eqref{lem_fj_2} and using~\eqref{flow_fj} yields:
 \begin{equation} \label{lem_fj_3}
 \begin{split}
 \partial_\tau F_j 
 & = - (\partial_p \phi_j)^{-1} f_j'(\phi_j) f_0(\phi_j)
 + (\partial_p \phi_j)^{-1} f_0'(\phi_j) f_j(\phi_j) \\
 & = (\partial_p \phi_j)^{-1} [f_j, f_0](\phi_j).
 \end{split}
 \end{equation}
 Equation~\eqref{lem_fj_3} proves~\eqref{Fj_1} by evaluation at $\tau = 0$
 thanks to definition~\eqref{def.fj}. Applying the same proof method to
 $\partial_\tau F_j$ yields~\eqref{Fj_2}.
\end{proof}

\begin{lem} \label{Lem:eq_xij}
 Let $1 \leqslant j \leqslant d$ and $\aux_j$ be defined by~\eqref{def:xj}.
 Then $\aux_j(0) = x(0)$ and $\aux_j$ satisfies:
 \begin{equation} \label{HR_l}
 \dot{\aux}_j = f_0(\aux_j) + u_j f_{j+1}(\aux_j)  + \sum_{l=0}^j u_l^2 \mathcal{K}_{l,j},
 \end{equation}
 where we define:
 \begin{equation} \label{def.Klj}
  \mathcal{K}_{l,j} := \left\lbrace 
  \begin{aligned}
    & \left( \partial_p \phi_j(u_j,\aux_j) \right)^{-1} \cdots \left( 
    \partial_p \phi_{l+1}(u_{l+1},\aux_{l+1}) \right)^{-1} G_l (u_l,\aux_l)
    & \text{ for } l<j, \\
    & G_j ( u_j , \aux_j ) & \text{ for } l=j.
  \end{aligned}
  \right.
 \end{equation}
\end{lem}

\begin{proof}
 Setting $\aux_0 := x$ and using (\ref{DEF_G0}) proves that~\eqref{HR_l} 
 holds for $j = 0$.
 Let $0 \leqslant j < d$.  We assume that~\eqref{HR_l}
 holds. From~\eqref{def:xj}:
 \begin{equation} \label{lem.aux.1}
 \aux_j = \phi_{j+1}(u_{j+1}, \aux_{j+1}).
 \end{equation}
 Differentiating~\eqref{lem.aux.1} with respect to time yields:
 \begin{equation} \label{lem.aux.2}
 \dot{\aux}_j = \dot{u}_{j+1} \partial_\tau \phi_{j+1} (u_{j+1},\aux_{j+1}) 
 + \partial_p \phi_{j+1} (u_{j+1},\aux_{j+1}) \dot{\aux}_{j+1}.
 \end{equation}
 Injecting~\eqref{def:uj},~\eqref{flow_fj} and~\eqref{HR_l} into~\eqref{lem.aux.2}
 yields:
 \begin{equation} \label{lem.aux.3}
 \partial_p \phi_{j+1} (u_{j+1},\aux_{j+1}) \dot{\aux}_{j+1}
 = f_0\big(\phi_{j+1}(u_{j+1}, \aux_{j+1})\big) 
 + \sum_{l=0}^j u_l^2 \mathcal{K}_{l,j}.
 \end{equation}
 From~\eqref{def.Klj}, one has:
 \begin{equation} \label{lem.aux.4}
 \mathcal{K}_{l,j+1} = \left( \partial_p \phi_{j+1} \big( u_{j+1} , \aux_{j+1} \big) \right)^{-1} \mathcal{K}_{l,j}.
 \end{equation}
 Gathering~\eqref{def:Fj},~\eqref{lem.aux.3} and~\eqref{lem.aux.4} gives:
 \begin{equation} \label{lem.aux.5}
 \dot{\aux}_{j+1} = F_{j+1}(u_{j+1},\aux_{j+1})
 + \sum_{l=0}^j u_l^2 \mathcal{K}_{l,j+1}.
 \end{equation}
 Moreover, by~\eqref{def:Fj},~\eqref{def:Gj} and Lemma~\ref{Lem_Fj}, one has:
 \begin{equation} \label{lem.aux.6}
 \begin{split}
 F_{j+1}(u_{j+1}, \aux_{j+1})
 & = F_{j+1}(0,\aux_{j+1}) + u_{j+1} \partial_\tau F_{j+1} (0,\aux_{j+1})
 + u_{j+1}^2 G_{j+1}(u_{j+1}, \aux_{j+1}) \\
 & = f_0(\aux_{j+1}) + u_{j+1} f_{j+2}(\aux_{j+1})
 + u_{j+1}^2 G_{j+1}(u_{j+1}, \aux_{j+1}).
 \end{split}
 \end{equation}
 Hence,~\eqref{lem.aux.5} and~\eqref{lem.aux.6} conclude the proof of~\eqref{HR_l}
 for $j+1$.
\end{proof}

\subsection{An important notation for estimates}

Despite the difference in the optimal functional framework between control-affine
and nonlinear systems, we are going to prove Theorems~\ref{Theorem:QANL}
and~\ref{Theorem:QAA} in a unified way. To that end, we introduce two parameters
$\gamma$ and $q$:
\begin{itemize}
 \item $\gamma:=1$ and $q:=1$, when $\partial_u^2 f = 0$ on an open neighborhood of $0$ 
 (control-affine systems),
 \item $\gamma:=0$ and $q:=\infty$ otherwise
 (general nonlinear systems).
\end{itemize}
Hence, a trajectory $(x,u)$ belongs to $C^0([0,T],\R^n) \times L^q((0,T),\R)$.
Moreover, the following notation is used throughout the paper as it 
lightens both the statements and the proofs:

\begin{definition} \label{Definition:Ogamma}
 Given two observables $A(x,u)$ and $B(x,u)$ of interest, we will write 
 that $A(x,u) = \Og{B(x,u)}$ when: for any $\tm > 0$, there exists 
 $C, \eta > 0$ such that, for any $T \in (0,\tm]$ and any trajectory
 $(x,u) \in C^0([0,T],\R^n) \times L^q((0,T),\R)$ with $x(0) = 0$
 which satisfies $\norm{u_\gamma}{L^\infty} \leqslant \eta$,
 one has $|A(x,u)| \leqslant C |B(x,u)|$. Hence, this notation refers to
 the convergence $\norm{u_\gamma}{L^{\infty}} \to 0$ and holds
 uniformly with respect to the trajectories on a time interval 
 $[0,T] \subset [0,\tm]$. For observables depending on $t \in [0,T]$,
 it is implicit that the notation $A(x,u,t) = \Og{B(x,u,t)}$ 
 always holds uniformly with respect to $t \in [0,T] \subset [0,\tm]$.
 Eventually, we will use the notation $\Ou{\cdot}$ when the estimate
 holds without any smallness assumption on the control.
\end{definition}

\subsection{Estimations for the auxiliary systems}
\label{subsec:Estimate_Aux_NL}

We start with an estimate of the solutions of the system~\eqref{system.fxu},
requiring little regularity on dynamic
(this will be useful in the sequel).

\begin{lem} \label{Lemma:xu}
 Let $f \in C^1(\R^n \times \R, \R^n)$ with $f(0,0)=0$. One has:
 \begin{equation} \label{Estimate:Lemma:xu}
  |x(t)| = \Oz{\norm{u}{L^1}}.
 \end{equation}
\end{lem}

\begin{proof}
 Let $\tm > 0$.
 We define $M:=\mathrm{max}\{ |f'(p,\tau)| ; \enskip (p,\tau) \in \overline{B}(0,1) \times [-1,1] \}$,
 $C:=Me^{M\tm}$ and $\eta:=\min\left\{ 1, 1 / (2C\tm) \right\}$.
 Let  $T \in (0,\tm]$ and $(x,u) \in C^0([0,T],\R^n) \times L^\infty((0,T),\R)$ be a trajectory 
 of system~\eqref{system.fxu} with $x(0) = 0$ which satisfies $\norm{u}{L^{\infty}} \leqslant \eta$.
 Let $T_1:=\sup\{ t \in (0,T) ; \enskip \forall s \in [0,t]\,, |x(s)| \leqslant 1 \}$.
 For every $t \in [0,T_1]$, one has, by the first-order Taylor expansion:
 \begin{equation}
 |x(t)| 
 = \left| \int_0^t f(x(s),u(s))\ds\right| 
 \leqslant M \int_0^t \left(|x(s)|+|u(s)| \right) \ds.
 \end{equation}
 Thus, by Gr\"onwall's lemma:
 \begin{equation}
 |x(t)| \leqslant M e^{Mt} \int_0^t |u(s)| \ds 
 \leqslant C \|u\|_{L^1} 
 \leqslant C \tm \|u\|_{L^\infty} \leqslant C \tm \eta \leqslant \frac{1}{2}.
 \end{equation}
 This proves that $T_1=T$ and that \eqref{Estimate:Lemma:xu} holds,
 in the sense of Definition~\ref{Definition:Ogamma}.
\end{proof}

In the particular case of control-affine systems~\eqref{system.f0f1}, the size of the solution can be estimated by a weaker norm of the control,
according to the following statement.

\begin{lem} \label{Lemma:x1u1}
 Let $f_0, f_1 \in C^1(\R^n, \R^n)$ with $f_0(0)=0$. One has:
 \begin{align} 
  \label{Estimate:Lemma:x1u1}
  |x(t)| & = \Oo{|u_1(t)| + \norm{u_1}{L^1}}, \\
  \label{Estimate:Aux1}
  |\aux_1(t)| & = \Oo{\norm{u_1}{L^1}}. 
 \end{align}
\end{lem}

\begin{proof}
 Let $\tm > 0$,
 $\phi_1: \Omega_1 \subset \R \times \R^n \rightarrow \R^n$ 
 and $G_1:\Omega_1 \rightarrow \mathbb{R}^n$
 be defined by~\eqref{flow_fj} and~\eqref{def:Gj} for $j=1$.
 Let $r>0$ be such that $[-r,r] \times \overline{B}(0,r) \subset \Omega_1$. 
 Let $M_1>0$ be such that:
 \begin{equation} \label{DEF_M1}
 \forall (p,\tau) \in \overline{B}(0,r) \times [-r,r], \quad
 |f_0'(p)|, \enskip |f_2(p)+\tau G_1(\tau,p)|, \enskip |\phi_1'(\tau,p)| \leqslant M_1.
 \end{equation}
 Let $C_1:=M_1 e^{M_1\tm}$, $\eta:=\min\left( r , r/(2C_1\tm) \right)$ 
 and $C:=M_1 \max\{1,C_1\}$.
 Let $T \in (0,\tm]$ and  $(x,u) \in C^0([0,T],\R^n) \times L^1((0,T),\R)$ be a  trajectory of system~\eqref{system.f0f1} with $x(0) = 0$ 
 which satisfies $\norm{u}{W^{-1,\infty}} \leqslant \eta$. 
 Let $u_1$ be defined by  (\ref{def:uj}), $\aux_1$ be defined by (\ref{def:xj}), which solves~\eqref{HR_l},
 and $T_1:=\sup\{ t \in [0,T] ; \enskip \forall s \in [0,t]\,, |\aux_1(s)| \leqslant r \}$.
 For every $t \in [0,T_1]$, one has, by a first-order Taylor expansion:
 \begin{equation}
 \begin{split}
 |\aux_1(t)| & = \left|\int_0^t \left( f_0(\aux_1(s)) + u_1(s) \Big( f_2(\aux_1(s))+u_1(s) G_1(u_1(s),\aux_1(s)) \Big) \right) \ds \right| \\
 & \leqslant M_1 \int_0^t \left( |\aux_1(s)| + |u_1(s)| \right) \ds.
 \end{split}
 \end{equation}
 Thus, by Gr\"onwall's lemma:
 \begin{equation}
 |\aux_1(t)| \leqslant M_1 e^{M_1 t} \int_0^t |u_1(s)| \ds 
 \leqslant C_1 \|u_1\|_{L^1}
 \leqslant C_1 \tm \|u_1\|_{L^\infty} \leqslant C_1 \tm \eta \leqslant \frac{r}{2}.
 \end{equation}
 This proves that $T_1=T$ and~\eqref{Estimate:Aux1} holds for $t \in [0,T]$.
 Then, for every $t\in [0,T]$, one has:
 \begin{equation}
 \begin{split}
 |x(t)| & = |\phi_1(u_1(t),\aux_1(t))| \\
 & \leqslant M_1 \left( |u_1(t)| + |\aux_1(t)| \right) \\
 & \leqslant M_1 |u_1(t)| + M_1 C_1 \|u_1\|_{L^1} \\
 & \leqslant C\left( |u_1(t)| + \|u_1\|_{L^1} \right),
 \end{split}
 \end{equation}
 which holds then $\norm{u_1}{L^\infty} \leqslant \eta$ and
 thus concludes the proof.
\end{proof}

\begin{lem} \label{Lemma:AuxEst1}
 Let $f \in C^\infty(\R^n \times \R, \R^n)$ with $f(0,0) = 0$. 
 The auxiliary states are well-defined and satisfy:
 \begin{equation} \label{auxj_bound}
  \left| \aux_j(t) \right| = \Og{\norm{u_j}{L^1} + \norm{u_{\gamma}}{L^2}^2}.
 \end{equation}
\end{lem}

\begin{proof}
 The following estimate obtained 
 from~\eqref{def:uj} for $1 \leqslant j \leqslant d$ will be useful:
 \begin{equation} \label{uju}
 \forall p \in [1, +\infty], \quad
 \norm{u_j}{L^p} = \Og{\norm{u_\gamma}{L^p}}.
 \end{equation}
 
 \bigskip \noindent \textbf{Well-posedness of the auxiliary systems.}
 Let us first explain why the auxiliary states are well-defined for $t \in [0,T]$.
 From~\eqref{def:xj}, for $t \in [0,T]$, one has:
 \begin{equation} \label{xjwp.1}
 \aux_j(t) = \phi_j^{-u_j(t)} \circ \cdots \circ \phi_1^{-u_1(t)}(x(t)).
 \end{equation}
 The flows $\phi_j$ for $1 \leqslant j \leqslant d$ are well-defined
 on open neighborhoods $\Omega_j$ of $(0,0)$ in $\R \times \R^n$. 
 Thanks to the continuity of the flows,  
 one deduces from~\eqref{xjwp.1} that the auxiliary states are well-defined
 on $[0,T]$ provided that $|x|$ and the $|u_j|$ stay small enough. 
 Thanks to Lemma~\ref{Lemma:xu} in the nonlinear case, 
 Lemma~\ref{Lemma:x1u1} in the control-affine case and to~\eqref{uju}, this is true if
 $\norm{u_\gamma}{L^{\infty}}$ is small enough. Moreover, the regularity of the 
 flows implies that:
 \begin{equation} \label{xjwp.2}
 \aux_j(t) = \Og{\norm{u_\gamma}{L^1}}.
 \end{equation}
 
 \bigskip \noindent \textbf{Estimates of the auxiliary states.}
 From~\eqref{HR_l}, one has:
 \begin{equation} \label{xjwp.3}
 \aux_j(t) = \int_0^t \left(f_0(\aux_j(s)) + u_j(s) f_{j+1}(\aux_j(s))  + \sum_{l=\gamma}^j u_l(s)^2 \mathcal{K}_{l,j}(s)\right) \ds.
 \end{equation}
 From~\eqref{xjwp.2},~\eqref{xjwp.3}, $f_0(0) = 0$ and the regularity of the functions involved, one has:
 \begin{equation} \label{xjwp.4}
 \aux_j(t) = \int_0^t \Og{|\aux_j(s)|} \ds 
 + \Og{\norm{u_j}{L^1} + \sum_{l=\gamma}^j \norm{u_l}{L^2}^2}.
 \end{equation}
 From~\eqref{uju} and~\eqref{xjwp.4}:
 \begin{equation} \label{xjwp.5}
 \aux_j(t) = \int_0^t \Og{|\aux_j(s)|} \ds 
 + \Og{\norm{u_j}{L^1} + \norm{u_\gamma}{L^2}^2}.
 \end{equation}
 Application of Gr\"onwall's lemma to~\eqref{xjwp.5} proves~\eqref{auxj_bound}.
\end{proof}

\section{Alternative for well-prepared smooth systems} 
\label{Section:well_prepared}

We prove our quadratic alternative for a particular class
of well-prepared smooth systems. More precisely, we consider smooth 
control systems such that:
\begin{equation} \label{hyp.integrator}
\left(f_0'(0)\right)^d f_1(0) = H_0^d b = 0,
\end{equation}
where $d$ is the dimension of $\Sone$. This condition simplifies the linear
dynamic of the system as it is reduced to an iterated integrator. 
We will explain in Subsection~\ref{Subsection:Brunovski} how one can use static
state feedback to transform any system into such a well-prepared system.

\subsection{Enhanced estimates for the last auxiliary system}

Under assumption~\eqref{hyp.integrator}, the linear dynamic has been fully taken
into account once arrived at the auxiliary state $\aux_d$. Indeed, let us 
consider a trajectory with $x(0) = 0$. Hence $\aux_d(0) = 0$ and, using~\eqref{HR_l}, we have:
\begin{equation} \label{eq_x_k_BIS}
 \dot{\aux}_d 
 = f_0(\aux_d) + u_d f_{d+1}(\aux_d) 
 + \sum_{l=\gamma}^{d} u_l^2 \mathcal{K}_{l,d}.
\end{equation}
Recalling~\eqref{def:H0b}, the linearized system of (\ref{eq_x_k_BIS}) 
around the null equilibrium is:
\begin{equation} \label{yd}
\dot{y}_d = H_0 y_d + u_d f_{d+1}(0).
\end{equation}
From~\eqref{hyp.integrator}, $f_{d+1}(0) = H_0^d b = 0$ and $y_d = 0$ 
since $y_d(0) = 0$. There dependence of $\aux_d$ on the control is thus 
at least quadratic.
The second-order approximation of (\ref{eq_x_k_BIS}) around the null
equilibrium is given by:
\begin{equation} \label{zd.evol}
 \dot{z_d}
 = H_0 z_d + \frac{1}{2} f_0''(0).\big(y_d,y_d\big) 
 + u_d f_{d+1}'(0) y_d + \sum_{l=\gamma}^{d} u_l^2  G_l(0,0).
\end{equation}
We deduce from the relation $y_d = 0$, the initial condition $z_d(0) = 0$
and~\eqref{zd.evol} that:
\begin{equation} \label{zd.eq}
 z_d(t) = \sum_{l=\gamma}^{d} \int_0^t u_l^2(s) e^{(t-s)H_0} G_l(0,0) \ds.
\end{equation}
These remarks lead to the following estimates.

\begin{lem} \label{Lemma:Enhanced}
 Let $f \in C^\infty(\R^n \times \R, \R^n)$ with $f(0,0) = 0$ 
 satisfying~\eqref{hyp.integrator}. One has:
 \begin{align}
 \label{auxd_bound}
 \left| \aux_d(t) \right| & = \Og{\norm{u_{\gamma}}{L^2}^2}, \\ 
 \label{auxd_zd_bound}
 \left| \aux_d(t) - z_d(t)  \right| & = \Og{\norm{u_{\gamma}}{L^3}^3}.
 \end{align}
\end{lem}

\begin{proof}
 The integral form of~\eqref{eq_x_k_BIS} is:
 \begin{equation} \label{auxdwp.1}
 \aux_d(t) = \int_0^t \left(f_0(\aux_d(s)) + u_d(s) f_{d+1}(\aux_d(s))
 + \sum_{l=\gamma}^d u_l(s)^2 \mathcal{K}_{l,d}(s)\right) \ds.
 \end{equation}
 Using the integrator assumption~\eqref{hyp.integrator}, one has
 $f_{d+1}(0) = 0$. Moreover $u_d = \Og{1}$.
 Hence, thanks to the
 regularity of the functions involved in~\eqref{auxdwp.1}, one obtains:
 \begin{equation} \label{auxdwp.2}
 \aux_d(t) = \int_0^t \Og{|\aux_d(s)|} \ds 
 + \Og{\norm{u_\gamma}{L^2}^2}.
 \end{equation}
 Estimate~\eqref{auxd_bound} follows from the application of 
 Gr\"onwall's lemma to~\eqref{auxdwp.2}. We turn to the next order
 bound, using the following integral formulation:
 \begin{equation} \label{auxdwp.3}
 \begin{split}
 \aux_d(t) - z_d(t)
 & = \int_0^t \Big( f_0(\aux_d(s)) - H_0 z_d(s) + u_d(s) f_{d+1}(\aux_d(s)) \Big) \ds \\
 & \quad + \sum_{j=\gamma}^d \int_0^t u_j^2(s) \Big(\mathcal{K}_{j,d}(s) - G_j(0, 0) \Big) \ds.
 \end{split}
 \end{equation}
 We estimate separately the different parts of the integrand in~\eqref{auxdwp.3}.
 First, using the regularity of $f_0$ and~\eqref{auxd_bound}, one has:
 \begin{equation} \label{R1}
 \begin{split} 
 f_0(\aux_d) - H_0 z_d 
 & = H_0 (\aux_d - z_d) + \Og{|\aux_d|^2} \\
 & = \Og{|\aux_d - z_d|} + \Og{\norm{u_\gamma}{L^2}^4}.
 \end{split}
 \end{equation}
 Moreover, using once again the integrator assumption~\eqref{hyp.integrator},
 the regularity of $f_{d+1}$, estimate~\eqref{auxd_bound} and H\"older's inequality, 
 one has:
 \begin{equation} \label{R2}
 \begin{split}
 \int_0^t u_d(s) f_{d+1}(\aux_d(s)) \ds
 & = \Og{ \int_0^t |u_d(s)| | \aux_d(s) | \ds } \\
 & = \Og{ \norm{u_\gamma}{L^2}^2 \int_0^t |u_d(s)| \ds }
 =  \Og{ \norm{u_\gamma}{L^3}^3 }.
 \end{split}
 \end{equation}
 For $1 \leqslant l \leqslant d$, the bound~\eqref{auxj_bound} 
 yields $\aux_l = \Og{1}$. The $C^2$ regularity of $\phi_l$ justifies that:
 \begin{equation} \label{dpphil}
  \partial_p \phi_l (u_l(t), \aux_l(t)) 
  = \partial_p \phi_l(0, \aux_l(t)) + \Og{|u_l(t)|}
  =   \mathrm{Id} + \Og{|u_l(t)|}.
 \end{equation}
 Inverting~\eqref{dpphil} proves that, for $0 \leqslant j \leqslant d$:
 \begin{equation} \label{R3}
 \begin{split}
 \mathcal{K}_{j,d} - G_j(0,0) & = \left(\partial_p \phi_d\big(u_d,\aux_d\big)\right)^{-1} \cdots
 \left(\partial_p \phi_{j+1}\big(u_{j+1},\aux_{j+1}\big)\right)^{-1} 
 G_j\big(u_j,\aux_j\big) - G_j(0,0) \\
 & = G_j\big(u_j,\aux_j\big) - G_j(0,0) + \Og{|u_d| + \ldots + |u_{j+1}|} \\
 & = \Og{|\aux_j| + |u_d| + \ldots + |u_{j}|}.
 \end{split}
 \end{equation}
 We deduce from~\eqref{auxj_bound},~\eqref{uju}, (\ref{R1}), (\ref{R2}) and (\ref{R3}) and H\"older's 
 inequality that~\eqref{auxdwp.3} can be written:
 \begin{equation} \label{auxdwp.4_AFF}
 \aux_d(t) - z_d(t) = \Og{\norm{u_\gamma}{L^3}^3} + \int_0^t \Og{|\aux_d - z_d|}. 
 \end{equation}
 Applying Gr\"onwall's lemma to~\eqref{auxdwp.4_AFF} concludes
 the proof of estimate~\eqref{auxd_zd_bound}.
\end{proof}

\subsection{Construction of the invariant manifold} 
\label{subsec:Implicit_function_G}

We construct the smooth manifold $\M$, as the graph of an implicit function $G$. 
The following lemma is a key ingredient in the construction of $G$.

\begin{lem} \label{Lem:F}
 Let $f \in C^\infty(\R^n \times \R, \R^n)$ with $f(0,0) = 0$
 satisfying~\eqref{hyp.integrator}.
 There exists a smooth map:
 \begin{equation}
 F: \left\{
 \begin{aligned}
 \Sone \times \Sone^\perp & \to \Sone^\perp, \\
 (p_\parl, p_\perp) & \mapsto F(p_\parl, p_\perp),
 \end{aligned}
 \right.
 \end{equation}
 such that $F(0,0)=0$, $\partial_\perp F(0,0)=\mathrm{Id}$ on $\Sone^\perp$ and
 one has:
 \begin{equation} \label{F:x/xd}
 F\big( \mathbb{P}x(t), \mathbb{P}^\perp x(t) \big) 
 - \mathbb{P}^\perp \aux_d(t) 
 = \Og{\norm{u_\gamma}{L^3}^3}.
 \end{equation}
\end{lem}

\begin{proof}
 \textbf{Construction of} $F$. 
 We introduce the smooth map $\psi: \R^d \to \Sone$ defined by:
 \begin{equation}
 \label{def.psi}
 \psi(u_1, \ldots ,u_d) := 
 \mathbb{P} \left( \phi_1^{u_1} \circ \cdots \circ \phi_d^{u_d}(0) \right).
 \end{equation}
 Using (\ref{flow_fj}), we obtain by differentiating~\eqref{def.psi} 
 that, for every $h \in \R^d$,
 \begin{equation} \label{psi'}
 \psi'(0) \cdot h = h_1 f_1(0) + \ldots + h_d f_d(0).
 \end{equation}
 From~\eqref{psi'}, $\psi'(0) : \R^d \to \Sone$ is bijective because
 $(f_1(0), \ldots, f_d(0))$ is a basis of $\Sone$.
 By the inverse mapping theorem, $\psi$ is a local 
 $C^\infty$-diffeomorphism around $0$.
 We introduce the smooth map 
 $\alpha: \Omega \subset \Sone \to \R^d$ defined (locally around $0$) by:
 \begin{equation}
 \label{def.alpha}
 \alpha(p_\parl) = \big( \alpha_1(p_\parl) , \ldots , \alpha_d(p_\parl) \big) 
 = \psi^{-1} (p_\parl).
 \end{equation}
 In particular, $\alpha_i(0)=0$ for $1 \leqslant i \leqslant d$.
 We define $F$ as:
 \begin{equation} 
 \label{def.F}
 F(p_\parl, p_\perp) 
 := \mathbb{P}^\perp \Big( \phi_d^{-\alpha_d(p_\parl)} \circ \cdots \circ \phi_1^{-\alpha_1(p_\parl)}(p_\parl + p_\perp) \Big).
 \end{equation}
 Then, $F(0,\cdot)$ is the identity on $\Sone^\perp$ 
 because $\alpha(0)=0$ and $\phi_j^0 = \mathrm{Id}$. 
 One has:
 \begin{equation} \label{eq.dF.perp}
 F(0,0)=0\,, \quad
 \frac{\partial F}{\partial p_\parl}(0,0) = 0
 \quad \text{and} \quad 
 \frac{\partial F}{\partial p_\perp}(0,0) = \mathrm{Id}.
 \end{equation}
 
 \bigskip \noindent \textbf{Proof of estimate~\eqref{F:x/xd}.}
 First, using~\eqref{def:xj} 
 and the $C^1$ regularity of the flows, we have:
 \begin{equation} \label{aj.uj.1}
 \mathbb{P} x - \psi(u_1, \ldots, u_d)
 = \mathbb{P} \phi_1^{u_1} \circ \cdots \circ \phi_d^{u_d}(\aux_d)
 - \mathbb{P} \phi_1^{u_1} \circ \cdots \circ \phi_d^{u_d}(0)
 = \Og{|\aux_d|}.
 \end{equation}
 Hence, plugging estimate~\eqref{auxd_bound} from Lemma~\ref{Lemma:Enhanced}
 into~\eqref{aj.uj.1} and using that $\psi^{-1}$ is locally 
 Lipschitz-continuous proves that:
 \begin{equation} \label{aj.uj}
 \left| \alpha_j (\mathbb{P}x) - u_j \right|
 = \Og{\norm{u_{\gamma}}{L^2}^2}.
 \end{equation}
 Using~\eqref{aj.uj}, we will now prove that:
 \begin{equation} \label{HR_flots_comp}
 \phi_j^{-\alpha_j(\mathbb{P}x)} \circ \cdots \circ 
 \phi_1^{-\alpha_1(\mathbb{P}x)} (x)
 = \aux_j + \sum_{l=1}^j \Big( u_l-\alpha_l(\mathbb{P}x) \Big) f_l (0) 
 + \Og{\norm{u_{\gamma}}{L^3}^3}.
 \end{equation}
 Equation~\eqref{HR_flots_comp} with $j = d$ proves~\eqref{F:x/xd}
 because the projection $\mathbb{P}^\perp$ involved in~\eqref{def.F}
 kills the $d$ terms $f_1(0), \ldots, f_{d}(0)$ from $\Sone$.
 We proceed by induction on $1 \leqslant j \leqslant d$ to prove (\ref{HR_flots_comp}).
 
 \bigskip \noindent \textbf{Initialization for $j = 1$.} 
 By definition of $\aux_1$, we have:
 \begin{equation} \label{eq.p1a1}
 \phi_1^{-\alpha_1(\mathbb{P}x)} (x)
 = \phi_1^{-\alpha_1(\mathbb{P}x)} \circ \phi_1^{u_1} (\aux_1)
 = \phi_1\left( u_1 - \alpha_1(\mathbb{P}x), \aux_1\right).
 \end{equation}
 Then, using a Taylor formula with respect to the time-like variable
 of the flow $\phi_1$, we get from~\eqref{flow_fj},~\eqref{aj.uj}
 and~\eqref{eq.p1a1}:
 \begin{equation}
 \begin{split}
 \phi_1^{-\alpha_1(\mathbb{P}x)} (x)
 & = \aux_1 + \left(u_1 - \alpha_1(\mathbb{P}x)\right) \partial_\tau \phi_1(0, \aux_1)
 + \Og{|u_1 - \alpha_1(\mathbb{P}x)|^2} \\
 & = \aux_1 + \left(u_1 - \alpha_1(\mathbb{P}x)\right) f_1(\aux_1)
 + \Og{\norm{u_{\gamma}}{L^2}^4} \\
 & = \aux_1 + \left(u_1 - \alpha_1(\mathbb{P}x)\right) f_1(0)
 + \Og{\norm{u_{\gamma}}{L^3}^3},
 \end{split}
 \end{equation}
 thanks to the $C^1$ regularity of $f_1$, estimate~\eqref{auxj_bound} and H\"older's inequality.
 
 \bigskip \noindent \textbf{Heredity.} Let $1 \leqslant j < d$ and assume 
 that~\eqref{HR_flots_comp} holds. Applying 
 $\phi_{j+1}^{-\alpha_{j+1}(\mathbb{P}x)}$ to this relation
 and using Taylor's formula gives:
 \begin{equation} \label{her.1}
 \begin{split}
 & \phi_{j+1}^{-\alpha_{j+1}(\mathbb{P}x)} \circ \phi_j^{-\alpha_j(\mathbb{P}x)}
 \circ \cdots \circ \phi_1^{-\alpha_1(\mathbb{P}x)} (x) \\
 & = \phi_{j+1}^{-\alpha_{j+1}(\mathbb{P}x)}
 \left( \aux_j + \sum_{l=1}^j \Big( u_l-\alpha_l(\mathbb{P}x) \Big) f_l (0) 
 + \Og{\norm{u_{\gamma}}{L^3}^3} \right) \\
 & = \phi_{j+1}^{-\alpha_{j+1}(\mathbb{P}x)}(\aux_j)
 + \sum_{j=1}^l (u_l-\alpha_l(\mathbb{P}x)) \partial_p \phi_{j+1}^{-\alpha_{j+1}(\mathbb{P}x)}
 (\aux_j) f_l(0) + \Og{\norm{u_{\gamma}}{L^3}^3} \\
 & = \phi_{j+1}^{u_{j+1} - \alpha_{j+1}(\mathbb{P}x)} (\aux_{j+1})
 + \sum_{j=1}^l (u_l-\alpha_l(\mathbb{P}x)) f_l(0) + \Og{\norm{u_{\gamma}}{L^3}^3},
 \end{split}
 \end{equation}
 because, thanks to~\eqref{aj.uj}, for $1 \leqslant l \leqslant j$:
 \begin{equation}
 \begin{split}
 (u_l-\alpha_l(\mathbb{P}x)) \left(\partial_p \phi_{j+1}^{-\alpha_{j+1}(\mathbb{P}x)} - \mathrm{Id}  \right) f_l(0) 
 & = \Og{\norm{u_{\gamma}}{L^2}^2  \left|\alpha_{j+1}(\mathbb{P}x)\right|} \\
 & = \Og{\norm{u_{\gamma}}{L^2}^2 \left(|u_{j+1}| + \norm{u_{\gamma}}{L^2}^2 \right)} \\
 & = \Og{\norm{u_{\gamma}}{L^3}^3}.
 \end{split}
 \end{equation}
 Moreover, the same arguments as in the initialization lead to:
 \begin{equation} \label{her.2}
 \phi_{j+1}^{u_{j+1} - \alpha_{j+1}(\mathbb{P}x)} (\aux_{j+1})
 = \aux_{j+1} + \left(u_{j+1} - \alpha_{j+1}(\mathbb{P}x)\right) f_{j+1}(0)
 + \Og{\norm{u_{\gamma}}{L^3}^3}.
 \end{equation}
 Thus, we obtain (\ref{HR_flots_comp}) at the next order 
 by incorporating~\eqref{her.2} in~\eqref{her.1}.
\end{proof}

\begin{lem}
 Let $f \in C^\infty(\R^n \times \R, \R^n)$ with $f(0,0) = 0$
 satisfying~\eqref{hyp.integrator}.
 There exists a smooth map $G : \Sone \to \Sone^\perp$ with $G(0) = 0$
 and $G'(0) = 0$ such that:
 \begin{equation} \label{Estimate:G}
  \mathbb{P}^\perp x(t) - G(\mathbb{P}x(t)) - \mathbb{P}^\perp \aux_d(t)
  = \Og{\norm{u_\gamma}{L^3}^3}.
 \end{equation}
\end{lem}

\begin{proof}
 We consider the smooth map:
 \begin{equation} \label{def:FFtilde}
 \tilde{F} : \left\{
 \begin{aligned}
 \Sone \times \Sone^\perp \times \Sone^\perp & \to \Sone^\perp, \\
 (p_\parl, p_\perp, \rho) & \mapsto F(p_\parl, p_\perp) - \rho.
 \end{aligned}
 \right.
 \end{equation}
 One checks from~\eqref{eq.dF.perp}  that $\tilde{F}(0,0,0) = 0$ 
 and $\partial_\perp \tilde{F}(0,0,0) = \mathrm{Id}$.
 By the implicit function theorem, there exists an open neighborhood $U$ of
 $(0,0)$ in $\Sone\times\Sone^\perp$, an open neighborhood of $V$ of $0$
 in $\Sone^\perp$ and a smooth map $\Theta : U \to V$ such that,
 for any $(p_\parl, \rho) \in U$:
 \begin{equation}
 \Big( p_\perp \in V \text{ and } 
 \tilde{F}(p_\parl, p_\perp, \rho) = 0 
 \Big)
 \Leftrightarrow
 \Big(
 p_\perp = \Theta(p_\parl, \rho)
 \Big).
 \end{equation}
 In particular, $\Theta(0,0) = 0$. Moreover, by differentiating the relation:
 \begin{equation} \label{eq.FtildeTheta}
 \tilde{F}(p_\parl,\Theta(p_\parl, \rho), \rho) = 0,
 \end{equation}
 we obtain thanks to~\eqref{eq.dF.perp}:
 \begin{equation} \label{diff_d=Id}
 \frac{\partial \Theta}{\partial p_\parl}(0,0) = 0
 \quad \text{and} \quad
 \frac{\partial \Theta}{\partial \rho}(0,0) = - \mathrm{Id}.
 \end{equation}
 We define $G:\Sone \rightarrow \Sone^\perp$ by:
 \begin{equation} \label{def.G}
 G(p_\parl) := \Theta(p_\parl,0).
 \end{equation}
 One checks that $G(0) = 0$ and $G'(0) = 0$ from~\eqref{diff_d=Id}.
 Moreover, using (\ref{F:x/xd}), the $C^1$ regularity of $\Theta$,
 \eqref{auxd_bound} and (\ref{diff_d=Id}) we get:
 \begin{equation} \label{rel_G}
 \begin{split}
 \mathbb{P}^\perp x(t)
 & =  \Theta\Big( \mathbb{P} x(t) , \mathbb{P}^\perp \aux_d(t) + 
 \Og{\norm{u_\gamma}{L^3}^3} \Big) \\
 & = \Theta\Big( \mathbb{P}x(t) , 0 \Big) + \mathbb{P}^\perp \aux_d(t) + 
 \Og{\norm{u_\gamma}{L^3}^3} \\
 & =  G\big(\mathbb{P}x(t)\big) + \mathbb{P}^\perp \aux_d(t)  +
 \Og{\norm{u_\gamma}{L^3}^3}.
 \end{split}
 \end{equation}
 Equation~\eqref{rel_G} concludes the proof of~\eqref{Estimate:G}.
\end{proof}

\begin{rk}
 The functions $F$ and $G$ constructed in this subsection are only defined in small 
 neighborhoods of the origin. However, since we are always considering controls 
 which are small in $W^{-\gamma,\infty}$ and since Lemmas~\ref{Lemma:xu} and~\ref{Lemma:x1u1} imply that,
 for such controls, trajectories stay in a small neighborhood of the origin,
 we only use the constructed functions where they are well-defined. In fact,
 one can also choose any smooth extension of these functions to the whole
 space. This procedure simplifies the statement of our theorems, avoiding the
 need to mention this detail. 
\end{rk}

\subsection{Invariant manifold case}

When $\Stwo + \mathbb{R} d_0 = \Sone$, then, by (\ref{zd.eq}),  $z_d(t) \in \Sone$ for any $t\in[0,T]$ because $G_j(0,0)$ belongs to $\Sone$ 
(for $j = 0$, see~\eqref{DEF_d0} and~\eqref{DEF_G0}, for $1 \leqslant j \leqslant d$, see~\eqref{def:Gj} and~\eqref{Fj_2})
and this space is stable by $H_0$.  Thus $\mathbb{P}^\perp z_d(t) = 0$.
Taking this fact into account and combining~\eqref{auxd_zd_bound} 
from Lemma~\ref{Lemma:Enhanced} with~\eqref{Estimate:G} gives:
\begin{equation} \label{xd_dand_S1}
\mathbb{P}^\perp x - G(\mathbb{P}x) = \Og{\norm{u_\gamma}{L^3}^3}.
\end{equation}
Estimate~\eqref{xd_dand_S1} and the meaning of the notation $\Og{\cdot}$ yield
the existence of positive constants $C$ and $\eta$ such that the 
conclusions~\eqref{Estimate:TQANL.1} and~\eqref{CCL:TQA:1} 
of Theorems~\ref{Theorem:QANL} and~\ref{Theorem:QAA} 
hold for controls that are smaller than $\eta$ in $W^{-\gamma,\infty}$-norm,
in the particular case of well-prepared systems satisfying~\eqref{hyp.integrator}.

\subsection{Quadratic drift case}

We consider the case $\Stwo + \mathbb{R} d_0  \not\subset \Sone$ and prove that
the state drifts towards the direction $d_k$ defined by~\eqref{DEF_d0} for $k=0$
and~\eqref{def.dk} for $k \geqslant 1$.

\subsubsection{Coercivity of the quadratic drift for small-times}
\label{Subsubsection:coercivity}

We know from Lemma~\ref{Lemma:ComingOut} that there exists 
$0 \leqslant k \leqslant d$ such that $G_j(0,0) \in \Sone$ for 
$0\leqslant j < k$ and $G_k(0,0) \notin \Sone$. Indeed, thanks 
to~\eqref{DEF_G0},~\eqref{def:Gj} and~\eqref{Fj_2}, 
one has $G_k(0,0) = \frac{1}{2} d_k$.
The heuristic is then that:
\begin{equation} \label{pp.zd.dk}
\mathbb{P}^\perp z_d(t) 
\approx
\frac{1}{2} \left(\int_0^t u_k^2(s) \ds\right) 
d_k.
\end{equation}
To make this statement more precise, we define, for $t \geqslant 0$, 
$Q_t(u) := \langle z_d(t), d_k \rangle$ and we introduce the following set:
\begin{equation} \label{def:TT}
 \mathbb{T} := \left\{ T > 0; \enskip  \exists C_T > 0\,, \forall t \in (0,T]\,, \forall v \in L^2(0,t), Q_t(v) \geqslant C_T \int_0^t v_k(s)^2 \ds \right\}.
\end{equation}

\begin{lem} \label{Lemma:Coercive}
 The set $\mathbb{T}$ is non empty.
\end{lem}

\begin{proof}
 From~\eqref{zd.eq}, using that $\langle G_j(0,0), d_k \rangle = 0$ for $j < k$, 
 we compute:
 \begin{equation} \label{heur.est0}
  Q_t(u) = \sum_{j=k}^{d} \int_0^t u_j(s)^2 \langle 
 e^{(t-s)H_0} G_j(0,0), d_k \rangle \ds,
 \end{equation}
 There exists $C>0$ such that, for $T$ small enough
 and $t \in [0,T]$:
 \begin{equation} \label{heur.est1}
 \left| \int_0^t u_k(s)^2 \left( e^{(t-s)H_0} - \mathrm{Id} \right) G_k(0,0) \ds \right| 
 \leqslant C t \int_0^t u_k(s)^2 \ds
 \end{equation}
 and, for every $k+1 \leqslant j \leqslant d$,
 \begin{equation} \label{heur.est2}
 \left| \int_0^t u_j(s)^2 e^{(t-s)H_0} G_j(0,0) \ds \right| 
 \leqslant C t^2\int_0^t u_k(s)^2 \ds,
 \end{equation}
 because:
 \begin{equation}
 |u_j(s)| \leqslant s^{j-k-1} \int_0^t |u_k(\tau)| \mathrm{d}\tau 
 \leqslant s^{j-k-\frac{1}{2}} \left(\int_0^s |u_k(\tau)|^2 \mathrm{d}\tau\right)^{\frac{1}{2}}.
 \end{equation}
 Gathering~\eqref{heur.est0},~\eqref{heur.est1} and~\eqref{heur.est2}
 one has $Q_t(u) \geqslant \frac{1}{4} \norm{u_k}{L^2(0,t)}$
 for $T$ small enough.
\end{proof}

Thanks to Lemma~\ref{Lemma:Coercive}, we can define the coercivity time 
$T^*$ as:
\begin{equation} \label{def:T*}
 T^* := \sup \left\{ T > 0; \enskip T \in \mathbb{T} \right\}.
\end{equation}
In the general case $T^* < + \infty$ (see Example~\ref{Example:Competition}
where $T^* = \pi$).
However, in some particular easy cases, it is also possible that $T^* = + \infty$ (see  Example~\ref{Example:easy_drift}). In such cases, Theorems~\ref{Theorem:QANL}
and~\ref{Theorem:QAA} actually lead to the conclusion that the associated
systems are not even \emph{$W^{2k-3\gamma}$ large time locally controllable}
in the sense that, for any $T > 0$, there exists $\eta > 0$ such that,
for any $\delta > 0$, there exists $x^\dagger \in \R^n$ with 
$|x^\dagger|\leqslant \delta$ such that there is no trajectory from $x(0) = 0$
to $x(T) = x^\dagger$ with a control such that $\norm{u}{W^{2k-3\gamma}} \leqslant \eta$.

\subsubsection{Absorption of cubic residuals by interpolation}
\label{Subsubsection:Gagliardo}

Let $\tm \in (0,T^*)$. From~\eqref{def:TT} and~\eqref{def:T*}, there exists 
$C_\tm > 0$ such that, for any $T \in (0,\tm]$, 
$v \in L^2(0,T)$, and any $t \in [0,T]$:
\begin{equation} \label{qtv}
Q_t(v) \geqslant C_\tm \int_0^t v(s)^2 \ds.
\end{equation}
Thanks to~\eqref{auxd_zd_bound},~\eqref{Estimate:G} and~\eqref{qtv}, 
one has:
\begin{equation} \label{eq.drift.23}
\begin{split}
\left\langle \mathbb{P}^\perp x(t) - G \left( \mathbb{P} x(t) \right) , d_k \right\rangle
& = \left\langle \mathbb{P}^\perp \aux_d(t) , d_k \right\rangle 
+ \Og{\norm{u_\gamma}{L^3(0,t)}^3} \\
& = \left\langle z_d(t)  , d_k \right\rangle + 
\Og{\norm{u_\gamma}{L^3(0,t)}^3}  \\
& \geqslant C_\tm \norm{u_k}{L^2(0,t)}^2 + \Og{\norm{u_\gamma}{L^3(0,t)}^3}.
\end{split}
\end{equation}
Therefore, there exists $\eta_0 > 0$ and $M > 0$ such that,
if $\norm{u_\gamma}{L^\infty} \leqslant \eta_0$, one has:
\begin{equation}\label{eq.drift.23.bis}
\left\langle \mathbb{P}^\perp x(t)  - G \left( \mathbb{P} x(t) \right) , d_k \right\rangle 
\geqslant C_\tm \norm{u_k}{L^2(0,t)}^2 - M \norm{u_{\gamma}}{L^3(0,t)}^3.
\end{equation}

\begin{prop}
 \label{Prop:GNII}
 Let $\kappa \geqslant 1$. There exists constants $C_1, C_2 > 0$ 
 (depending on $\kappa$), such that, for any $L > 0$ and any 
 $\psi \in W^{3\kappa-3,\infty}((0,L),\R)$:
 \begin{equation} \label{est.gnii}
 \norm{\psi^{(\kappa-1)}}{L^3}^3 \leqslant
 C_1 \norm{\psi}{L^2}^2 \norm{\psi^{(3\kappa-3)}}{L^\infty} +
 C_2 L^{\frac{5}{2}-3\kappa} \norm{\psi}{L^2}^3.
 \end{equation} 
 Moreover, for any $L > 0$ and any 
 $\psi \in W^{3\kappa-3}_0((0,L),\R)$:
 \begin{equation} \label{est.gnii_0}
 \norm{\psi^{(\kappa-1)}}{L^3}^3 \leqslant
 (C_1+C_2) \norm{\psi}{L^2}^2 \norm{\psi^{(3\kappa-3)}}{L^\infty}.
 \end{equation} 
\end{prop}

\begin{proof}
 This is a particular case of the Gagliardo-Nirenberg interpolation inequality.
 Indeed, for any $\kappa \geqslant 1$:
 \begin{equation} \label{eq.gnii}
 \frac{1}{3} = \frac{\kappa-1}{1} + 
 \frac{1}{3} \cdot \left(\frac{1}{\infty} - \frac{3\kappa-3}{1}\right)
 + \left(1-\frac{1}{3}\right)\frac{1}{2}.
 \end{equation}
 From~\eqref{eq.gnii}, we can apply~\cite[Theorem p.125]{MR0109940} for
 functions defined on $[0,1]$. The generalization to functions defined on
 $[0,L]$ uses a straightforward scaling argument which gives the power
 of $L$ in front of $C_2$ in (\ref{est.gnii}).
 Then, for $\psi \in W^{3\kappa-3,\infty}_0((0,L),\R)$, 
 using the Cauchy-Schwarz inequality, iterated integration 
 and the conditions $\psi^{(j)}(0)=0$ for $j=0,...,3\kappa-4$ we obtain:
 \begin{equation}
  L^{\frac{5}{2}-3\kappa} \norm{\psi}{L^2} 
  \leqslant L^{3-3\kappa} \norm{\psi}{L^\infty}
  \leqslant \norm{\psi^{(3\kappa-3)}}{L^\infty},
 \end{equation}
 which proves (\ref{est.gnii_0}).
\end{proof}

We apply Proposition~\ref{Prop:GNII} to $\psi := u_k$  with $\kappa=k-\gamma+1$.

\bigskip

\textbf{First case}: $u \in W^{2k-3\gamma,\infty}_0(0,T)$.
Let $t \in [0,T]$.
In this case, $\psi = u_k \in W^{3k-3\gamma}_0(0,t)$ 
because $u \in W^{2k-3\gamma}_0(0,t)$ and $u_k^{(j)} (0) = u_{k-j}(0)=0$ for $j=0,...,k-1$.
Thanks to~\eqref{est.gnii_0} for $L = t$ and to the equality $u_k^{(k)} = u$, one has:
\begin{equation}
 \begin{split}
 \norm{u_\gamma}{L^3(0,t)}^3
 = \norm{u_k^{(k-\gamma)}}{L^3(0,t)}^3
 & \leqslant (C_1+C_2) \|u_k\|_{L^2(0,t)}^2 \|u_k^{(3k-3\gamma)}\|_{L^\infty(0,t)} \\
 & \leqslant (C_1+C_2) \|u_k\|_{L^2(0,t)}^2 \|u^{2k-3 \gamma} \|_{L^\infty(0,t)}.
 \end{split}
\end{equation}
Then, we deduce from (\ref{eq.drift.23.bis}) that:
\begin{equation}
\left\langle \mathbb{P}^\perp x(t)  - G \left( \mathbb{P} x(t) \right) , d_k \right\rangle
\geqslant \|u_k\|_{L^2(0,t)}^2 \Big( C_{\tm} - M (C_1+C_2) \| u^{(2k-3\gamma)} \|_{L^\infty(0,t)}  \Big)
\end{equation}
In particular, when $u$ satisfies:
\begin{equation}
 \norm{u}{W^{2k-3\gamma,\infty}} \leqslant \eta(\tm) := \min \left( \eta_0, \frac{C_{\tm}}{2M(C_1+C_2)} \right),
\end{equation}
then:
\begin{equation} \label{qd_ccl}
\left\langle \mathbb{P}^\perp x(t)  - G \left( \mathbb{P} x(t) \right) , d_k \right\rangle
\geqslant \frac{1}{2} C_{\tm} \|u_k\|_{L^2(0,t)}^2 \geqslant 0\,.
\end{equation}
This gives the conclusions (\ref{Estimate:TQANL.2}) and (\ref{Estimate:TQANL.3}) of Theorems~\ref{Theorem:QANL} and \ref{Theorem:QAA}.
Moreover, we recover the inequality announced in Subsection~\ref{Subsection:Comments} (the drift is
quantified by the $H^{-k}$-norm of $u$).
\bigskip

\textbf{Second case}: $u \in W^{2k-3\gamma,\infty}(0,T)$.
 When the control $u$ only belong to $W^{2k-3\gamma,\infty}(0,T)$ 
 instead of $W^{2k-3\gamma,\infty}_0(0,T)$,
 inequality (\ref{est.gnii}) with $L=T$ yields:
 \begin{equation}
 \begin{split}
 \norm{u_\gamma}{L^3}^3
 & \leqslant \|u_k\|_{L^2}^2 \Big( C_1 \norm{u^{(2k-3\gamma)}}{L^\infty} + C_2 T^{\frac{5}{2}-3(k-\gamma+1)} \|u_k\|_{L^2} \Big) \\
 & \leqslant \|u_k\|_{L^2}^2 \Big( C_1 \norm{u^{(2k-3\gamma)}}{L^\infty} + C_2 T^{3-3(k-\gamma+1)} \|u_k\|_{L^\infty} \Big) \\
 & \leqslant \|u_k\|_{L^2}^2 \Big( C_1 \norm{u^{(2k-3\gamma)}}{L^\infty} + C_2 T^{-2k+2\gamma} \|u_\gamma\|_{L^\infty} \Big) \\
 & \leqslant \|u_k\|_{L^2}^2 \Big( C_1  + C_2 T^{-2k+2\gamma}  \Big) \|u \|_{W^{2k-3\gamma,\infty}}\,.
 \end{split}
 \end{equation}
 In particular, when $u$ satisfies:
 \begin{equation}
  \norm{u}{W^{2k-3\gamma,\infty}} \leqslant \eta(T) 
  := \min \left( \eta_0, \frac{C_{\tm}}{2M(C_1+C_2 T^{-2k+2\gamma})} \right),
 \end{equation}
 then (\ref{qd_ccl}) holds at the final time $t = T$. 
 In particular, this proves that system~\eqref{system.fxu} is not \stlc{$W^{2k-3\gamma,\infty}$}.

\section{Reduction to well-prepared smooth systems} 
\label{Section:reduction}

We prove Theorems~\ref{Theorem:QANL} and~\ref{Theorem:QAA} for general systems 
by reduction to the case of well-prepared smooth systems considered in
Section~\ref{Section:well_prepared}, by means of a linear static state feedback 
transformation. We start with smooth systems then extend our results to less 
regular dynamics.

\subsection{Linear static state feedback and Lie brackets}

We consider linear static state feedback transformations of
the control. We prove that the structural properties of the Lie spaces $\Sone$ 
and $\Stwo$ that are involved in our theorems are invariant under such 
transformations. 

\bigskip

Let $f \in C^\infty(\R^n\times\R,\R^n)$ and $\beta \in \R^n$. We consider the 
transformed control $v := u - (\beta, x)$. The state now evolves
under the equation $\dot{x} = g(x,v)$ where we define:
\begin{equation} \label{def.fg}
 g(x,v) := f(x, v + (\beta, x)).
\end{equation}
As in~\eqref{def.f0f1}, we use the notations $f_0(x) := f(x,0)$,
$g_0(x) := g(x,0)$, $f_1(x) := \partial_u f(x,0)$, $g_1(x) := \partial_v g(x,0)$
and $d_0 := \partial_{uu} f(0,0) = \partial_{vv} g(0,0)$. Using~\eqref{def.fg}, 
a Taylor expansion with respect to the control and Lemma~\ref{Lemma:Taylor}, 
one has:
\begin{align}
 \label{eq.g0}
 g_0(x) & \eqmod{3} f_0(x) + (\beta, x) f_1(x) + \frac{1}{2} (\beta, x)^2 d_0, \\
 \label{eq.g1}
 g_1(x) & \eqmod{2} f_1(x) + (\beta, x) d_0.
\end{align}
The linear controllable spaces associated with the dynamics 
$f$ and $g$ are the same. We will denote them by $\Sone$. Indeed, 
$g_1(0) = f_1(0)$ and $g_0'(0)=f_0'(0) + (\beta, \cdot) f_1(0)$. Thus, the 
iterated powers of $g_0'(0)$ applied to $g_1(0)$ span the same space as the 
iterated powers of $f_0'(0)$ applied to $f_1(0)$. We prove that second-order
Lie brackets of $g_0$ and $g_1$ inherit the following behavior.

\begin{lem} \label{Lemma:Lie_Bruno}
 Let $k \in \N$. Assume that $d_0 \in \Sone$ and that the original system satisfies:
 \begin{equation} \label{equiv.hyp}
 \left[f_1, \ad^{j}_{f_0}(f_1)\right](0) \in \Sone,
 \quad \text{for any } 1 \leqslant j \leqslant k.
 \end{equation}
 Then, the transformed system satisfies:
 \begin{equation} \label{equiv.csq}
 \left[g_1, \ad^{j}_{g_0}(g_1)\right](0) \in \Sone,
 \quad \text{for any } 1 \leqslant j \leqslant k.
 \end{equation}
 Moreover, one has the following equality modulo terms in $\Sone$:
 \begin{equation} \label{equiv.equiv}
 \left[g_1, \ad^{k+1}_{g_0}(g_1)\right](0) \sim
 \left[f_1, \ad^{k+1}_{f_0}(f_1)\right](0).
 \end{equation}
\end{lem}

\begin{proof}
 \textbf{We start with first-order Lie brackets.} 
 For any $0 \leqslant m \leqslant k + 1$, we prove that there exists
 $m$ smooth scalar functions $\lambda^m_0, \ldots, \lambda^m_{m-1} : \R^n \to \R$
 and a smooth function $\mu_m : \R^n \to \Sone$ with $\mu_m(0) = 0$ such that:
 \begin{equation} \label{adm.tilde}
 \ad^m_{g_0}(g_1)(x)
 \eqmod{2}
 \ad^m_{f_0}(f_1)(x) 
 + \mu_m(x)
 + \sum_{j=0}^{m-1} \lambda^m_j(x) \ad^j_{f_0}(f_1)(x).
 \end{equation}
 We proceed by induction on $m$. For $m = 0$, the sum in~\eqref{adm.tilde} is 
 empty by convention. Thus,~\eqref{adm.tilde} holds for $m=0$ with 
 $\mu_0(x) := (\beta,x) d_0$ thanks to~\eqref{eq.g1}. Let $0 \leqslant m \leqslant k$, we 
 compute the next Lie bracket using~\eqref{eq.g0},~\eqref{adm.tilde} 
 and Lemma~\ref{Lemma:eqmod123}:
 \begin{equation} \label{adm.tilde1}
 \begin{split}
 \left[g_0, \ad^m_{g_0}(g_1)\right] 
 \eqmod{2} & 
 \frac{1}{2} \alp{\cdot}^2 \left(\ad^m_{g_0}(g_1)\right)' d_0
 + \left(\ad^m_{f_0}(f_1)\right)' \left(f_0 + \alp{\cdot}f_1 \right) 
 + \mu_m' g_0 \\
 & + \sum_{j=0}^{m-1} \left({\left(\lambda_j^m\right)}', g_0 \right) \ad^j_{f_0}(f_1) 
 + \lambda_j^m \left(\ad^j_{f_0}(f_1)\right)' \left( f_0 + \alp{\cdot}f_1\right) \\
 & - f_0' \left( \ad^m_{f_0}(f_1) + \sum_{j=0}^{m-1} \lambda^m_j \ad^j_{f_0}(f_1) 
 + \mu_m \right)
 - \alp{\ad^m_{g_0}(g_1)} f_1 \\
 & - \alp{\cdot} f_1' \left( \ad^m_{f_0}(f_1) + \sum_{j=0}^{m-1} \lambda^m_j \ad^j_{f_0}(f_1) 
 + \mu_m \right) \\
 & - \alp{\cdot} \alp{\ad^m_{g_0}(g_1)} d_0.
 \end{split}
 \end{equation}
 Using $\alp{\cdot} f_1' \mu_m \eqmod{2} 0$ and
 $\alp{\cdot}^2 \left(\ad^m_{g_0}(g_1)\right)' d_0 \eqmod{2} 0$,
 reordering the terms in~\eqref{adm.tilde1} yields:
 \begin{equation} \label{adm.tilde2}
 \begin{split}
 \left[g_0, \ad^m_{g_0}(g_1)\right] 
 \eqmod{2} & \enskip \ad^{m+1}_{f_0}(f_1) 
 + \mu_m' g_0 - f_0' \mu_m - \alp{\cdot} \alp{\ad^m_{g_0}(g_1)} d_0
 \\
 &
 + \alp{\cdot} \left( [f_1, \ad^m_{f_0}(f_1)]  + \sum_{j=0}^{m-1} \lambda^m_j 
 [f_1, \ad^j_{f_0}(f_1)] \right) \\
 & - \alp{\ad^m_{g_0}(g_1)} f_1
 + \sum_{j=0}^{m-1} \left({\left(\lambda_j^m\right)}', \tilde{f_0} \right) \ad^j_{f_0}(f_1)  
 + \lambda^m_j \ad^{j+1}_{f_0}(f_1).
 \end{split}
 \end{equation}
 Since $m \leqslant k$, hypothesis~\eqref{equiv.hyp} yields the existence of
 a constant $\gamma_m \in \Sone$ such that:
 \begin{equation} \label{def.gammam}
 [f_1, \ad^m_{f_0}(f_1)]  + \sum_{j=0}^{m-1} \lambda^m_j 
 [f_1, \ad^j_{f_0}(f_1)] \eqmod{1} \gamma_m.
 \end{equation}
 From Lemma~\ref{Lemma:Taylor}, $f_0' \eqmod{1} H_0$. Thus,
 $f_0' \mu_m \eqmod{2} H_0 \mu_m$. Using these remarks, we define:
 \begin{equation} \label{def.mum}
 \mu_{m+1} := \mu_m' g_0 - H_0 \mu_m + \alp{\cdot} \gamma_m
 - \alp{\cdot} \alp{\ad^m_{g_0}(g_1)} d_0.
 \end{equation}
 Since $\Sone$ is stable under multiplication by $H_0$ and the image of 
 $\mu_m$ is included in $\Sone$, so is the image of $\mu_{m+1}$
 and $\mu_{m+1}(0) = 0$. Eventually, we define:
 \begin{align}
 \label{def.lambdam0}
 \lambda^{m+1}_0 & := - \alp{\ad^m_{g_0}(g_1)} 
 + \left({\left(\lambda_0^m\right)}', \tilde{f_0} \right), \\
 \label{def.lambdamj}
 \lambda^{m+1}_j & := \lambda^m_{j-1}
 + \left({\left(\lambda_j^m\right)}', \tilde{f_0} \right),
 \quad \text{for} \quad 1 \leqslant j < m, \\
 \label{def.lambdamm}
 \lambda^{m+1}_m & := \lambda^m_{m-1}.
 \end{align}
 In the particular case $m = 0$, we only use~\eqref{def.lambdam0} with the 
 convention that $\lambda^0_0 = 0$. For $m = 1$, we use~\eqref{def.lambdam0} 
 and~\eqref{def.lambdamm}. For larger $m$, we use all three formulas
 including~\eqref{def.lambdamj}. Plugging into~\eqref{adm.tilde2} the 
 definitions~\eqref{def.gammam} and~\eqref{def.mum} proves~\eqref{adm.tilde}
 at order $m+1$. 
 \bigskip
 
 \textbf{We move on to second-order Lie brackets.}
 Let $1 \leqslant m \leqslant k +1$. Using Lemma~\ref{Lemma:eqmod22} 
 and formula~\eqref{adm.tilde}, we compute:
 \begin{equation} \label{ad1m.tilde1}
 \begin{split}
 \left[g_1, \ad^m_{g_0}(g_1)\right](0)
 & = \left(\ad^m_{f_0}(f_1)\right)'(0) f_1(0) + \mu_m'(0) f_1(0) \\
 & + \sum_{j=0}^{m-1} \lambda_j^m(0) \left(\ad^j_{f_0}(f_1)\right)'(0) f_1(0)
 + ((\lambda_j^m)'(0), f_1(0)) \ad^j_{f_0}(f_1)(0) \\
 & - f_1'(0) \ad^m_{f_0}(f_1) - f_1'(0) \mu_m(0)
 - \sum_{j=0}^{m-1} \lambda^m_j(0) f_1'(0) \ad^j_{f_0}(f_1)(0).
 \end{split}
 \end{equation}
 We have $\mu_m(0) = 0$. For $j \in \N$, the vectors
 $\ad^j_{f_0}(f_1)(0)$ belong to $\Sone$ by definition of $\Sone$.
 Moreover, since the image of $\mu_m$ is contained in $\Sone$, 
 $\mu_m'(0) f_1(0) \in \Sone$. Hence, from~\eqref{ad1m.tilde1}:
 \begin{equation} \label{ad1m.tilde2}
 \left[g_1, \ad^m_{g_0}(g_1)\right](0)
 \sim \left[f_1, \ad^m_{f_0}(f_1)\right](0)
 + \sum_{j=0}^{m-1} \lambda_j^m(0) \left[f_1, \ad^j_{f_0}(f_1)\right](0).
 \end{equation}
 Using~\eqref{equiv.hyp} and~\eqref{ad1m.tilde2} 
 proves~\eqref{equiv.csq} for $1 \leqslant m \leqslant k$
 and~\eqref{equiv.equiv} for $m = k+1$.
\end{proof}

We proved Lemma~\ref{Lemma:Lie_Bruno} for nonlinear systems.
It also holds in the particular case of control-affine systems. For such systems, 
approximate equalities~\eqref{eq.g0} and~\eqref{eq.g1} become equalities which 
hold with $d_0 = 0 \in \Sone$.

\subsection{Generalization of the proof using a Brunovsk\'y transformation}
\label{Subsection:Brunovski}

We prove Theorem~\ref{Theorem:QANL} and 
Theorem~\ref{Theorem:QAA} for general systems by reduction to the well-prepared 
class studied in Section~\ref{Section:well_prepared}. The main argument is a 
linear transformation first proposed by Brunovsk\'y in~\cite{MR0284247} 
(see~\cite[Theorem 2.2.7]{MR2224013} for a modern proof). 

\subsubsection{A linear transformation}
\label{Paragraph:Brunovsky}

We consider the linearized system~\eqref{system.yh}. 
Denoting by $d$ the dimension of $\Sone$, there 
exists a matrix $R \in GL_n(\R)$ such that $Rb = e_1$ and:
\begin{equation} \label{bruno}
R H_0 R^{-1} = \begin{pmatrix}
\Lambda_d & * \\
0 & *
\end{pmatrix},
\quad \text{where} \quad
\Lambda_d := \begin{pmatrix}
-\alpha_1 & & \ldots & & - \alpha_d \\
1 & 0 & \ldots & 0 & 0 \\
0 & 1 & & & 0 \\
\vdots & & \ddots & & \vdots \\
0 & & & 1 & 0
\end{pmatrix}
\end{equation}
and the $\alpha_i$ are the coefficients of the characteristic polynomial
of the controllable part of~$H_0$: 
$\chi(X) := X^d + \alpha_1 X^{d-1} + \ldots + \alpha_d$. We introduce 
$\alpha$ a column vector with $n$ components whose first $d$ components
are the $\alpha_i$ and whose last $n-d$ components are null.
Let us denote by $\beta := \tr{R} \alpha$
and $v := u - (\beta, {y})$. Hence:
\begin{equation} \label{system.ytilde}
 \dot{y} = \mathcal{H}_0 y + v b,
\end{equation}
with $\mathcal{H}_0 := H_0 + R^{-1} e_1 \tr{\alpha} R$.  
By construction of~\eqref{bruno}, one has $\mathcal{H}_0^d b = 0$, which 
corresponds to the well-prepared nilpotent integrator form that we studied in 
Section~\ref{Section:well_prepared}.

\subsubsection{Generalization of the proof for nonlinear systems}

We start with the following lemma which proves that the analytic notions
involved in Theorem~\ref{Theorem:QANL} are invariant under 
linear static state feedback transformations.

\begin{lem} \label{Lemma:Norms}
 Let $f \in C^\infty(\R^n\times\R,\R^n)$ with $f(0,0) = 0$, $\beta \in \R^n$ 
 and $\tm > 0$.
 There exists $C, \eta > 0$ and a family of constants $C_m,\eta_m > 0$ 
 for $m \in \N$ such that, 
 for any $T \in (0,\tm)$ and any trajectory
 $(x,u) \in C^0([0,T],\R^n) \times L^\infty((0,T),\R)$ of system~\eqref{system.fxu}
 with $x(0) = 0$, letting $v := u - \alp{x}$, one has:
 \begin{align}
  \label{eq.Lemma:Norms.1}
  \norm{u}{L^\infty} \leqslant \eta & \quad \Rightarrow \quad 
  \norm{v}{L^3} \leqslant C \norm{u}{L^3}, \\
  \label{eq.Lemma:Norms.2}
  \norm{u}{W^{m,\infty}} \leqslant \eta_m & \quad \Rightarrow \quad 
  \norm{v}{W^{m,\infty}} \leqslant C_m \norm{u}{W^{m,\infty}}.
 \end{align}
 Moreover, when $u(0) = \ldots = u^{(m-1)}(0) = 0$, 
 then $v(0) = \ldots = v^{(m-1)}(0) = 0$.
\end{lem} 

\begin{proof}
 \textbf{First}, from Lemma~\ref{Lemma:xu}, 
 one has $x = \Oz{\norm{u}{L^1}}$. Thus,
 from H\"older's inequality:
 \begin{equation} \label{v.L3}
  \norm{v}{L^3} 
  = \Oz{\norm{u}{L^3}+\norm{u}{L^1}}
  = \Oz{\norm{u}{L^3}}.
 \end{equation}
 From Definition~\ref{Definition:Ogamma}, 
 estimate~\eqref{v.L3} proves~\eqref{eq.Lemma:Norms.1}. 
 \bigskip
 
 \noindent \textbf{Second}, we prove by induction on $m \geqslant 0$ that there 
 exists a family of smooth functions $P_m : \R^m \times \R^n \to \R^n$, 
 vanishing at zero, such that:
 \begin{equation} \label{dt.m.v}
  \partial_t^m v = \partial_t^m u + P_m(u, \partial_t u, \ldots, \partial_t^{m-1}u, x).
 \end{equation}
 For $m = 0$,~\eqref{dt.m.v} is a rephrasing of $v = u - \alp{x}$ with 
 $P_0(x) := - \alp{x}$. For a fixed $m \in \N$, we differentiate~\eqref{dt.m.v} 
 with respect to time using~\eqref{system.fxu}:
 \begin{equation} \label{dt.mm.x}
  \partial_t^{m+1} v 
  = \partial_t^{m+1} u
  + (\partial_x P_m) f(x,u) 
  + \sum_{j = 0}^{m-1} \partial_t^{j+1} u (\partial_j P_m).
 \end{equation}
 Hence,~\eqref{dt.mm.x} proves~\eqref{dt.m.v} at order $m+1$ provided that we 
 define:
 \begin{equation} \label{def.pm}
  \begin{split}
   P_{m+1}(a_0, \ldots a_{m}, x)
   & := \sum_{j=0}^{m-1} a_{j+1} (\partial_j P_m)(a_0, \ldots a_{m-1}, x) \\
   & \quad \quad + (\partial_x P_m) (a_0, \ldots a_{m-1}, x) f(x, a_0).
  \end{split}
 \end{equation}
 From~\eqref{def.pm} and since $f(0,0) = 0$, $P_{m+1}$ vanishes at zero.
 In particular, from~\eqref{dt.m.v} and the null value of the $P_m$ at zero, 
 the derivatives of $v$ vanish at the initial time as soon as those 
 of $u$ vanish for trajectories with $x(0) = 0$.
 
 \bigskip
 
 \noindent \textbf{Last}, we prove~\eqref{eq.Lemma:Norms.2}.
 From the smoothness of $P_m$ and the null value of $P_m$ at zero, we deduce the
 existence of $A_m, \rho_m > 0$ such that, for any
 $(a_0, \ldots, a_{m-1}) \in \R^{m}$, for any $x \in \R^n$,
 \begin{equation} \label{Pm.implication}
 |x| + \sum_{j=0}^{m-1} |a_j| \leqslant \rho_m
 \Rightarrow
 \left| P_m(a_0,\ldots, a_{m-1}, x)\right| \leqslant
 A_m \left(|x| + \sum_{j=0}^{m-1} |a_j|\right).
 \end{equation}
 From Lemma~\ref{Lemma:xu}, $x = \Oz{\norm{u}{L^\infty}}$.
 Thus, from~\eqref{Pm.implication} there exists $B_m, \eta_m > 0$ such that:
 \begin{equation} \label{Pm.implication2}
  \norm{u}{W^{m,\infty}} \leqslant \eta_m
  \quad \Rightarrow \quad 
  \left| P_m(u,\partial_t u, \ldots, \partial_t^{m-1} u, x)\right| \leqslant
  B_m \norm{u}{W^{m,\infty}}.
 \end{equation}
 From~\eqref{dt.m.v} and~\eqref{Pm.implication2}, 
 estimate~\eqref{eq.Lemma:Norms.2} holds with $C_m := 2 B_m$.
\end{proof}

We finish the proof of Theorem~\ref{Theorem:QANL} in the general case.
Let $f \in C^\infty(\R^n\times\R,\R^n)$ with $f(0,0) = 0$. 
Let $\beta \in \R^n$ be defined as in paragraph~\ref{Paragraph:Brunovsky}.
We apply the associated linear static state feedback transformation to
the full nonlinear system by setting $v := u - (\beta, x)$ and considering
the transformed system as in~\eqref{def.fg}. The new system $\dot{x} = g(x,v)$
satisfies the integrator assumption~\eqref{hyp.integrator}. Hence,
we know from Section~\ref{Section:well_prepared} that Theorem~\ref{Theorem:QANL}
is true for this system. Let us check that it also holds for the initial system.
\begin{itemize}
 
 \item Let us assume that $\Stwo + \R d_0 = \Sone$ for the initial $f$-system.
 Then, thanks to Lemma~\ref{Lemma:Lie_Bruno}, this is also the case for the
 transformed $g$-system. From Theorem~\ref{Theorem:QANL} applied to $g$,
 there exists a map $G \in C^\infty(\Sone, \Sone^\perp)$ such that, for
 any $T > 0$, there exists $M, \eta > 0$ such that, for any trajectory
 $(x,v) \in C^0([0,T],\R^n) \times L^\infty((0,T),\R)$ of 
 $\dot{x} = g(x,v)$ with $x(0) = 0$ and $\norm{v}{L^\infty} \leqslant \eta$:
 \begin{equation}
  \forall t \in [0,T], \quad
  \left| \mathbb{P}^\perp x(t) - G(\mathbb{P}x(t)) \right| \leqslant
  M \norm{v}{L^3}^3.
 \end{equation}
 Thanks to Lemma~\ref{Lemma:Norms}, if 
 $(x,u) \in C^0([0,T],\R^n) \times L^\infty((0,T),\R)$ is a trajectory
 of~\eqref{system.fxu} with $x(0) = 0$ and $\norm{u}{L^\infty} \leqslant 
 \min (\eta_0, \eta/C_0)$
 then $\norm{v}{L^\infty} \leqslant \eta$ and:
 \begin{equation}
  \forall t \in [0,T], \quad
  \left| \mathbb{P}^\perp x(t) - G(\mathbb{P}x(t)) \right| \leqslant
  C M \norm{u}{L^3}^3.
 \end{equation}
 
 \item Let us assume that $\Stwo + \R d_0 \not\subset \Sone$ for the initial
 $f$-system. Then there exists a smallest $0 \leqslant k \leqslant d$
 such that $d_k \neq 0$. Moreover, thanks to Lemma~\ref{Lemma:Lie_Bruno},
 this direction is the same for the $g$-system (see~\eqref{equiv.equiv}).
 Thus, there is no ambiguity. From Theorem~\ref{Theorem:QANL} applied to~$g$, 
 there exits $T^* > 0$ such that, for any $\tm < T^*$ and any
 $T \in (0,\tm]$, there exists $\eta > 0$, for any trajectory  
 $(x,v) \in C^0([0,T],\R^n) \times L^\infty((0,T),\R)$ of 
 $\dot{x} = g(x,v)$ with $x(0) = 0$, $v \in W^{2k,\infty}_0(0,T)$  
  and $\norm{v}{W^{2k,\infty}} \leqslant \eta$:
 \begin{equation} \label{drift.dk.v}
  \forall t \in [0,T], \quad
  \left\langle \mathbb{P}^\perp x(t) - G(\mathbb{P}x(t)), d_k \right\rangle 
  \geqslant 0.
 \end{equation}
  Thanks to Lemma~\ref{Lemma:Norms}, if 
  $(x,u) \in C^0([0,T],\R^n) \times L^\infty((0,T),\R)$ is a trajectory
  of~\eqref{system.fxu} with $x(0) = 0$, such that $u \in W^{2k,\infty}_0(0,T)$ and $\norm{u}{W^{2k,\infty}} \leqslant 
  \min (\eta_{2k}, \eta/C_{2k})$
  then $v \in W^{2k,\infty}(0,T)$,
  $\norm{v}{W^{2k,\infty}} \leqslant \eta$ and~\eqref{drift.dk.v} holds.
\end{itemize}

\subsubsection{Generalization of the proof for control-affine systems}

We finish the proof of Theorem~\ref{Theorem:QAA} in the general case for
control-affine systems. We proceed exactly as for nonlinear systems:
using a Brunovsk\'y transformation and the invariance of the geometric
and analytic notions involved.

It only remains to be checked that the weaker norms $W^{-1,3}$ and
$W^{-1,\infty}$ of the control are preserved under linear static state 
feedback transformations. From Lemma~\ref{Lemma:x1u1}, one has:
\begin{equation} \label{x.o.u1}
 x(t) = \Oo{|u_1(t)| + \norm{u_1}{L^1}}.
\end{equation}
Integrating~\eqref{x.o.u1} yields:
\begin{equation} \label{bx.o.u1}
 \alp{\int_0^t x(s) \ds} = \Oo{\norm{u_1}{L^1}}.
\end{equation}
In particular, still denoting by $v = u - \alp{x}$, 
equation~\eqref{bx.o.u1} and H\'older's inequality 
yield the missing estimates:
\begin{align}
 \label{v1.L3}
 \norm{v_1}{L^3} & = \Oo{\norm{u_1}{L^3}}, \\
 \label{v1.Linf}
 \norm{v_1}{L^\infty} & = \Oo{\norm{u_1}{L^\infty}}.
\end{align}

\subsubsection{Coercivity estimate for ill-prepared systems}
\label{Subsubsection:IllCoercivity}

In paragraph~\ref{Subsubsection:coercivity}, we proved that, for well-prepared
systems satisfying~\eqref{hyp.integrator}, the quadratic drift along $d_k$ is
quantified by the $H^{-k}$-norm of the control (the $L^2$-norm of $u_k$). For
ill-prepared systems we therefore obtain that the drift is quantified by the 
$L^2$-norm of $v_k$, with $v = u - \alp{x}$ (where the transformation is chosen
from using the Brunovsk\'y approach). In the following lines, we prove that, 
for small enough times, this also yields coercivity with respect to the 
$L^2$-norm of $u_k$, thereby justifying the comment announced in 
Subsection~\ref{Subsection:Comments}.

\paragraph{Obstruction of order $0$ (for nonlinear systems).} 
Let us assume that we are considering a nonlinear system for which 
$d_0 \notin \Sone$. Hence, for $\tm < T^*$ (where $T^*$ is defined 
by~\eqref{def:TT} and~\eqref{def:T*}), 
the drift along $d_0$ is quantified by the $L^2$-norm
of $v$. From Lemma~\ref{Lemma:xu} and the Cauchy-Schwarz inequality, one has:
\begin{equation}
 |x(t)| = \Oz{\sqrt{t} \norm{u}{L^2}}.
\end{equation}
Hence:
\begin{equation}
 \norm{v}{L^2(0,t)} = \norm{u}{L^2(0,t)} \left(1 + \Oz{t}\right) \geqslant
 \frac{1}{2}\norm{u}{L^2(0,t)},
\end{equation}
provided that $\tm$ is small enough and that $u$ is small enough in $L^\infty$.
Therefore, the drift is coercive with respect to $\norm{u}{L^2}$ for small
enough times and controls.

\paragraph{Obstruction of order $1 \leqslant k \leqslant d$.}
Let us assume that we are considering a system for which $d_0 \in \Sone$
(or $d_0 = 0$ for control-affine systems) but $d_k \notin \Sone$. 
Hence, for $\tm < T^*$, the drift along $d_k$ is quantified by the $L^2$-norm
of $v_k$. From Lemma~\ref{Lemma:AuxEst1} and the Cauchy-Schwarz inequality, 
we have:
\begin{equation}
 |\aux_d(t)| = \Og{\sqrt{t} \norm{u_d}{L^2(0,t)} + \norm{u_\gamma}{L^2(0,t)}^2}.
\end{equation}
Moreover, from the $C^2$ regularity of the flow $\phi_1, \ldots \phi_d$, one has:
\begin{equation} \label{decomp.x}
 x = u_1 f_1(0) + \ldots + u_d f_d(0) + \Og{|\aux_d| + |u_1|^2 + \ldots + |u_d|^2}.
\end{equation}
Integrating~\eqref{decomp.x} with respect to time $k$ times yields:
\begin{equation} \label{vk.uk}
 \begin{split}
  \norm{v_k}{L^2(0,t)} = \norm{u_k}{L^2(0,t)}
   & +  \Ou{\norm{u_{k+1}}{L^2(0,t)} + \ldots + \norm{u_{k+d}}{L^2(0,t)}} \\
   & + t^{k+\frac{1}{2}} \Og{ \norm{u_d}{L^2(0,t)}} \\
   &  + t^k \Og{ \norm{u_\gamma}{L^2(0,t)}^2} \\
   & + t^{k-1} \Og{\norm{u_1}{L^2(0,t)}^2 + \ldots + \norm{u_d}{L^2(0,t)}^2}.
 \end{split}
\end{equation}
For $1 \leqslant j \leqslant d$, $\norm{u_{k+j}}{L^2(0,t)} = \Ou{t \norm{u_k}{L^2(0,t)}}$.
Hence, the error terms in the first line are small when the time is small enough.
The error term of the second line is also small because $k \leqslant d$ so 
it is of order $\Og{t^{d+1/2} \norm{u_k}{L^2(0,t)}}$. In the
fourth line, the first error term dominates the following ones for small enough 
times. Hence, it remains to be checked that the following quantity is small enough:
\begin{equation} \label{def.U}
 U(t) := t^k \norm{u_\gamma}{L^2(0,t)}^2 + t^{k-1} \norm{u_1}{L^2(0,t)}^2 
\end{equation}
Thanks to the Gagliardo-Nirenberg interpolation inequality, for each 
$1 \leqslant k \leqslant d$, there exists $C_1, C_2 > 0$ (independent on $T$)
such that, for $l \in \{0,1\}$:
\begin{equation} \label{GN:Bis}
 \norm{u_l}{L^2(0,t)}^2 \leqslant
 C_1 \norm{u^{(2k-2l)}_k}{L^2(0,t)} \norm{u_k}{L^2(0,t)} + C_2 t^{2l-2k} \norm{u_k}{L^2(0,t)}^2.
\end{equation}

\bigskip \noindent \textit{First case: $k = 1$}. We start with this low-order
case which is handled a little differently. For control-affine systems, we
directly obtain from~\eqref{def.U} that:
\begin{equation} \label{U.k1.1}
 U(t) \leqslant (1+t) \norm{u_1}{L^2(0,t)}^2 
 \leqslant (1+t) \sqrt{t} \norm{u_1}{L^\infty} \norm{u_1}{L^2(0,t)}.
\end{equation}
For nonlinear systems, we use the interpolation inequality~\eqref{GN:Bis} 
for the first term:
\begin{equation} \label{U.k1.2}
 \begin{split}
  t \norm{u}{L^2(0,t)}^2 
  & \leqslant t C_1 \norm{\dot{u}}{L^2(0,t)} \norm{u_1}{L^2(0,t)} 
  + C_2 t^{-1} \norm{u_1}{L^2(0,t)}^2 \\
  & \leqslant t^{\frac{3}{2}} C_1 \norm{\dot{u}}{L^\infty} \norm{u_1}{L^2(0,t)} 
  + C_2 t^{\frac{1}{2}} \norm{u}{L^\infty} \norm{u_1}{L^2(0,t)}. 
 \end{split}
\end{equation}
The second term is estimated as $\norm{u_1}{L^2(0,t)}^2 
\leqslant t^{\frac{3}{2}} \norm{u}{L^\infty} \norm{u_1}{L^2(0,t)}$. 
Both estimates~\eqref{U.k1.1} and~\eqref{U.k1.2} lead to the conclusion
that $U(t)$ is small with respect to $\norm{u_1}{L^2(0,t)}$ when 
$u_\gamma \to 0$ in $L^\infty$. From~\eqref{vk.uk}, we conclude that:
\begin{equation} \label{vk.uk.12}
 \norm{v_k}{L^2(0,t)} \geqslant \frac{1}{2} \norm{u_k}{L^2(0,t)},
\end{equation}
for small enough times and small enough controls.

\bigskip \noindent \textit{Second case: $k \geqslant 2$}.
Thanks to~\eqref{GN:Bis}, for $l \in \{0, 1\}$, one has:
\begin{equation}
 \begin{split}
 t^{k-l} \norm{u_l}{L^2(0,t)}^2 
 & \leqslant
 C_1 t^{k-l} \norm{u^{(2k-2l)}_k}{L^2(0,t)} \norm{u_k}{L^2(0,t)} 
 + C_2 t^{l-k} \norm{u_k}{L^2(0,t)}^2 \\
 & \leqslant
 C_1 t^{k-l+\frac{1}{2}} \norm{u^{(k-2l)}}{L^\infty} \norm{u_k}{L^2(0,t)} 
 + C_2 t^{l+\frac{1}{2}} \norm{u}{L^\infty} \norm{u_k}{L^2(0,t)}.
 \end{split}
\end{equation}
Using~\eqref{def.U}, we conclude that:
\begin{equation}
 U(t) = \Og{\norm{u}{W^{k-2\gamma,\infty}} + \norm{u}{W^{k-2,\infty}}} \norm{u_k}{L^2}.
\end{equation}
For nonlinear systems, $k < 2k$ and $k - 2 < 2k$.
For control-affine systems, $k - 2 < 2k-3$. Hence, in both cases:
\begin{equation}
U(t) = \Og{\norm{u}{W^{2k-3\gamma,\infty}} \norm{u_k}{L^2}} .
\end{equation}
Hence, when $u \to 0$ in $W^{2k-3\gamma}$, $U(t)$ is small with respect
to $\norm{u_k}{L^2(0,t)}$ and we conclude that~\eqref{vk.uk.12}
also holds for small enough times and small enough controls.

\subsection{Persistence of results for less regular systems} 
\label{Subsection:rough}

Theorems~\ref{Theorem:QANL} and~\ref{Theorem:QAA} are stated with smooth vector 
fields to facilitate their understanding. However, the same conclusions can be 
extended to less regular dynamics and this highlights that our method only 
relies on the quadratic behavior of the system. We explain 
how one can extend the quadratic alternative to such less regular systems. 
As observed in Subsection~\ref{Subsec:S2_nonsmooth}, the spaces $\Sone$ 
and $\Stwo$ can be defined as soon as $f \in C^2(\R^n\times\R,\R^n)$.
We start with general nonlinear systems.

\begin{cor} \label{Corollary:NonlinearC3}
 Let $f \in C^3(\R^n\times\R,\R^n)$ with $f(0,0) = 0$. 
 The conclusions of Theorem~\ref{Theorem:QANL} hold.
\end{cor}

\begin{proof}
 We start by considering a truncated system with initial condition $\hat{x}(0) = 0$:
 \begin{equation} \label{system.trunc.nl}
  \dot{\hat{x}} 
  = T_2 f(\hat{x}, u)
  = H_0 \hat{x} + u b + \frac{1}{2} Q_0(\hat{x},\hat{x})
  + u H_1 \hat{x} + \frac{1}{2} d_0 u^2,
 \end{equation}
 where $T_2 f$ denotes the second-order Taylor expansion of $f$ around the 
 origin, with straightforward notations already used in the previous subsections.
 This defines a new nonlinear system, which is smooth and therefore satisfies
 the quadratic alternative of Theorem~\ref{Theorem:QANL}. Moreover, by
 definition (see Subsection~\ref{Subsec:S2_nonsmooth}), 
 the spaces $\Sone$ and $\Stwo$ for this new system
 are the same as those of the original system. Let us prove that $x$ and 
 $\hat{x}$ are close. One has:
 \begin{equation} \label{xxhat.evol}
  \begin{split}
   \dot{x} - \dot{\hat{x}}
   & = f(x,u) - T_2 f(\hat{x},u) \\
   & = \big(f(x,u) - T_2 f(x,u)\big) + (H_0 + u H_1) (x - \hat{x})
    + \frac{1}{2} Q_0(x-\hat{x},x+\hat{x}).
  \end{split}
 \end{equation} 
 From Lemma~\ref{Lemma:xu} applied to $f$ and $T_2f$:
 \begin{align}
  \label{ptnl.bigO.x1}
  x(t) & = \Oz{1}, \\
  \label{ptnl.bigO.hatx}
  \hat{x}(t) & = \Oz{1}, \\
  \label{ptnl.bigO.xu}
  x(t) & = \Oz{\norm{u}{L^1}}.
  \end{align}
 Since $u(t) = \Oz{1}$, one has:
 \begin{equation} \label{xxhat.1}
  (H_0 + u H_1) (x - \hat{x}) = \Oz{|x-\hat{x}|}.
 \end{equation}
 From~\eqref{ptnl.bigO.x1} and~\eqref{ptnl.bigO.hatx}, one has:
 \begin{equation} \label{xxhat.2}
  Q_0(x-\hat{x},x+\hat{x}) = \Oz{|x-\hat{x}|}.
 \end{equation}
 Since $f \in C^3(\R^n\times\R,\R^n)$:
 \begin{equation} \label{hyp.C3}
  |f(x,u) - T_2f(x,u)| =\Oz{|x|^3 + |u|^3}.
 \end{equation}
 From~\eqref{hyp.C3},~\eqref{ptnl.bigO.xu} and H\"older's inequality, one has:
 \begin{equation} \label{xxhat.3}
  f(x,u) - T_2 f(x,u) 
   = \Oz{ \norm{u}{L^1}^3 + |u|^3 }
   = \Oz{ \norm{u}{L^3}^3 + |u|^3 }.
 \end{equation}
 Plugging estimates~\eqref{xxhat.1},~\eqref{xxhat.2} and~\eqref{xxhat.3}
 into~\eqref{xxhat.evol} and integrating yields:
 \begin{equation} \label{xxhat.evol2}
  x(t) - \hat{x}(t) = \Oz{ \norm{u}{L^3}^3 
  + \int_0^t |x(s) - \hat{x}(s)| \ds }.
 \end{equation}
 Applying Gr\"onwall's lemma to~\eqref{xxhat.evol2} provides the estimate:
 \begin{equation} \label{xxhat.evol3}
  x(t) - \hat{x}(t) =  \Oz{ \norm{u}{L^3}^3 }.
 \end{equation}
 Let $\alpha \in \R^n$. We consider the transformed control $v := u - 
 (\alpha,\hat{x})$. With this new control, the evolution 
 equation~\eqref{system.trunc.nl} becomes $\dot{\hat{x}} = g(\hat{x},v)$ where 
 $g(\hat{x},v) := T_2 f(\hat{x},v+(\alpha,\hat{x}))$.
 Hence, applying Lemma~\ref{Lemma:xu} to this system yields:
 \begin{equation} \label{hatx.ov}
  \hat{x}(t) = \Oz{\norm{v}{L^1}} = \Oz{\norm{u-(\alpha,\hat{x})}{L^1}}.
 \end{equation}
 Combining~\eqref{xxhat.evol3} with~\eqref{hatx.ov}, using 
 the triangular inequality and H\"older's inequality gives:
 \begin{equation} \label{xxhat.evol4}
 x(t) - \hat{x}(t) =  \Oz{ \norm{u - (\alpha,\hat{x})}{L^3}^3 }.
 \end{equation}
 Let $G : \Sone \to \Sone^\perp$ be the smooth function associated with
 system~\eqref{system.trunc.nl} by Theorem~\ref{Theorem:QANL}. Since
 $G'$ is bounded in a vicinity of zero, estimate~\eqref{xxhat.evol4} yields:
 \begin{equation} \label{xxhat.evol5}
  \mathbb{P}^\perp x(t) - G(\mathbb{P}x(t)) 
  = \mathbb{P}^\perp \hat{x}(t) - G(\mathbb{P}\hat{x}(t)) +
  \Oz{\norm{u - (\alpha,\hat{x})}{L^3}^3 }.
 \end{equation}
 Estimate~\eqref{xxhat.evol5} allows to transpose all the conclusions on the 
 state $\hat{x}$ to conclusions on the state $x$ since the remainder is exactly
 of the same size as those that have been managed in 
 Section~\ref{Section:well_prepared}, provided that one chooses the $\alpha$
 corresponding to the Brunovsk\'y transform.
\end{proof}

\begin{cor} \label{Corollary:AffineC3}
 Let $f_0 \in C^3(\R^n,\R^n)$ and $f_1 \in C^2(\R^n,\R^n)$ with $f_0(0) = 0$.
 There exists $G \in C^2(\Sone, \Sone^\perp)$ with $G(0) = 0$ and $G'(0) = 0$
 such that the conclusions of Theorem~\ref{Theorem:QAA} hold.
\end{cor}

\begin{proof}
 We consider a regularized system:
 \begin{equation} \label{system.hatf0f1}
   \dot{\hat{x}} = \hat{f}_0(\hat{x}) + u \hat{f}_1(\hat{x}),
 \end{equation}
 where $\hat{f}_0 := T_2 f_0$ is the second-order Taylor expansion of $f_0$
 at zero and $\hat{f}_1 := T_1 f_1$ is the first-order Taylor expansion of
 $f_1$ at zero. Using the notations introduced in 
 Section~\ref{Section:well_prepared}, we consider the
 respective auxiliary systems $\aux_1$ and $\hat{\aux}_1$
 associated with systems~\eqref{system.f0f1} and~\eqref{system.hatf0f1}. 
 One has:
 \begin{align} 
  \label{evol.aux1}
  \dot{\aux}_1 
  & = f_0(\aux_1) + u_1 f_2(\aux_1) + u_1^2 G_1(u_1, \aux_1), \\
  \label{evol.aux1.hat}
  \dot{\hat{\aux}}_1 
  & = \hat{f_0}(\hat{\aux}_1) + u_1 \hat{f}_2(\hat{\aux}_1)
  + u_1^2 \hat{G}_1(u_1, \hat{\aux}_1).
 \end{align}
 From Lemma~\ref{Lemma:x1u1} and estimate~\eqref{Estimate:Aux1},
 we have $\aux_1 = \Oo{1}$ and $\hat{\aux}_1 = \Oo{1}$.
 Since $f_0 \in C^3(\R^n,\R^n)$ and $\hat{f}_0 = T_2 f_0$,
 \begin{equation} \label{est.rough.1}
  \begin{split}
   f_0(\aux_1) - \hat{f_0}(\hat{\aux}_1)
   & = \left(f_0(\aux_1) - f_0(\hat{\aux}_1)\right)
     + \left(f_0(\hat{\aux}_1) - \hat{f_0}(\hat{\aux}_1)\right) \\
   & = \Oo{|\aux_1 - \hat{\aux}_1|}
     + \Oo{|\hat{\aux}_1|^3}.
  \end{split}
 \end{equation}
 Similarly, since $f_2 = -[f_0,f_1] \in C^1(\R^n,\R^n)$
 and $\hat{f}_2 = - [\hat{f}_0,\hat{f}_1] = f_2 + O(x^2)$
 one has:
 \begin{equation} \label{est.rough.2}
  \begin{split}
   f_2(\aux_1) - \hat{f_2}(\hat{\aux}_1)
   & = \left(f_2(\aux_1) - f_2(\hat{\aux}_1)\right)
     + \left(f_2(\hat{\aux}_1) - \hat{f_2}(\hat{\aux}_1)\right) \\
   & = \Oo{|\aux_1 - \hat{\aux}_1|}
     + \Oo{|\hat{\aux}_1|^2}.
  \end{split}
 \end{equation} 
 Last, one checks that 
 $\hat{G}_1(0,0) = \frac{1}{2} [\hat{f}_1,\hat{f}_2](0)
 = \frac{1}{2} [f_1,f_2](0) = G_1(0,0)$. 
 Moreover, we have
 $G_1 \in C^1(\R^n,\R^n)$.
 Indeed, since $f_1$ is $C^2$, $\phi_1$ is $C^3$ with
 respect to $\tau$ and $C^2$ with respect to~$p$, thus
 $F_1$ is $C^3$ with respect to $\tau$ and $C^1$ with respect to~$p$.
 Thus $G_1$, which is obtained from $F_1$ through
 a Taylor formula with integral remainder, is at least $C^1$
 in $(\tau,p)$.
 Hence:
 \begin{equation} \label{est.rough.3}
   G_1(u_1,\aux_1) - \hat{G_1}(u_1,\hat{\aux}_1)
   = \Oo{|u_1| + |\aux_1| + |\hat{\aux}_1|}.
 \end{equation}
 From Lemma~\ref{Lemma:x1u1} and estimate~\eqref{Estimate:Aux1},
 we also have $\aux_1 = \Oo{\norm{u_1}{L^1}}$ 
 and $\hat{\aux}_1 = \Oo{\norm{u_1}{L^1}}$.
 Therefore, plugging estimates~\eqref{est.rough.1},
 \eqref{est.rough.2} and~\eqref{est.rough.3} into
 the difference of~\eqref{evol.aux1.hat} with \eqref{evol.aux1},
 integrating in time and using H\"older's inequality yields:
 \begin{equation} \label{est.rough.4}
  \aux_1(t) - \hat{\aux}_1(t)
  = \Oo{\norm{u_1}{L^3}^3}
  + \int_0^t \Oo{|\aux_1(s) - \hat{\aux}_1(s)|} \ds.
 \end{equation}
 Applying Gr\"onwall's lemma to~\eqref{est.rough.4} gives:
 \begin{equation} \label{est.rough.5}
  r(t) := \aux_1(t) - \hat{\aux}_1(t) = \Oo{\norm{u_1}{L^3}^3}.
 \end{equation}
 We denote by $\phi_1$ and $\hat{\phi}_1$, the flows associated 
 with $f_1$ and $\hat{f}_1$. Hence:
 \begin{equation} \label{x.Phi.O1}
   x 
   = \phi_1(u_1, \aux_1)
   = \phi_1(u_1, \hat{\aux}_1 + r)
   = \phi_1(u_1, \hat{\phi}_1(-u_1, \hat{x}) + r)
   = \Phi(u_1, \hat{x}) + \Oo{|r|},
 \end{equation}
 thanks to the $C^1$ regularity of $\phi_1$,
 where we introduce the map:
 \begin{equation} \label{def.Phi}
   \Phi(\tau,p) := \phi_1\left(\tau, \hat{\phi}_1(-\tau, p)\right).
 \end{equation}
 Differentiating~\eqref{def.Phi} 
 and using the shorthand notation $\hat{p}_1 := \hat{\phi}_1(-\tau, p)$
 yields:
 \begin{equation} \label{dtau.Phi}
   \begin{split}
     \partial_\tau \Phi(\tau, p)
     & = \partial_\tau \phi_1\left(\tau, \hat{p}_1\right)
       - \partial_p \phi_1\left(\tau, \hat{p}_1\right)
         \cdot \partial_\tau \hat{\phi}_1(-\tau,p) \\
     & = f_1 (\phi_1(\tau, \hat{p}_1))
       - \partial_p \phi_1\left(\tau, \hat{p}_1\right)
         \cdot \hat{f}_1(\hat{p}_1) \\
     & = \Psi(\tau, \hat{p}_1) 
     - \partial_p \phi_1(\tau, \hat{p}_1) \cdot 
     \left( \hat{f}_1(\hat{p}_1) - f_1(\hat{p}_1) \right),
   \end{split}
 \end{equation}
 where we introduced:
 \begin{equation} \label{def.Psi}
   \Psi(\tau, p) := f_1 (\phi_1(\tau, p)) 
    - \partial_p \phi_1\left(\tau, p\right) \cdot f_1(p). 
 \end{equation}
 One checks that $\Psi(0, p) = 0$. 
 Moreover, using Schwarz's theorem, we obtain:
 \begin{equation} \label{eq.psi'}
   \partial_\tau \Psi(\tau,p) 
   = f_1'(\phi_1(\tau,p)) \cdot f_1(\phi_1(\tau,p))
   - \partial_{\tau p} \phi_1(\tau, p) \cdot f_1(p)
   = f_1'(\phi_1(\tau,p)) \cdot \Psi(\tau,p).
 \end{equation}
 From~\eqref{eq.psi'}, we deduce that $\Psi(\tau,p) = 0$.
 Hence,~\eqref{dtau.Phi} yields:
 \begin{equation} \label{est.dtauphi}
  \partial_\tau \Phi(u_1,\hat{x}) 
  = \Oo{|\hat{\aux}_1|^2}.
 \end{equation}
 For $p \in \R^n$, let us denote by $\mathbb{P}_0(p)$ the component
 of $\mathbb{P}p$ along $b_0 = b = f_1(0)$ in the basis of $\Sone$
 made up of $(b_0, \ldots b_{d-1})$ (see~\eqref{def.vk}).
 Considering the function 
 $\tau \mapsto \mathbb{P}_0(\hat{\phi}_1(\tau, 0))$,
 thanks to the local inversion theorem, there exists a smooth function 
 $\beta_0 : \R \to \R$ such that, if $p = \hat{\phi}_1(\tau, 0)$
 then $\tau = \beta_0 (\mathbb{P}_0 (p))$ for $\tau$ and $p$
 small enough. Hence, since $\hat{x} = \hat{\phi}_1(u_1, \hat{\aux}_1)$,
 the $C^1$ regularity of $\hat{\phi}_1$ yields:
 \begin{equation} \label{u1b1}
  u_1 = \beta_0(\mathbb{P}_0 \hat{x}) + \Oo{|\hat{\aux}_1|}.
 \end{equation}
 Gathering~\eqref{est.rough.5},~\eqref{x.Phi.O1},~\eqref{est.dtauphi} and~\eqref{u1b1} yields:
 \begin{equation}
  x = \Phi(\beta_0(\mathbb{P}_0 \hat{x}), \hat{x}) 
  + \Oo{\norm{u_1}{L^3}^3}.
 \end{equation}
 The cubic remainder is of the same size as those that are absorbed during
 the usual proof (in fact, during the proof, one absorbs such remainders on
 a transformed control $v = u - (\alpha,\hat{x})$ 
 but we have already seen that the involved norms on $u_1$ and $v_1$
 can be interchanged, as in~\eqref{v1.L3} and~\eqref{v1.Linf}). 
 Thus, we only need to study the first
 part of this equation and prove that it defines a manifold.
 We look at:
 \begin{equation}
  x = \Phi(\beta_0(\mathbb{P}_0 \hat{x}), \mathbb{P}\hat{x} + \hat{G}(\mathbb{P}\hat{x})).
 \end{equation}
 Once again, the local inversion theorem tells us that we can express
 $\mathbb{P}\hat{x}$ as a $C^2$ function $\rho(\mathbb{P}x)$. Thus, we set:
 \begin{equation}
  G(p_\parl) := 
  \mathbb{P}^\perp \Phi(\beta_0(\mathbb{P}_0\rho(p_\parl)), \rho(p_\parl) + \hat{G}(\rho(p_\parl)).
 \end{equation}
 This concludes the proof of the theorem for less regular
 vector fields.
\end{proof}

\section{Explicit approximation of the invariant manifold} \label{Section:explicit}

We construct an explicit quadratic approximation of the 
invariant manifold involved in our quadratic alternative theorems. We also
exhibit a second-order approximation of the differential systems
that stays exactly within this manifold when $\Stwo = \Sone$. 

\bigskip

We continue to denote by $1 \leqslant d < n$ the dimension of $\Sone$.
Using the notations introduced in Subsection~\ref{Subsection:C2}, \eqref{def.vk},~\eqref{def.L0} and~\eqref{def.Lk}, we define:
\begin{equation} \label{def.M2}
 \mathcal{M}_2 := \left\{
  \sum_{0 \leqslant i < d} \alpha_i b_i
  + \mathbb{P}^\perp \left( \frac{1}{2} \sum_{0 \leqslant i < d} \alpha_i^2  L_i b_i
  + \sum_{0 \leqslant i < j < d} \alpha_i \alpha_j 
   L_i b_j \right),
 \enskip \alpha \in \R^d
 \right\}.
\end{equation}
Since $b_0, \ldots b_{d-1}$ are $d$ independent vectors which 
span $\Sone$, equation~\eqref{def.M2} defines a global smooth manifold of 
$\R^n$ of dimension $d$. Indeed, it is defined as the image of a smooth
injective function, with smooth inverse. We recover from~\eqref{def.M2}
that the tangent subspace to~$\mathcal{M}_2$ at the origin is 
$\Sone$. For $x \in \R^n$, and $0 \leqslant k \leqslant d-1$, we introduce 
the notation $\mathbb{P}_k (x)$ to denote the component of $\mathbb{P}x$
along $b_k$ in the basis $(b_0, \ldots b_{d-1})$. Formally, 
one could also define $\mathcal{M}_2$ as the set of $x \in \R^n$ such that:
\begin{equation} \label{eq.manifold}
 \mathbb{P}^\perp x
 = 
  \frac{1}{2} \sum_{0 \leqslant i < d} \mathbb{P}_i(x)^2 \mathbb{P}^\perp(L_i b_i)
  + \sum_{0 \leqslant i < j < d} \mathbb{P}_i(x) \mathbb{P}_j(x)
  \mathbb{P}^\perp(L_i b_j).
\end{equation}

\subsection{Local expansion of the invariant manifold}

\begin{prop} \label{Prop:manifolds}
 The manifold $\mathcal{M}_2$ defined by~\eqref{def.M2} is a local 
 approximation of second-order of the manifold $\M$ involved in the 
 main quadratic alternative theorems around the origin.
\end{prop}

\begin{proof}
 The manifold $\M$ is defined in~\eqref{def.MG1}.
 First, both manifolds contain the origin and admit $\Sone$ as their tangent
 subspace at the origin. Here, we check that the second-order terms 
 contained in $\Sone^\perp$ are the same. Since both definitions are given
 as graphs, we need to check that the second-order Taylor expansions match.
 From~\eqref{def.G}, we have $G(p_\parl) = \Theta(p_\parl, 0)$. 
 From~\eqref{def:FFtilde} and~\eqref{eq.FtildeTheta}, we have:
 \begin{equation} \label{eq.FG}
  \forall p_\parl \in \Sone, \quad
  F(p_\parl, G(p_\parl)) = 0.
 \end{equation}
 Differentiating~\eqref{eq.FG} twice with respect to $p_\parl$ yields:
 \begin{equation} \label{eq.diff.FG}
  \partial_\parl^2 F + 2 \partial_\perp\partial_\parl F \cdot G'
  + \partial_\perp^2 F \cdot G' G' + \partial_\perp F \cdot G'' = 0.
 \end{equation}
 At the origin, $G = 0$, $G' = 0$ and $\partial_\perp F = \mathrm{Id}$ 
 from~\eqref{diff_d=Id}. Hence,~\eqref{eq.diff.FG} yields:
 \begin{equation} \label{eq.G''F}
  G''(0) = - \partial_\parl^2 F(0,0).
 \end{equation}
 We must compute a Taylor approximation of $F$ with respect to
 $p_\parl$. Let $\phi_j$ be defined by~\eqref{flow_fj}.
 Using a Taylor expansion with respect to time, one has:
 \begin{equation} \label{eq.phitaup.1}
  \phi_j(\tau, p) \eqmod{3} \phi_j(0, p) + \tau \partial_\tau \phi_j(0,p)
  + \frac{1}{2} \tau^2 \partial_\tau^2 \phi_j(0,p).
 \end{equation}
 Using~\eqref{flow_fj} and~\eqref{eq.phitaup.1}, one has:
 \begin{equation} \label{eq.phitaup.2}
  \phi_j(\tau, p) \eqmod{3} p + \tau \left(f_j(0) + f'_j(0)p\right)
  + \frac{1}{2} \tau^2 f'_j(0)f_j(0).
 \end{equation}
 Let $\alpha_1, \ldots \alpha_d \in \R^d$.
 Iterated application of~\eqref{eq.phitaup.2} yields:
 \begin{equation} \label{eq.phitaup.3}
  \begin{split}
   \phi_d^{-\alpha_d} \circ \cdots \circ \phi_1^{-\alpha_1}(p)
   & \eqmod{3} p - \sum_{i=1}^d \alpha_i f_i(0) 
   + \frac{1}{2} \sum_{i=1}^d \alpha_i^2 f_i'(0) f_i(0)  \\
   & \quad - \sum_{i=1}^d \alpha_i f_i'(0) \left( p 
   - \sum_{j = 1}^{i-1} \alpha_j f_j(0) \right).
  \end{split}
 \end{equation}
 Using~\eqref{def.psi} and~\eqref{def.alpha}, one obtains,
 for $p_\parl \in \Sone$, since $\alpha_i(0) = 0$:
 \begin{equation} \label{xparl.alpha}
  p_\parl \eqmod{2} \alpha_1(p_\parl) f_1(0) + \ldots + \alpha_d(p_\parl) f_d(0).
 \end{equation}
 Plugging~\eqref{eq.phitaup.3} and~\eqref{xparl.alpha} into
 definition~\eqref{def.F} yields:
 \begin{equation} \label{F.eqmod2}
  \begin{split}
   F(p_\parl, 0)
   & \eqmod{3} \frac{1}{2} \sum_{i=1}^d \alpha_i^2(p_\parl) \mathbb{P}^\perp (f_i'(0) f_i(0))
   - \sum_{i=1}^d \alpha_i (p_\parl)
   \sum_{j=i}^d \alpha_j (p_\parl) \mathbb{P}^\perp (f_i'(0) f_j(0)) \\
   & \eqmod{3} - \frac{1}{2} \sum_{i=1}^d \alpha_i^2(p_\parl) \mathbb{P}^\perp (f_i'(0) f_i(0))
   - \sum_{1 \leqslant i < j \leqslant d} \alpha_i(p_\parl) \alpha_j (p_\parl)
   \mathbb{P}^\perp (f_i'(0) f_j(0)).
  \end{split}
 \end{equation}
 For $1 \leqslant k \leqslant d$, recalling the definition of $f_k$ 
 (see~\eqref{def.fj}) and using Lemma~\ref{Lemma:adkbklk}, 
 one has $f_k(0) = (-1)^{k-1} b_{k-1}$ and $f_k'(0) = (-1)^{k-1} L_{k-1}$.
 From~\eqref{eq.G''F} and~\eqref{F.eqmod2}, we have:
 \begin{equation} \label{G.eqmod3}
 G(p_\parl) \eqmod{3} \frac{1}{2} \sum_{i=0}^{d-1} 
 \mathbb{P}_{i}(p_\parl)^2 \mathbb{P}^\perp (L_i b_i)
 + \sum_{0 \leqslant i < j < d}
 \mathbb{P}_{i}(p_\parl) \mathbb{P}_{j}(p_\parl)
 \mathbb{P}^\perp (L_i b_j).
 \end{equation}
 Since~\eqref{G.eqmod3} matches~\eqref{eq.manifold}, it
 concludes the proof of Proposition~\ref{Prop:manifolds}.
\end{proof}

\subsection{Construction of an homogeneous second-order system}

We construct an homogeneous second-order system that provides
a good approximation of any differential system and stays exactly within
$\mathcal{M}_2$, under the assumption that $\Stwo = \Sone$. Approximating the 
behavior of nonlinear systems using homogeneous (with respect
to amplitude dilatations) approximations has already been used in various contexts:
for small-time local controllability (around an equilibrium in~\cite{MR872457}
and around a trajectory in~\cite{MR1051629}), for large-time local 
controllability in~\cite{MR1395834}, for stabilization in~\cite{MR1096761} and 
the construction of Lyapunov functions in~\cite{MR1195304}.

\bigskip

As proposed in Section~\ref{Section:well_prepared}, given some control-affine
or nonlinear differential system, we decompose the state $x$ 
as $y + z + r$, where $y$ denotes the linear part, 
$z$~the quadratic part and $r$ a remainder (which is thus at 
least cubic in the control). Thanks to the linear theory, we know that $y$ 
lives in $\Sone$. Hence, an homogeneous approximation of $x$ up to the second
order is the quantity: 
\begin{equation} \label{def.zeta}
 \zeta := \mathbb{P}y + \mathbb{P}^\perp z.
\end{equation}
In~\eqref{def.zeta}, we write $\mathbb{P}y$ instead of $y$ to 
highlight the orthogonality with respect to the second term $\mathbb{P}^\perp z$.
Using the notations introduced in~\eqref{DEF_d0},~\eqref{def:H0b},~\eqref{def:H1} 
and~\eqref{def.Q0}, we recall that:
\begin{align}
\label{eq.y.recall}
\dot{y} & = H_0 y + u b, \\
\label{eq.z.recall}
\dot{z} & = H_0 z + u H_1 y + Q_0(y, y) + \frac{1}{2} u^2 d_0.
\end{align}
Since $\Sone$ is stable under multiplication by $H_0$, we have the relations:
\begin{align}
\label{eq.php}
\mathbb{P} H_0 \mathbb{P} & = H_0 \mathbb{P}, \\
\label{eq.pphpp}
\mathbb{P}^\perp H_0 \mathbb{P}^\perp & = \mathbb{P}^\perp H_0.
\end{align}
Applying $\mathbb{P}$ to~\eqref{eq.y.recall} and using~\eqref{eq.php},
then $\mathbb{P}^\perp$ to~\eqref{eq.z.recall}
and using~\eqref{eq.pphpp}, one obtains:
\begin{align} 
\label{eq.y.zeta}
\mathbb{P} \dot{\zeta} & = H_0 \mathbb{P} \zeta + u \mathbb{P} b, \\
\label{eq.z.zeta}
\mathbb{P}^\perp \dot{\zeta} & = \mathbb{P}^\perp H_0 \zeta + u \mathbb{P}^\perp H_1 \mathbb{P} \zeta
+ \mathbb{P}^\perp Q_0(\mathbb{P} \zeta, \mathbb{P} \zeta)
+ \frac{1}{2} u^2 \mathbb{P}^\perp d_0.
\end{align}
Combining~\eqref{eq.y.zeta} and~\eqref{eq.z.zeta} 
leads to the following ODE for $\zeta$:
\begin{equation} \label{system.zeta}
 \dot{\zeta} = g_0(\zeta) + u g_1(\zeta) + \frac{1}{2} u^2 \mathbb{P}^\perp d_0,
\end{equation}
where:
\begin{align}
\label{def.g0}
g_0(\zeta) & := (H_0 \mathbb{P} + \mathbb{P}^\perp H_0) \zeta + \mathbb{P}^\perp Q_0(\mathbb{P}\zeta, \mathbb{P}\zeta), \\
\label{def.g1}
g_1(\zeta) & := b + \mathbb{P}^\perp H_1 \mathbb{P} \zeta.
\end{align}
We prove in the following lemma that system~\eqref{system.zeta}
exhibits nice properties concerning the Lie brackets of $g_0$ and
$g_1$ since they are fully explicit.

\begin{lem} \label{Lemma:Lie_zeta}
 Let $j, k \in \N$. For any $\zeta \in \R^n$, we have:
 \begin{align} \label{eq.adk.g0g1}
 \ad^k_{g_0}(g_1)(\zeta) & = b_k + \mathbb{P}^\perp L_k \mathbb{P} \zeta, 
 \\
 \label{eq.adkj.g0g1}
 \left[\ad^k_{g_0}(g_1), \ad^j_{g_0}(g_1)\right](\zeta) & = 
 \mathbb{P}^\perp (L_j b_k - L_k b_j).
 \end{align}
\end{lem}

\begin{proof}
 We proceed by induction on $k \in \N$. From~\eqref{def.vk},~\eqref{def.L0} 
 and~\eqref{def.g1},~\eqref{eq.adk.g0g1} holds for $k = 0$. Let $k \in \N$ be 
 such that~\eqref{eq.adk.g0g1} holds. Using~\eqref{def.g0}, we compute the next 
 bracket:
 \begin{equation} \label{eq.adk.g0g1.1}
 \begin{split}
 \ad^{k+1}_{g_0}(g_1) (\zeta)
 & = \left[g_0, \ad^k_{g_0}(g_1)\right](\zeta) \\
 & = (\mathbb{P}^\perp L_k \mathbb{P}) \left(H_0 \mathbb{P} \zeta + \mathbb{P}^\perp H_0 \zeta + \mathbb{P}^\perp Q_0(\mathbb{P}\zeta, \mathbb{P}\zeta)\right) \\
 & \quad - H_0 \mathbb{P} \left(b_k + \mathbb{P}^\perp L_k \mathbb{P} \zeta\right) 
 - \mathbb{P}^\perp H_0 \left(b_k + \mathbb{P}^\perp L_k \mathbb{P} \zeta\right) \\
 & \quad - 2 \mathbb{P}^\perp Q_0(b_k, \mathbb{P}\zeta)
 - 2 \mathbb{P}^\perp Q_0(\mathbb{P} \mathbb{P}^\perp L_k \mathbb{P}\zeta, \mathbb{P}\zeta).
 \end{split}
 \end{equation}
 Thanks to~\eqref{eq.php},~\eqref{eq.pphpp} and the relation $\mathbb{P} \mathbb{P}^\perp = 0$,
 we deduce from~\eqref{eq.adk.g0g1.1} that:
 \begin{equation} \label{eq.adk.g0g1.2}
 \ad^{k+1}_{g_0}(g_1) (\zeta)
 = b_{k+1} + \mathbb{P}^\perp \left(L_k H_0 - H_0 L_k - 2 Q_0(b_k, \cdot)\right) \mathbb{P} \zeta.
 \end{equation}
 The conclusion follows from~\eqref{eq.adk.g0g1.2} because we obtain the
 same recursion relation as in~\eqref{def.Lk}. Thus~\eqref{eq.adk.g0g1}
 holds for any $k \in \N$.
 Using~\eqref{eq.adk.g0g1} and $\mathbb{P}\mathbb{P}^\perp = 0$, we compute:
 \begin{equation} \label{eq.adkj.g0g1.1}
 \begin{split}
 \left[\ad^k_{g_0}(g_1), \ad^j_{g_0}(g_1)\right](\zeta)
 & = \mathbb{P}^\perp L_j \mathbb{P} \left(b_k + \mathbb{P}^\perp L_k \mathbb{P} \zeta\right)
 - \mathbb{P}^\perp L_k \mathbb{P} \left(b_j + \mathbb{P}^\perp L_j \mathbb{P} \zeta\right) \\
 & = \mathbb{P}^\perp (L_j b_k - L_k b_j).
 \end{split}
 \end{equation}
 Hence,~\eqref{eq.adkj.g0g1.1} concludes the proof of~\eqref{eq.adkj.g0g1}.
\end{proof}

\begin{lem} \label{Lemma:MFRAK}
 Assume that $\Stwo + \R d_0 = \Sone$. Then, there exists a smooth 
 manifold $\mathfrak{M} \subset \R^n$ of dimension $d := \mathrm{dim~} \Sone$, 
 such that any trajectory of system~\eqref{system.zeta} with $\zeta(0) = 0$
 satisfies $\zeta(t) \in \mathfrak{M}$ for any $t \geqslant 0$.
\end{lem}

\begin{proof}
 We start by proving that, for any $\zeta \in \R^n$:
 \begin{equation} \label{def.Lie.g0g1}
  \Lie~\left\{ g_0, g_1 \right\}(\zeta) = 
  \Span \left\{ b_k + \mathbb{P}^\perp L_k \mathbb{P} \zeta, \enskip 0 \leqslant k < d \right\}.
 \end{equation}
 We use Lemma~\ref{Lemma:Lie_zeta}. 
 From~\eqref{eq.adkj} and~\eqref{eq.adkj.g0g1}, we obtain:
 \begin{equation} \label{eq.adkj.g0g1.2}
  \begin{split}
  \left[\ad^k_{g_0}(g_1), \ad^j_{g_0}(g_1)\right](\zeta)
  & = \mathbb{P}^\perp (L_j b_k - L_k b_j) \\
  & = \mathbb{P}^\perp \left[\ad^k_{f_0}(f_1), \ad^j_{f_0}(f_1)\right](0) \\
  & = 0.
  \end{split}
 \end{equation}
 Hence, from~\eqref{eq.adkj.g0g1.2}, all brackets containing
 $g_1$ at least two times vanish identically. Thus, the space 
 $\Lie~\left\{ g_0, g_1 \right\}(\zeta)$ is spanned
 by brackets containing $g_1$ exactly once.
 Using~\eqref{eq.adk.g0g1} and 
 $\mathbb{P}^\perp L_j b_k = \mathbb{P}^\perp L_k b_j$, we compute:
 \begin{equation} \label{eq.adjg0g1}
 \begin{split}
 \ad^{j}_{g_0}(g_1) (\zeta) & = b_j + \mathbb{P}^\perp L_j \mathbb{P} \zeta \\
 & = b_j + \mathbb{P}^\perp L_j \left(\sum_{l=0}^{d-1} \mathbb{P}_l(\zeta) b_l\right) \\
 & = b_j + \sum_{l=0}^{d-1} \mathbb{P}_l(\zeta) \mathbb{P}^\perp L_l  b_j \\
 & = b_j + \sum_{k=0}^{d-1} \mathbb{P}_k(b_j) \sum_{l=0}^{d-1}  
 \mathbb{P}_l(\zeta) \mathbb{P}^\perp L_l b_k \\
 & = \sum_{k=0}^{d-1} \mathbb{P}_k(b_j) b_k + \sum_{k=0}^{d-1} \mathbb{P}_k(b_j) \mathbb{P}^\perp L_k \left(\sum_{l=0}^{d-1}  
 \mathbb{P}_l(\zeta) b_l\right) \\
 & = \sum_{k=0}^{d-1} \mathbb{P}_k(b_j) \left(b_k + \mathbb{P}^\perp L_k \mathbb{P} \zeta \right).
 \end{split}
 \end{equation}
 Equation~\eqref{eq.adjg0g1} proves that the $d$ first brackets span 
 $\Lie~\left\{ g_0, g_1 \right\}(\zeta)$
 and thus~\eqref{def.Lie.g0g1} holds.
 From the definition of $\Sone$ and $d = \mathrm{dim~} \Sone$, the family 
 $b_0, \ldots b_{d-1}$ is free. Thus, from~\eqref{def.Lie.g0g1},
 we have that, for any $\zeta \in \R^n$, the dimension of 
 $\Lie~\left\{ g_0, g_1 \right\}(\zeta)$ is exactly $d$.
 Hence, from the Frobenius theorem (as stated in~\cite[Corollary 3.26]{MR2302744}), 
 the state $\zeta(t)$ must evolve within a manifold of $\R^n$ of 
 dimension $d$ (see also~\cite[Theorem 4, page 45]{MR1425878} 
 or~\cite[Theorem 2.20, page 48]{MR1967689} for proofs of this geometric result).
\end{proof}

\subsection{Exact evolution within the quadratic manifold}

We prove that Lemma~\ref{Lemma:MFRAK} actually holds with
$\mathfrak{M} = \mathcal{M}_2$ defined in~\eqref{def.M2}. Using the assumption
$\Stwo = \Sone$ and~\eqref{eq.adkj}, one has $\mathbb{P}^\perp L_j b_k =
\mathbb{P}^\perp L_k b_j$. Hence~\eqref{def.M2} can be rewritten in a more
symmetric way as:
\begin{equation} \label{def.Mquad}
 \mathcal{M}_2 = \left\{
 \sum_{k=0}^{d-1} \alpha_k b_k 
 + \frac{1}{2} \sum_{j=0}^{d-1} \sum_{k=0}^{d-1} \alpha_j \alpha_k \mathbb{P}^\perp L_j b_k,
 \enskip
 (\alpha_0, \ldots \alpha_{d-1}) \in \R^d
 \right\}.
\end{equation}
Equivalently, it corresponds to the set of $x \in \R^n$ 
for which the following vector-valued second-order polynomial vanishes:
\begin{equation} \label{def.Q}
 Q(x) := \mathbb{P}^\perp x - \frac{1}{2}
 \sum_{j=0}^{d-1} \sum_{k=0}^{d-1} 
 \mathbb{P}_j(x) \mathbb{P}_k(x) \mathbb{P}^\perp L_j b_k.
\end{equation}
We compute the evolution of $Q$ along trajectories $t \mapsto \zeta(t)$
by differentiating~\eqref{def.Q}:
\begin{equation} \label{eq.Q.1}
\frac{\mathrm{d}}{\dt} Q(\zeta(t))
= \mathbb{P}^\perp \dot{\zeta} 
- \frac{1}{2} \sum_{j=0}^{d-1} \sum_{k=0}^{d-1} \mathbb{P}_j(\dot{\zeta}) \mathbb{P}_k(\zeta) \mathbb{P}^\perp L_j b_k
- \frac{1}{2} \sum_{j=0}^{d-1} \sum_{k=0}^{d-1} \mathbb{P}_j(\zeta) \mathbb{P}_k(\dot{\zeta}) \mathbb{P}^\perp L_j b_k.
\end{equation}
Recalling that $\mathbb{P}^\perp L_j b_k = \mathbb{P}^\perp L_k b_j$ and using~\eqref{eq.y.zeta} 
and~\eqref{eq.z.zeta}, we have from~\eqref{eq.Q.1}:
\begin{equation} \label{eq.Q.2}
\begin{split}
\frac{\mathrm{d}Q}{\dt}
& = \mathbb{P}^\perp H_0 \mathbb{P}^\perp \zeta + u \mathbb{P}^\perp H_1 \mathbb{P} \zeta + \mathbb{P}^\perp Q_0(\mathbb{P} \zeta, \mathbb{P} \zeta) \\
& \quad - \sum_{j=0}^{d-1} \sum_{k=0}^{d-1} \mathbb{P}_j(H_0 \mathbb{P} \zeta + u b) \mathbb{P}_k(\zeta) \mathbb{P}^\perp L_j b_k.
\end{split}
\end{equation}
Since $b = b_0$ (see~\eqref{def.vk}) and $L_0 = H_1$ (see~\eqref{def.L0}), we 
have:
\begin{equation} \label{eq.Q.3}
\sum_{j=0}^{d-1} \sum_{k=0}^{d-1} \mathbb{P}_j(u b) \mathbb{P}_k(\zeta) \mathbb{P}^\perp L_j b_k
= u \sum_{k=0}^{d-1} \mathbb{P}_k(\zeta) \mathbb{P}^\perp L_0 b_k
= u \mathbb{P}^\perp H_1 \mathbb{P} \zeta.
\end{equation}
From~\eqref{eq.Q.2} and~\eqref{eq.Q.3}, one has:
\begin{equation} \label{eq.Q.4}
\frac{\mathrm{d}Q}{\dt}
= \mathbb{P}^\perp H_0 \mathbb{P}^\perp \zeta + \mathbb{P}^\perp Q_0(\mathbb{P} \zeta, \mathbb{P} \zeta)
- \sum_{j=0}^{d-1} \sum_{k=0}^{d-1} \mathbb{P}_j(H_0 \mathbb{P} \zeta) \mathbb{P}_k(\zeta) \mathbb{P}^\perp L_j b_k.
\end{equation}
From~\eqref{def.Lk},~\eqref{eq.pphpp} and~\eqref{def.Q}, one has:
\begin{equation} \label{eq.Q.5}
\begin{split}
\mathbb{P}^\perp H_0 \mathbb{P}^\perp \zeta & + \mathbb{P}^\perp Q_0(\mathbb{P} \zeta, \mathbb{P} \zeta) \\
& = \mathbb{P}^\perp H_0 Q
+ \frac{1}{2} \sum_{j=0}^{d-1} \mathbb{P}_j(\zeta) \mathbb{P}^\perp H_0 \mathbb{P}^\perp L_j \mathbb{P} \zeta  
+ \mathbb{P}^\perp Q_0(\mathbb{P} \zeta, \mathbb{P} \zeta) \\
& = \mathbb{P}^\perp H_0 Q
+ \frac{1}{2} \sum_{j=0}^{d-1} \mathbb{P}_j(\zeta) \mathbb{P}^\perp \left(H_0 L_j \mathbb{P} \zeta + 2 Q_0(b_j, \mathbb{P}\zeta)\right) \\
& = \mathbb{P}^\perp H_0 Q - \frac{1}{2} \sum_{j=0}^{d-1} \mathbb{P}_j(\zeta) \mathbb{P}^\perp L_{j+1} \mathbb{P}\zeta
+ \frac{1}{2} \sum_{j=0}^{d-1} \mathbb{P}_j(\zeta) \mathbb{P}^\perp L_j H_0 \mathbb{P}\zeta.
\end{split}
\end{equation}
The last term in~\eqref{eq.Q.5} can be further simplified using~\eqref{def.vk}. Indeed:
\begin{equation} \label{eq.Q.6}
\begin{split}
\frac{1}{2} \sum_{j=0}^{d-1} \mathbb{P}_j(\zeta) \mathbb{P}^\perp L_j H_0 \mathbb{P}\zeta
& = \frac{1}{2} \sum_{j=0}^{d-1} \sum_{k=0}^{d-1} \mathbb{P}_j(\zeta) \mathbb{P}_k(\zeta) \mathbb{P}^\perp L_j H_0 b_k \\
& = - \frac{1}{2} \sum_{j=0}^{d-1} \sum_{k=0}^{d-1} \mathbb{P}_j(\zeta) \mathbb{P}_k(\zeta) \mathbb{P}^\perp L_j b_{k+1} \\
& = - \frac{1}{2} \sum_{j=0}^{d-1} \sum_{k=0}^{d-1} \mathbb{P}_j(\zeta) \mathbb{P}_k(\zeta) \mathbb{P}^\perp L_{k+1} b_j \\
& = - \frac{1}{2} \sum_{k=0}^{d-1} \mathbb{P}_k(\zeta) \mathbb{P}^\perp L_{k+1} \mathbb{P} \zeta.
\end{split}
\end{equation}
Similarly, the last term of~\eqref{eq.Q.4} can be rewritten:
\begin{equation} \label{eq.Q.7}
\begin{split}
- \sum_{j=0}^{d-1} \sum_{k=0}^{d-1} \mathbb{P}_j(H_0 \mathbb{P} \zeta) \mathbb{P}_k(\zeta) \mathbb{P}^\perp L_j b_k
& = - \sum_{j=0}^{d-1} \sum_{k=0}^{d-1} \sum_{l=0}^{d-1} \mathbb{P}_l(\zeta) \mathbb{P}_j(H_0 b_l) \mathbb{P}_k(\zeta) \mathbb{P}^\perp L_k b_j \\
& = - \sum_{k=0}^{d-1} \sum_{l=0}^{d-1} \mathbb{P}_l(\zeta) \mathbb{P}_k(\zeta) \mathbb{P}^\perp L_k H_0 b_l \\
& = \sum_{k=0}^{d-1} \sum_{l=0}^{d-1} \mathbb{P}_l(\zeta) \mathbb{P}_k(\zeta) \mathbb{P}^\perp L_k b_{l+1} \\
& = \sum_{k=0}^{d-1} \sum_{l=0}^{d-1} \mathbb{P}_l(\zeta) \mathbb{P}_k(\zeta) \mathbb{P}^\perp L_{l+1} b_k \\
& = \sum_{l=0}^{d-1} \mathbb{P}_l(\zeta) \mathbb{P}^\perp L_{l+1} \mathbb{P}\zeta.
\end{split}
\end{equation}
Hence, grouping~\eqref{eq.Q.4},~\eqref{eq.Q.5},~\eqref{eq.Q.6} and~\eqref{eq.Q.7}, 
we conclude that:
\begin{equation} \label{eq.Q.8}
 \frac{\mathrm{d}}{\dt}Q\left(\zeta(t)\right) = \mathbb{P}^\perp H_0 Q\left(\zeta(t)\right).
\end{equation}
Straight-forward integration of~\eqref{eq.Q.8} yields, for any $t \geqslant 0$:
\begin{equation} \label{eq.Q.9}
 Q\left(\zeta(t)\right) = e^{t \mathbb{P}^\perp H_0} Q(\zeta(0)).
\end{equation}
In particular, we deduce from~\eqref{eq.Q.9} that the evolution of $Q$ along
trajectories of $\zeta(t)$ does not depend on the control. Moreover, when
$\zeta(0) = 0$, $Q(\zeta(0)) = Q(0) = 0$ and this remains true for any
positive time: the quadratic homogeneous model cannot leave the manifold.
Hence Lemma~\ref{Lemma:MFRAK} holds with $\mathfrak{M} = \mathcal{M}_2$.

\begin{rk}
 In this paper, we introduced two different "quadratic" approximations for a
 nonlinear system: $y+z$ and $\zeta$. The decomposition $y+z$ used in 
 Section~\ref{Section:well_prepared} is quite classical but it does not 
 behave as well as $\zeta$. Indeed $\zeta$ provides an approximation which is
 homogeneous with respect to dilatations of order one in $\Sone$ and order
 two in $\Sone^\perp$ (while $y+z$ mixes first and second-order terms in
 $\Sone$). A clear indication that $\zeta$ behaves more nicely than $y+z$ is 
 Lemma~\ref{Lemma:MFRAK}, since $\zeta$ lives exactly within a given manifold.
 This remark might hint towards introducing well-prepared homogeneous 
 approximations instead of standard approximations to study local properties.
 Of course, the approximation $y+z$ is also relevant. It lives within 
 $\M_2$ up to a cubic residual. Indeed, since $\zeta = y + \mathbb{P}^\perp z$
 and using~\eqref{def.Q}:
 \begin{equation}
  Q(y+z) = Q(\zeta) + \Og{|\zeta| |Pz|}
  = \Og{\norm{u_\gamma}{L^\infty}^3}. 
 \end{equation}
\end{rk}

\subsection{Examples of approximate invariant manifolds}

We consider variations around the toy system exposed in 
Example~\ref{Example:toy_manifold} 
to illustrate the difference between
$\M$ and its approximation $\M_2$. When the system is already second-order
homogeneous, we have $\M_2 = \M$. However, this is not always the case.

\begin{example}[Higher-order terms] \label{Example:M_HOT}
 We consider the following variation:
 \begin{equation} \label{system.hot}
  \left\{
  \begin{aligned}
  \dot{x}_1 & = u, \\
  \dot{x}_2 & = 2 u x_1 + 3 u x_1^2.
  \end{aligned}
  \right.
 \end{equation}
 System~\eqref{system.hot} also satisfies $\Sone = \R e_1$ and 
 $\Stwo = \Sone$. However, the invariant manifold
 $\M$ defined in~\eqref{def.MG1} takes into account the cubic term
 and is given by:
 \begin{equation} \label{M.hot}
  \M = \left\{ (x_1, x_2) \in \R^2, \enskip
  x_2 = x_1^2 + x_1^3 \right\}.
 \end{equation}
 Here, $\M_2$ is only a local approximation of~\eqref{M.hot} 
 of second-order (see Figure~\ref{Figure:m_vs_m2}, left plot).
\end{example}

In Example~\ref{Example:M_HOT}, the difference between $\M$ 
and $\M_2$ is due to cubic terms. It is of also possible
to build systems for which $f(x,u)$ is a polynomial of degree
two but $\M \neq \M_2$. Indeed, the homogeneity with respect
to dilatations for a system is not constrained by the
degree of the polynomials defining the dynamics.

\begin{example}[Non-polynomial invariant manifold]
 We consider the following variation
 \begin{equation} \label{system.bent}
 \left\{
 \begin{aligned}
 \dot{x}_1 & = u, \\
 \dot{x}_2 & = 2 u x_1 + u x_2.
 \end{aligned}
 \right.
 \end{equation}
 System~\eqref{system.bent} also satisfies $\Sone = \R e_1$ and 
 $\Stwo = \Sone$. The invariant manifold is given by:
 \begin{equation} \label{def.mbent}
 \mathcal{M} =  
 \left\{ (x_1, x_2) \in \R^n, \enskip 
 x_2 = 2 \left(e^{x_1} - 1 - x_1\right) \right\}.
 \end{equation}
 One checks that $\M_2$ is indeed the local approximation
 of~\eqref{def.mbent} because of the Taylor expansion
 $2(e^{x_1}-1-x_1) = x_1^2 + O(x_1^3)$
 for $|x_1|\leqslant 1$ (see Figure~\ref{Figure:m_vs_m2}, right plot).
\end{example}

\afficherfigure{
\begin{figure}[ht!]
 \begin{center}
  \includegraphics[width=6.6cm]{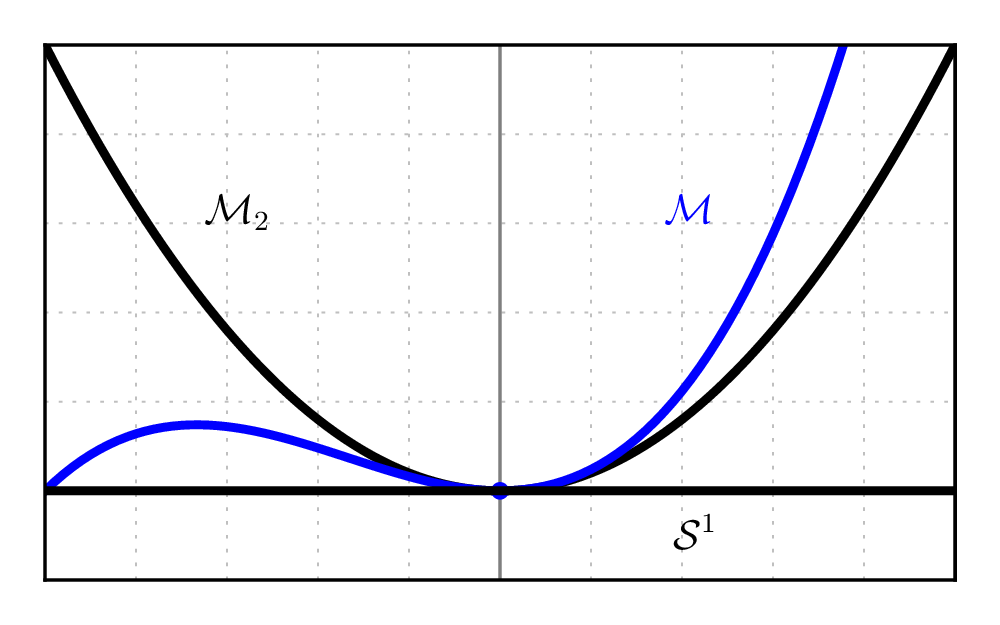}
  \includegraphics[width=6.6cm]{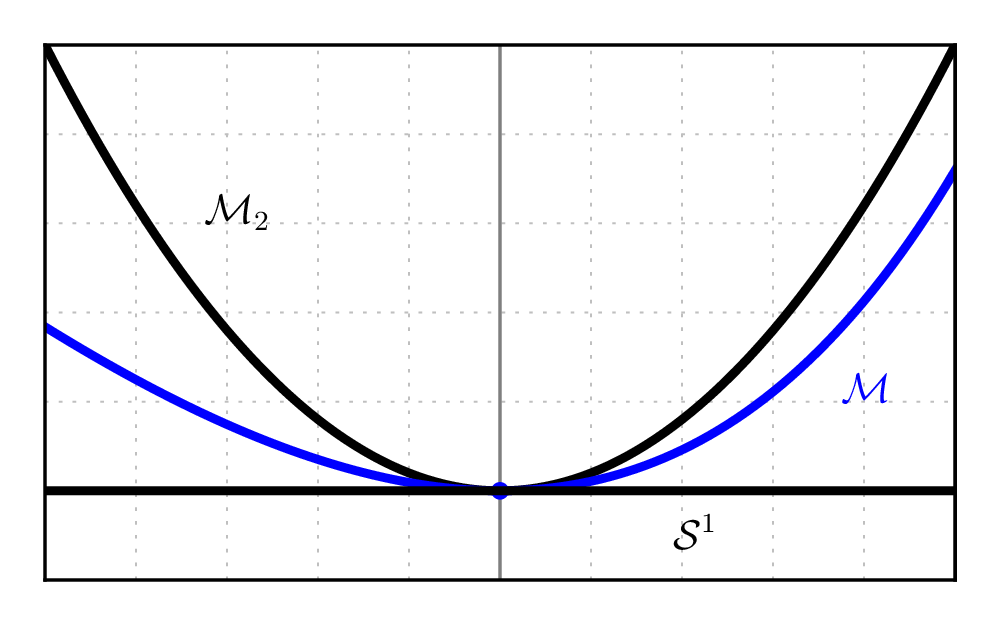}
  \caption{\label{Figure:m_vs_m2}
   Influence of left-out terms on the invariant manifold.}
 \end{center}
\end{figure}
}

\section{On other notions of small-time local controllability}
\label{Section:SSSTLC}

As sketched in the introduction, multiple definitions of small-time local
controllability can be found in the literature. In this work, we chose to
put the focus on the smallness assumption concerning the control, because we
think that it is the easiest way to highlight the links between the functional
setting and the geometric properties of the Lie brackets. However, other choices
are possible; we explore one and explain how it relates 
to our definition.

\subsection{Small-state small-time local controllability}

\begin{definition} \label{Definition:SSSTLC}
 We say that a differential system is \emph{small-state small-time locally controllable}
 when, for any $T > 0$, for any $\varepsilon > 0$, there exists $\delta > 0$
 such that, for any $x^*, x^\dagger \in \R^n$ with 
 $|x^*|+|x^\dagger| \leqslant \delta$, there exists a trajectory 
 $(x,u) \in C^0([0,T],\R^n) \times L^q((0,T),\R)$ such that
 $x(0) = x^*$, $x(T) = x^\dagger$ and:
 \begin{equation}
  \forall t \in [0,T], \quad |x(t)| \leqslant \varepsilon.
 \end{equation}
 In particular, the control is \emph{a priori} allowed to be large in $L^q((0,T),\R)$
 (with $q = \infty$ for nonlinear systems and $q = 1$ for control-affine systems).
\end{definition}

Here again, other choices would be possible: Definition~\ref{Definition:SSSTLC}
is linked to the $L^\infty$-norm of the state along the trajectory, but
one could also consider stronger norms.

\subsection{Relations between state-smallness and control-smallness}

Small-state small-time local controllability can be linked to the notions
studied in this work. 

\begin{lem}
 If a differential system is \stlc{$L^\infty$} (for
 Definition~\ref{Definition:STLC}), then it is small-state small-time locally 
 controllable (for Definition~\ref{Definition:SSSTLC}).
\end{lem}

\begin{proof}
 Let $f \in C^\infty(\R^n \times \R,\R^n)$ with $f(0,0) = 0$. 
 Let $T > 0$ and $\varepsilon > 0$. We define 
 $M := \sup \{ |f'(p,\tau)|, \enskip p \in \bar{B}(0,\varepsilon), \tau \in [-\varepsilon,\varepsilon] \}$, $C := (1 + MT) T e^{MT}$ and 
 $\sigma := \varepsilon / (2C)$. Moreover, considering
 $(x,u) \in C^0([0,T],\R^n) \times L^\infty((0,T),\R)$ with
 $|x(0)| \leqslant \sigma$ and $\norm{u}{L^\infty} \leqslant \sigma$,
 we have, as long as $x(t) \in \bar{B}(0,\varepsilon)$:
 \begin{equation} \label{xtx0.1}
  \begin{split}
   \left|x(t)\right|
   & \leqslant |x(0)| + \left| \int_0^t f(x(s),u(s)) \ds \right| \\
   & \leqslant |x(0)| + M t \norm{u}{L^\infty} + M \int_0^t |x(s)| \ds. 
  \end{split}
 \end{equation}
 Applying Gr\"onwall's lemma to~\eqref{xtx0.1} yields:
 \begin{equation} \label{xtx0.2}
  \left|x(t)\right| \leqslant e^{MT} (T |x(0)| + MT^2 \norm{u}{L^\infty}).
 \end{equation}
 Using~\eqref{xtx0.2} and the definition of $C$ leads to:
 \begin{equation} \label{xtx0.3}
 \left|x(t)\right| \leqslant C (|x(0)| + \norm{u}{L^\infty}) \leqslant \varepsilon.
 \end{equation}
 Since we assumed that the system is \stlc{$L^\infty$},
 from Definition~\ref{Definition:STLC}, there exists $\delta_\sigma > 0$
 such that, for any $x^*,x^\dagger \in \R$ with 
 $|x^*|+|x^\dagger| \leqslant \delta_\sigma$, there exists a trajectory
 from $x^*$ to $x^\dagger$ with a control smaller than $\sigma$.
 Therefore, thanks to~\eqref{xtx0.3}, 
 small-state small-time locally controllable if we set 
 $\delta := \min (\sigma, \delta_\sigma)$.
\end{proof}

For control-affine systems, a more precise result can be obtained.

\begin{lem}
 A control-affine system is \stlc{$W^{-1,\infty}$} (for 
 Definition~\ref{Definition:STLC}) if and only if it is small-state small-time 
 locally controllable (for Definition~\ref{Definition:SSSTLC}).
\end{lem}

\begin{proof}
 Let $f_0, f_1 \in C^\infty(\R^n,\R^n)$ with $f_0(0) = 0$. The forward implication
 is proved as above, using an enhanced estimate as in Lemma~\ref{Lemma:x1u1},
 obtained via the use of the first auxiliary system. 
 We turn to the reverse implication. 
 
 \bigskip \noindent \textbf{Heuristic.} The key argument is that
 $\norm{u_1}{L^\infty}$ can be estimated from $\norm{x}{L^\infty}$. Indeed, one has:
 \begin{equation} \label{xapproxu1}
  x = x(0) + u_1 f_1(0) + \text{lower order terms}.
 \end{equation}
 Thus we can hope to recover $u_1$ from the knowledge of the state.
 The lower order terms in~\eqref{xapproxu1} are easily estimated when it is known
 that $u_1$ is small. Here, this is not \emph{a priori} the case. Hence, we need
 to invert relation~\eqref{xapproxu1} so that the lower order terms can be estimated
 from the state and not from the control.
 
 \bigskip \noindent \textbf{Construction of an input-to-state map.}
 Let $T > 0$ be fixed. 
 For $p \in \R^n$, we denote as previously by $\mathbb{P}_0 p$ the component
 of $\mathbb{P}p$ along $b_0$ in the basis $(b_0, \ldots b_{d-1})$ of $\Sone$.
 We introduce the spaces:
 \begin{align}
  A & := \{ (x^*,v) \in \R^n \times C^0([0,T],\R); \enskip v(0)=0 \}, \\
  B & := \{ (x^*,\pi) \in \R^n \times C^0([0,T],\R); \enskip \pi(0)=\mathbb{P}_0 x^* \}
 \end{align}
 and the following input-to-state map:
 \begin{equation}
  \mathcal{F}: \left\{
   \begin{aligned}
    A & \to B, \\
    (x^*, v) & \mapsto \left(x^*, \mathbb{P}_0 \phi_1(v, \aux_1)\right),
   \end{aligned}
  \right.
 \end{equation}
 where $\phi_1$ is defined by~\eqref{flow_fj} and $\aux_1$ is the solution to:
 \begin{equation}
  \dot{\aux}_1 = f_0(\aux_1) + v f_2(\aux_1) + v^2 G_1(v, \aux_1),
 \end{equation}
 with initial data $\aux_1(0) = x^*$. One has $\mathcal{F}(0,0) = (0,0)$.
 Straightforward Gr\"onwall estimates prove that $\mathcal{F}$ is 
 well-defined and $C^1$ on a small neighborhood of $(0,0)$ in $A$.

 \bigskip \noindent \textbf{Local inversion at zero.}
 For $(x^*,v) \in A$, we compute:
 \begin{equation} \label{F'00}
  \mathcal{F}'(0,0) \cdot (x^*, v)
  = \left( x^*, \mathbb{P}_0\left( \partial_\tau \phi_1(0,0) v
  + \partial_p \phi_1(0,0) y_1 \right) \right),
 \end{equation}
 where $y_1$ is the solution to $y_1(0) = x^*$ and:
 \begin{equation} \label{evol.y1}
  \dot{y}_1 = H_0 y_1 + v H_0 b.
 \end{equation}
 Since $\mathbb{P}_0 \partial_\tau \phi_1(0,0)
 = \mathbb{P}_0 f_1(0) = 1$ and $\partial_p \phi_1(0,0) = \mathrm{Id}$,
 equation~\eqref{F'00} yields:
 \begin{equation} \label{F'00.2}
  \mathcal{F}'(0,0) \cdot (x^*, v)
  = \left( x^*, v + \mathbb{P}_0 y_1 \right),
 \end{equation}
 From~\eqref{evol.y1}, one obtains:
 \begin{equation}
  \mathbb{P}_0 y_1(t) = \mathbb{P}_0(e^{t H_0} x^*)
  + \int_0^t v(s) \mathbb{P}_0 \left(e^{(t-s)H_0} H_0 b\right) \ds.
 \end{equation}
 Let $(x^*,\pi) \in B$. Solving $\mathcal{F}'(0,0) \cdot (x^*, v) = (x^*,\pi)$
 amounts to finding a $v \in C^0([0,T],\R)$ with $v(0)=0$ such that:
 \begin{equation} \label{volt1}
  v(t) + \int_0^t v(s) \mathbb{P}_0 \left(e^{(t-s)H_0} H_0 b\right) \ds
  = \pi(t) - \mathbb{P}_0(e^{t H_0} x^*) =: h(t),
 \end{equation}
 where $h(0) = 0$.
 We are faced with a linear Volterra integral equation of second-kind,
 with a smooth kernel. We refer to~\cite[Section 1.2]{brunner_2017} for an introduction on this topic. 
 Classical theory for such problems (see e.g.\ \cite[Theorem 1.2.3]{brunner_2017})
 yields the existence of a continuous
 resolvent kernel $K$ such that~\eqref{volt1} is equivalent to:
 \begin{equation}
  v(t) = h(t) + \int_0^t K(t,s) h(s) \ds.
 \end{equation}
 Hence $\mathcal{F}'(0,0)$ is invertible. 
 The inverse function theorem then allows us to conclude that
 there exists $C, \delta_A, \delta_B > 0$ such that,
 for any $(x^*,v) \in A$, if $|x^*| \leqslant \delta_A$
 and $\norm{v}{C^0} \leqslant \delta_A$
 and $\mathcal{F}(x^*,v) = (x^*,\pi)$
 with $\norm{\pi}{C^0} \leqslant \delta_B$, then:
 \begin{equation} \label{ineq.final.0T}
  \norm{v}{L^\infty([0,T])} \leqslant \frac{C}{2} \left( |x^*| + \norm{\pi}{L^\infty([0,T])} \right)
  \leqslant C \norm{\phi_1(v,\aux_1)}{L^\infty([0,T])}.
 \end{equation}
 Moreover, since the map $\mathcal{F}$ is causal (in the sense that
 the value of $\mathbb{P}_0 \phi_1(v,\aux_1)$ at time $t \in [0,T]$ only
 depends on the values of $v$ on the past time interval $[0,t]$), the same
 property holds for its inverse. Thus, estimate~\eqref{ineq.final.0T} yields:
 \begin{equation} \label{ineq.final.I}
  \forall T' \in (0,T], \quad
   \norm{v}{L^\infty([0,T'])}
   \leqslant C \norm{\phi_1(v,\aux_1)}{L^\infty([0,T'])}.
 \end{equation}
 
 \bigskip \noindent \textbf{Progressive estimation.}
 We assume that the system is small-state small-time locally controllable.
 Let $\eta > 0$. We define  
 $\varepsilon := \min \{\delta_A, \delta_B, \delta_A / (2C), \eta / C\}$.
 Let $\delta > 0$ be given by the application of Definition~\ref{Definition:SSSTLC}.
 For any states $x^*, x^\dagger \in \bar{B}(0,\delta/2)$, there 
 exists a trajectory $(x,u) \in C^0([0,T],\R^n) \times L^1((0,T),\R)$ 
 with $x(0) = x^*$, $x(T) = x^\dagger$ and
 $\norm{x}{L^\infty} \leqslant \varepsilon$.
 Moreover, $u_1 \in C^0([0,T])$ with $u_1(0) = 0$.
 Let $\bar{T} := \sup \{ T' \in [0,T]; \enskip \norm{u_1}{L^\infty(0,T')} \leqslant \delta_A / 2 \}$. 
 Since $u_1$ is continuous and vanishes at the initial time, one has $\bar{T} > 0$.
 
 By contradiction, let us assume that $\bar{T} < T$. Then, by continuity of
 $u_1$, there exists $T' \in (\bar{T},T]$ such that $\norm{u_1}{L^\infty(0,T')} \leqslant \delta_A$.
 Thus, we can apply~\eqref{ineq.final.I} and obtain that $\norm{u_1}{L^\infty(0,T')} \leqslant C \varepsilon
 \leqslant \delta_A / 2$. Hence $\bar{T} = T$.
 
 Eventually, we can apply~\eqref{ineq.final.0T} and obtain 
 that $\norm{u_1}{L^\infty(0,T)} \leqslant C \varepsilon \leqslant \eta$.
 Therefore, the system is also \stlc{$W^{-1,\infty}$}.
\end{proof}

\section*{Conclusion and perspectives}

We proved that quadratic approximations for differential systems
can lead either to drifts quantified by Sobolev norms of the control or to 
the existence of an invariant manifold at the second-order. Thus, when a
nonlinear system does not satisfy the linear Kalman condition, one needs to
go at least up to the third order expansion to hope for positive results concerning
small-time local controllability. 

Our work highlights the importance of
the norm hypothesis in the definition of small-time local controllability,
even for differential systems.
Indeed, although the state lives in~$\R^n$, 
we have proved that the controllability properties depend strongly
on the norm of the control chosen in the definition of the notion.
We expect that other geometric results might be improved by exploring the
link between Lie brackets and functional settings.

For systems governed by partial differential equations, we expect that
the behaviors proved in finite dimension can also be observed. 
For example, the first author and Morancey obtain in~\cite{MR3167929}
a drift quantified by the $H^{-1}$-norm of the control, which prevents
small-time local controllability, under an 
assumption corresponding to $[f_1,[f_0,f_1]](0) \notin \Sone$.
In~\cite{MR2060480}, Coron and Cr\'epeau observe that the behavior of
the second-order expansion of a Korteweg-de-Vries system is fully
determined by the position of the linear approximation (thus recovering
a kind of invariant manifold up to the second order).

It is also known that new phenomenons can occur. For example, 
in~\cite{2015arXiv151104995M}, the second author obtains
a drift quantified by the $H^{-5/4}$-norm of the control for a Burgers
system, which thus does not seem to be linked with an integer order Lie 
bracket and is specific to the infinite dimensional setting.


\bibliographystyle{plain}
\bibliography{bibliography}

\begin{thebibliography}{10}

\bibitem{MR966201}
Andrei Agrach\"ev.
\newblock Quadratic mappings in geometric control theory.
\newblock In {\em Problems in geometry, {V}ol.\ 20 ({R}ussian)}, Itogi Nauki i
  Tekhniki, pages 111--205. Akad. Nauk SSSR, Vsesoyuz. Inst. Nauchn. i Tekhn.
  Inform., Moscow, 1988.
\newblock Translated in J. Soviet Math. {{\bf{5}}1} (1990), no. 6, 2667--2734.

\bibitem{MR2224821}
Andrei Agrach\"ev and Thomas Chambrion.
\newblock An estimation of the controllability time for single-input systems on
  compact {L}ie groups.
\newblock {\em ESAIM Control Optim. Calc. Var.}, 12(3):409--441, 2006.

\bibitem{MR2968064}
Cesar Aguilar and Andrew Lewis.
\newblock Small-time local controllability for a class of homogeneous systems.
\newblock {\em SIAM J. Control Optim.}, 50(3):1502--1517, 2012.

\bibitem{MR2004373}
Francesca Albertini and Domenico D'Alessandro.
\newblock Notions of controllability for bilinear multilevel quantum systems.
\newblock {\em IEEE Trans. Automat. Control}, 48(8):1399--1403, 2003.

\bibitem{MR3250374}
Karine Beauchard, Jean-Michel Coron, and Holger Teismann.
\newblock Minimal time for the bilinear control of {S}chr\"odinger equations.
\newblock {\em Systems Control Lett.}, 71:1--6, 2014.

\bibitem{beauchard:hal-01333537}
Karine Beauchard, Jean-Michel Coron, and Holger Teismann.
\newblock {Minimal time for the approximate bilinear control of Schr{\"o}dinger
  equations}.
\newblock working paper or preprint, June 2016.

\bibitem{MR3167929}
Karine Beauchard and Morgan Morancey.
\newblock Local controllability of 1{D} {S}chr\"odinger equations with bilinear
  control and minimal time.
\newblock {\em Math. Control Relat. Fields}, 4(2):125--160, 2014.

\bibitem{MR1051629}
Rosa~Maria Bianchini and Gianna Stefani.
\newblock Graded approximations and controllability along a trajectory.
\newblock {\em SIAM J. Control Optim.}, 28(4):903--924, 1990.

\bibitem{MR3110058}
Roger Brockett.
\newblock Controllability with quadratic drift.
\newblock {\em Math. Control Relat. Fields}, 3(4):433--446, 2013.

\bibitem{brunner_2017}
Hermann Brunner.
\newblock {\em Volterra Integral Equations: An Introduction to Theory and
  Applications}.
\newblock Cambridge Monographs on Applied and Computational Mathematics.
  Cambridge University Press, 2017.

\bibitem{MR0284247}
Pavol Brunovsk\'y.
\newblock A classification of linear controllable systems.
\newblock {\em Kybernetika (Prague)}, 6:173--188, 1970.

\bibitem{MR0001880}
Wei-Liang Chow.
\newblock \"{U}ber {S}ysteme von linearen partiellen {D}ifferentialgleichungen
  erster {O}rdnung.
\newblock {\em Math. Ann.}, 117:98--105, 1939.

\bibitem{MR2302744}
Jean-Michel Coron.
\newblock {\em Control and nonlinearity}, volume 136 of {\em Mathematical
  Surveys and Monographs}.
\newblock American Mathematical Society, Providence, RI, 2007.

\bibitem{MR2060480}
Jean-Michel Coron and Emmanuelle Cr\'epeau.
\newblock Exact boundary controllability of a nonlinear {K}d{V} equation with
  critical lengths.
\newblock {\em J. Eur. Math. Soc. (JEMS)}, 6(3):367--398, 2004.

\bibitem{MR0442564}
Harry Dym and Henry McKean.
\newblock {\em Fourier series and integrals}.
\newblock Academic Press, New York-London, 1972.
\newblock Probability and Mathematical Statistics, No. 14.

\bibitem{MR0149402}
Robert Hermann.
\newblock On the accessibility problem in control theory.
\newblock In {\em Internat. {S}ympos. {N}onlinear {D}ifferential {E}quations
  and {N}onlinear {M}echanics}, pages 325--332. Academic Press, New York, 1963.

\bibitem{MR0638354}
Henry Hermes.
\newblock Controlled stability.
\newblock {\em Ann. Mat. Pura Appl. (4)}, 114:103--119, 1977.

\bibitem{MR0493664}
Henry Hermes.
\newblock Lie algebras of vector fields and local approximation of attainable
  sets.
\newblock {\em SIAM J. Control Optim.}, 16(5):715--727, 1978.

\bibitem{MR1096761}
Henry Hermes.
\newblock Homogeneous coordinates and continuous asymptotically stabilizing
  feedback controls.
\newblock In {\em Differential equations ({C}olorado {S}prings, {CO}, 1989)},
  volume 127 of {\em Lecture Notes in Pure and Appl. Math.}, pages 249--260.
  Dekker, New York, 1991.

\bibitem{MR1395834}
Henry Hermes.
\newblock Large-time local controllability via homogeneous approximations.
\newblock {\em SIAM J. Control Optim.}, 34(4):1291--1299, 1996.

\bibitem{MR1425878}
Velimir Jurdjevic.
\newblock {\em Geometric control theory}, volume~52 of {\em Cambridge Studies
  in Advanced Mathematics}.
\newblock Cambridge University Press, Cambridge, 1997.

\bibitem{MR0155070}
Rudolf Kalman, Yu-Chi Ho, and Kumpati Narendra.
\newblock Controllability of linear dynamical systems.
\newblock {\em Contributions to Differential Equations}, 1:189--213, 1963.

\bibitem{MR1061394}
Matthias Kawski.
\newblock High-order small-time local controllability.
\newblock In {\em Nonlinear controllability and optimal control}, volume 133 of
  {\em Monogr. Textbooks Pure Appl. Math.}, pages 431--467. Dekker, New York,
  1990.

\bibitem{MR0433288}
Arthur Krener.
\newblock The high order maximal principle and its application to singular
  extremals.
\newblock {\em SIAM J. Control Optimization}, 15(2):256--293, 1977.

\bibitem{MR0145169}
Joseph LaSalle.
\newblock The time optimal control problem.
\newblock In {\em Contributions to the theory of nonlinear oscillations, {V}ol.
  {V}}, pages 1--24. Princeton Univ. Press, Princeton, N.J., 1960.

\bibitem{MR889459}
Bruce Lee and Lawrence Markus.
\newblock {\em Foundations of optimal control theory}.
\newblock Robert E. Krieger Publishing Co., Inc., Melbourne, FL, second
  edition, 1986.

\bibitem{2015arXiv151104995M}
Fr\'ed\'eric {Marbach}.
\newblock {An obstruction to small time local null controllability for a
  viscous Burgers' equation}.
\newblock {\em ArXiv e-prints}, November 2015.

\bibitem{MR0186487}
Lawrence Markus.
\newblock Controllability of nonlinear processes.
\newblock {\em J. Soc. Indust. Appl. Math. Ser. A Control}, 3:78--90, 1965.

\bibitem{MR0199865}
Tadashi Nagano.
\newblock Linear differential systems with singularities and an application to
  transitive {L}ie algebras.
\newblock {\em J. Math. Soc. Japan}, 18:398--404, 1966.

\bibitem{MR0109940}
Louis Nirenberg.
\newblock On elliptic partial differential equations.
\newblock {\em Ann. Scuola Norm. Sup. Pisa (3)}, 13:115--162, 1959.

\bibitem{MR0120435}
Lev Pontryagin.
\newblock Optimal regulation processes.
\newblock {\em Uspehi Mat. Nauk}, 14(1 (85)):3--20, 1959.

\bibitem{MR1195304}
Lionel Rosier.
\newblock Homogeneous {L}yapunov function for homogeneous continuous vector
  fields.
\newblock {\em Systems Control Lett.}, 19(6):467--473, 1992.

\bibitem{MR1776551}
Yuri Sachkov.
\newblock Controllability of invariant systems on {L}ie groups and homogeneous
  spaces.
\newblock {\em J. Math. Sci. (New York)}, 100(4):2355--2427, 2000.
\newblock Dynamical systems, 8.

\bibitem{MR1640001}
Eduardo Sontag.
\newblock {\em Mathematical control theory}, volume~6 of {\em Texts in Applied
  Mathematics}.
\newblock Springer-Verlag, New York, second edition, 1998.
\newblock Deterministic finite-dimensional systems.

\bibitem{MR935375}
Gianna Stefani.
\newblock On the local controllability of a scalar-input control system.
\newblock In {\em Theory and applications of nonlinear control systems
  ({S}tockholm, 1985)}, pages 167--179. North-Holland, Amsterdam, 1986.

\bibitem{MR0321133}
H\'ector Sussmann.
\newblock Orbits of families of vector fields and integrability of
  distributions.
\newblock {\em Trans. Amer. Math. Soc.}, 180:171--188, 1973.

\bibitem{MR710995}
H\'ector Sussmann.
\newblock Lie brackets and local controllability: a sufficient condition for
  scalar-input systems.
\newblock {\em SIAM J. Control Optim.}, 21(5):686--713, 1983.

\bibitem{MR872457}
H\'ector Sussmann.
\newblock A general theorem on local controllability.
\newblock {\em SIAM J. Control Optim.}, 25(1):158--194, 1987.

\bibitem{MR0338882}
H\'ector Sussmann and Velimir Jurdjevic.
\newblock Controllability of nonlinear systems.
\newblock {\em J. Differential Equations}, 12:95--116, 1972.

\bibitem{MR2224013}
Emmanuel Tr\'elat.
\newblock {\em Contr\^ole optimal}.
\newblock Math\'ematiques Concr\`etes. [Concrete Mathematics]. Vuibert, Paris,
  2005.
\newblock Th\'eorie \& applications. [Theory and applications].

\bibitem{MR1967689}
Claire Voisin.
\newblock {\em Hodge theory and complex algebraic geometry. {I}}, volume~76 of
  {\em Cambridge Studies in Advanced Mathematics}.
\newblock Cambridge University Press, Cambridge, 2002.
\newblock Translated from the French original by Leila Schneps.

\end{thebibliography}

\end{document}